\documentclass[a4paper,11 pt]{amsart}

\usepackage{geometry} 
\usepackage{comment}

\usepackage{amssymb}  
\usepackage{mathtools}      
\usepackage{mathabx}        
\usepackage[bb=fourier,cal=euler,scr=rsfs]{mathalfa}	
\usepackage{enumitem}       
\usepackage[colorlinks=true,backref=page]{hyperref} 

\usepackage[ansinew]{inputenc}

\hypersetup{bookmarksdepth = 3} 
\numberwithin{equation}{section}     

\setlist[enumerate,1]{label={\upshape(\roman*)},ref=\roman*}
\setlist[enumerate,2]{label={\upshape(\alph*)},ref=\alph*}

\newcommand{\R}{\mbox{$\mathbb{R}$}}

\newcommand{\wt}[1]{\widetilde{#1}}

 \def\RR{{\mathbb R}}  \def\TT{{\mathbb T}}
   
 \def\ZZ{{\mathbb Z}}

\def\cA{\mathcal{A}}  \def\cG{\mathcal{G}}  
\def\cB{\mathcal{B}}   \def\cN{\mathcal{N}} \def\cT{\mathcal{T}}
\def\cC{\mathcal{C}}   \def\cO{\mathcal{O}} \def\cU{\mathcal{U}}
   \def\cP{\mathcal{P}} \def\cV{\mathcal{V}}
   \def\cQ{\mathcal{Q}} \def\cW{\mathcal{W}}
\def\cF{\mathcal{F}}  \def\cL{\mathcal{L}}

\newtheorem*{teo*}{Informal Statement}

\newtheorem{teo}{Theorem}[section]

\newtheorem{quest}{Question}
\newtheorem{cor}[teo]{Corollary}

\newtheorem{claim}[teo]{Claim}
\newtheorem{lema}[teo]{Lemma}
\newtheorem{lemma}[teo]{Lemma}
\newtheorem{prop}[teo]{Proposition}
\newtheorem{proposition}[teo]{Proposition}

\newtheorem{thmintro}{Theorem}

\theoremstyle{definition}
\newtheorem{defi}[teo]{Definition}
\theoremstyle{remark}
\newtheorem{obs}[teo]{Remark}
\newtheorem{remark}[teo]{Remark}
\newtheorem{notation}[teo]{Notation}

\newcommand{\eps}{\varepsilon}

\newcommand{\en}{\subset}

\newcommand{\mt}{\widetilde{M}}

\newcommand{\ft}{\widetilde{f}}
\newcommand{\Ft}{\widetilde{\cF}}
\newcommand{\tild}[1]{\widetilde{#1}}

\newcommand{\cs}{\cW^{cs}}
\newcommand{\cu}{\cW^{cu}}

\newcommand{\fol}{\mathcal{F}}
\newcommand{\fn}{\widetilde{\mathcal{F}}}

\newcommand{\oo}{\mathcal{O}}
\newcommand{\lcs}{\mathcal{L}^{cs}}
\newcommand{\lcu}{\mathcal{L}^{cu}}
\newcommand{\lecs}{\mathcal{L}^{cs}_{\eps}}

\newcommand{\lc}{\mathcal{L}^c}

\newcommand{\fes}{\cF^{cs}_{\eps}}
\newcommand{\feu}{\cF^{cu}_{\eps}}

\newcommand{\fec}{\cF^c_\eps}
\newcommand{\wfes}{\widetilde{\mathcal{F}^{cs}_{\varepsilon}}}
\newcommand{\wfeu}{\widetilde{\mathcal{F}^{cu}_{\varepsilon}}}
\newcommand{\wfec}{\widetilde\fec}
\newcommand{\hs}{h^{cs}}
\newcommand{\hu}{h^{cu}}

\newcommand{\wcs}{\widetilde{\cW^{cs}}}
\newcommand{\wcu}{\widetilde{\cW^{cu}}}

\newcommand{\wF}{\widetilde{\cF}}
\newcommand{\wG}{\widetilde{\cG}}

\title[Collapsed Anosov flows]{Collapsed Anosov flows and self orbit equivalences}

\author[T.~Barthelm\'e]{Thomas Barthelm\'e}
\address{Queen's University, Kingston, ON}
\email{thomas.barthelme@queensu.ca}
\urladdr{sites.google.com/site/thomasbarthelme}
\author[S.~Fenley]{Sergio R.\ Fenley} 
\address{Florida State University, Tallahassee, FL 32306}
\email{fenley@math.fsu.edu}
\author[R.~Potrie]{Rafael Potrie} 
\address{Centro de Matem\'atica, Universidad de la Rep\'ublica, Uruguay}
\email{rpotrie@cmat.edu.uy}
\urladdr{http://www.cmat.edu.uy/~rpotrie/}
\thanks{T.B. was partially supported by the NSERC (Funding reference number RGPIN-2017-04592). S.F. was partially supported by Simons Foundation grant numbers 280429 and 637554. R. P. was partially supported by CSIC 618, FCE-1-2017-1-135352. This work started and mostly completed while R.P. was a Von Neumann fellow at IAS, funded by Minerva Research Fundation Membership Fund and NSF DMS-1638352 and he wishes to thank the IAS for the excellent working conditions. }

\keywords{Anosov flows, Partially hyperbolic diffeomorphisms, 3-manifolds, Classification.}
\subjclass[2010]{}
\begin{document}

\begin{abstract}
We propose a generalization of the concept of discretized Anosov flows
that covers a wide class of
partially hyperbolic diffeomorphisms in 3-manifolds, and that 
we call 
collapsed Anosov flows. They  are related to Anosov flows
via a self orbit equivalence of the
flow. 
We show that all the examples from \cite{BGHP} belong to this class, and that 
it is an open and closed class among partially hyperbolic diffeomorphisms. 
We provide some equivalent definitions which may be
more amenable to analysis and are useful in different situations.
Conversely, we describe the isotopy classes of partially hyperbolic diffeomorphisms
that are collapsed Anosov flows
associated with certain types of Anosov flows.
\end{abstract}

\maketitle

\section{Introduction}\label{s.intro}

For about $15$ years Pujals' conjecture \cite{BW} has served as a blueprint and motivation for the understanding and classification of partially hyperbolic diffeomorphisms in dimension $3$. In most $3$-manifolds, that is, those with non virtually solvable fundamental group\footnote{See, e.g., \cite{HPsurvey} for the case of manifolds with virtually solvable fundamental group.}, the conjecture affirmed that, up to iterates and finite lifts, a transitive partially hyperbolic diffeomorphism had to behave like the discretization of an Anosov flow: the diffeomorphism should globally fix each orbit of an associated Anosov flow, moving points along the orbits.

In the past few years, Pujals' conjecture was disproved: Examples built in \cite{BGP,BGHP} (see also \cite{BPP,BZ}) gave a plethora of new partially hyperbolic diffeomorphisms. All of these new examples are such that they have infinite order in the mapping class group of their underlying manifolds, contradicting Pujals' conjecture.

Thanks to a criterion developed in \cite{BGHP}, called $\varphi$-transversality (see Definition \ref{def.phi-transversality}), these new examples --- as well as the older examples of \cite{BW} ---
can be described in the following way: Start with an Anosov flow $\phi_t$ on a manifold $M$. Then find a diffeomorphism $\varphi$ of $M$ that preserves the transversality of the bundles of the Anosov splitting (more precisely, such that the flow $\phi_t$ is $\varphi$-transversal to itself, see 
Definition \ref{def.phi-transversality}). Finally, compose $\varphi$ with a very large time of the flow $\phi_t$ and obtain a partially hyperbolic diffeomorphism.

Finding good diffeomorphisms $\varphi$ is generally the difficult step, but one type of map that does work (as chosen in \cite{BW}) is a smooth symmetry of the Anosov flow.

In this article we show that, while not obvious from the constructions, all of the new examples of partially hyperbolic diffeomorphisms are related to symmetries (self orbit equivalences to be precise) of the initial flow.

More generally, the main goal of our article is to introduce, a new class of partially hyperbolic diffeomorphisms in dimension $3$, that we call \emph{collapsed Anosov flows}. A partially hyperbolic diffeomorphism is a collapsed Anosov flow if there exists a global collapsing map, homotopic to the identity, that semi-conjugates a self orbit equivalence of a topological Anosov flow with the diffeomorphism.

This class of diffeomorphisms has very interesting properties. 
In particular we show 
the following (formalized below as Theorems \ref{teo.main1} and \ref{teo.main6}).
\begin{teo*}
Collapsed Anosov flows form an open and closed class of partially hyperbolic diffeomorphisms in dimension 3 that contains all known examples in manifolds with non-virtually solvable fundamental group. 
\end{teo*}

Since our goal is in part to lay down the basis for a future study of this class, we introduce four definitions: Three of them (\emph{collapsed Anosov flow}, \emph{strong collapsed Anosov flow} and \emph{leaf space collapsed Anosov flow}, Definitions \ref{defi1w}, \ref{defi1} and \ref{defiw2} respectively) have to do with how restrictive one wants the semi-conjugacy to be in terms of its behavior with respect to either center
curves or the branching foliations of the partially hyperbolic diffeomorphisms. The last definition (\emph{quasigeodesic partially hyperbolic diffeomorphisms}, Definition \ref{defiQG}) is different as it instead asks for the center foliation to be by quasigeodesics inside each center stable and center unstable leaf\footnote{It also requires the center stable/unstable branching foliations to be by Gromov-hyperbolic leaves.}. 

Under some orientability conditions, we prove
equivalence  between quasigeodesic partially hyperbolic diffeomorphisms, strong collapsed Anosov flows and leaf space collapsed Anosov flows (Theorems \ref{teo.main3} and \ref{teo.main5}).
We believe that these equivalences will show themselves to be quite useful: For instance, the proof, obtained in \cite{FP-2}, that every hyperbolic $3$-manifold that admits a partially hyperbolic also admits an Anosov flow relies on these equivalences.

In light of the fact that the known counter-examples to Pujals' conjecture are all collapsed Anosov flows, it is natural to ask the following (thus extending \cite[Question 1]{BGP} and making \cite[Question 12]{PotICM} precise).

\begin{quest}\label{q.PHisCAF}
Let $M$ be a $3$-manifold with non virtually solvable fundamental group and $f\colon M \to M$ a (transitive) partially hyperbolic diffeomorphism. Is $f$ a collapsed Anosov flow? 
\end{quest}

One interest of Pujals' conjecture was to suggest that the classification of partially hyperbolic diffeomorphisms in dimension $3$ could be done up to classification of Anosov flows. If the question above admits a positive answer, then that view behind Pujals' conjecture may still be true, as it seems possible to understand all the self orbit equivalences of Anosov flows without having a full classification of the flows\footnote{Prior to the present article, self orbit equivalences of Anosov flows were only understood in very specific cases: for geodesic flows as can, for instance, be deduced from \cite{Matsumoto}, or for self orbit equivalences that are homotopic to identity \cite{BaG}. In section \ref{s.classification_results}, we describe self orbit equivalences of the Franks--Williams example. After the completion of this article, the first author and K.~Mann also obtained a general result that can be used to classify self orbit equivalences for any $\RR$-covered Anosov flow \cite{BarMann}.}.

While we will not suggest an answer to Question \ref{q.PHisCAF} in full generality, there are several contexts where we can say more:
\begin{enumerate}[label =(\roman*)]
 \item When $M$ is hyperbolic, the answer is proven to be positive in \cite{FP-2}.
 \item When the partially hyperbolic diffeomorphism is homotopic to the identity, the answer is likely positive (see \cite{BFFP_part1,BFFP_part2}).
 \item For Seifert manifolds, current work in progress by the second and third authors also indicates a positive answer.
\end{enumerate}

Another potential interest we see in collapsed Anosov flows is that it may allow one to successfully decouple the dynamical study of partially hyperbolic diffeomorphisms from the question of their classification. Indeed, one may be able to obtain fine dynamical properties when restricting to particular types of collapsed Anosov flows.

This strategy has previously been successfully used in \cite{AVW,BFT,BarthelmeGogolev_centralizers,DWX,FP-acc,GM} for \emph{discretized Anosov flows}, or similar concepts. 
Discretized Anosov flows were 
introduced in \cite{BFFP_part1} (although related notions appeared previously, for instance in \cite{BW,BG}). One can view them
as collapsed Anosov flows where the self orbit equivalence is trivial, meaning that it fixes every orbit of the flow (see \S\ref{s.AF}). By \cite{BFFP_part1,BFFP_part2}, discretized Anosov flows
represent a very large class of partially hyperbolic diffeomorphisms. 
An example of a dynamical consequence is \cite{FP-acc}, where it is shown
that discretized Anosov flows are always \emph{accessible} unless they come from suspensions (in particular, smooth volume preserving ones are ergodic).
In fact, another, albeit slightly weaker, accessibility result is obtained for some specific collapsed Anosov flows in \cite{FP-acc} (but without using our terminology), and it seems plausible that such results could be achievable for other classes of collapsed Anosov flows\footnote{After the completion of this work, the second and third authors proved accessibility of collapsed Anosov flow under the assumption that the non-wandering set is everything \cite{FP-3}.}.

Finally, we can use the interaction of collapsed Anosov flows with self orbit equivalences of Anosov flows to classify them up to isotopy. 
Over the years, a deep knowledge of the orbit space of Anosov flows in dimension $3$ has been attained. This in turn gives restrictions on how a self orbit equivalence can act. For instance, self orbit equivalences that are homotopic to the identity were classified in \cite{BaG}, showing that
there are at most two types of actions in that case. Knowing restrictions about self orbit equivalences (for instance which isotopy classes can support them) directly implies restrictions on possible collapsed Anosov flows.

On the other hand, a general method to build self orbit equivalences of Anosov flow has not yet been developed. The construction methods of \cite{BGHP} together with Theorem \ref{teo.main1} gives one such method.

We illustrate what consequences this interaction gives us in a few specific cases. In particular, we give a complete description of collapsed Anosov flows up to isotopy in the following cases: 1) when the manifold is the unit tangent bundle of a surface; \ 2) when the associated flow is the Franks--Williams example \cite{FranksWilliams}; or \ 3) when the collapsed Anosov flow is homotopic to the identity (see \S\ref{s.classification_results}).

The conceptualization of the notion of collapsed Anosov flows that we introduce here has been in part motivated by \cite{FP-2} (and to a lesser extent by \cite{BFFP_part1,BFFP_part2}). The indebtedness we have to these previous works does not translate, however, into their direct use in the present article. Indeed, the scope, as well as most of the techniques we use here, are different in nature from those in the aforementioned works.

\section{Collapsed Anosov flows}\label{section.CAF}

In this paper, $M$ will always denote a closed 3-dimensional manifold.
It is possible that some notions make sense in higher dimensions but we will repeatedly use facts about foliations and Anosov flows in dimension 3 that are unknown in higher dimensions and we have not checked to which extent arguments extend (even if only in part) to higher dimensions.

In this section, we will make precise the different definitions of collapsed Anosov flows alluded to earlier and formally state the main results of this article.

First we review the notion of topological Anosov flows, in order to be able to introduce collapsed Anosov flows.

\subsection{Anosov flows and topological Anosov flows}

Recall that an \emph{Anosov flow} is a flow $\phi_t \colon M \to M$ generated by a vector field $X$ such that $D\phi_t$ preserves a splitting $TM = E^s \oplus \R X \oplus E^u$ and there exists $a,b$ such that, for all $t>0$,
\begin{align*}
 \|D\phi_t v^s \| &< a e^{bt}\|v^s\| \text{ for all } v^s \in E^s \\
 \|D\phi_{-t} v^u \| &< a e^{bt}\|v^u\| \text{ for all } v^u \in E^u.
\end{align*}

It is well known that this property implies that the splitting is continuous. Moreover, it follows that the bundles $E^s$ and $E^u$ are uniquely integrable into $\phi_t$-invariant foliations $\cF^s$ and $\cF^u$ tangent respectively to $E^s$ and $E^u$ \cite{An} called the \emph{strong stable} and \emph{strong unstable} foliations of $\phi_t$. One obtains $\phi_t$-invariant foliations $\cF^{ws}$ and $\cF^{wu}$ called the \emph{weak stable} and \emph{weak unstable} foliations by taking the saturation of the previous foliations by the flow. Note that these foliations are the unique foliations tangent respectively to $E^{ws}:=E^s \oplus \R X$ and $E^{wu}:=\R X \oplus E^u$ \cite{An}.  
See \cite{HPS} or \cite[\S 4]{CP} and references therein for more details. 

The following generalizes Anosov flows:

 \begin{defi}\label{def_TAF}
  A \emph{topological Anosov flow} is a continuous flow $\phi_t \colon M \to M$ that satisfies the following:
\begin{enumerate}[label = (\roman*)]
 \item\label{item_TAF_tangentVF} The flow $\phi_t$ is generated by a continuous, non-singular, vector field $X$. In particular, the orbits of $\phi^t$ are $C^1$-curves in $M$;
\item \label{item_TAF_weak_foliations} The flow $\phi_t$ preserves two continuous, topologically transverse, $2$-   dimensional foliations $\cF^{ws}$ and $\cF^{wu}$.
\item \label{item_TAF_forward_asymptotic} Given any $x\in M$ and $y \in \cF^{ws}$ (resp.~$y \in \cF^{wu}$), there exists a continuous increasing reparametrization $h\colon \R \to \R$ such that $d\left(\phi_t(x), \phi_{h(t)}(y)\right)$ converges to $0$ as $t\to +\infty$ (resp.~$t\to -\infty$).
\item \label{item_TAF_backwards_expansivity} There exists $\eps>0$ such that for any $x\in M$ and any $y\in \cF^{ws}_{\eps}(x)$ (resp.~$y\in \cF^{wu}_{\eps}(x)$), with $y$ not on the same orbit as $x$, then for any continuous increasing reparametrization $h\colon \R \to \R$, there exists a $t\leq 0$ (resp.~$t\geq 0$) such that $d\left(\phi_t(x), \phi_{h(t)}(y)\right) >\eps$. 
\end{enumerate}
 \end{defi}

\begin{remark}
 Historically, topological Anosov flows were first considered as the class of pseudo-Anosov flows, as introduced by Mosher \cite{MosherI,MosherII}, that did not have any singular orbits.
 Since then, many different versions of the definition have been used in the literature.
 As we will see in \S\ref{s.AF}, thanks to works of Inaba and Matsumoto \cite{IM} and Paternain \cite{Pat}, one can now make a very succinct definition of topological Anosov flow as an expansive flow, tangent to a non-singular continuous vector field, and preserving a foliation (see Theorem \ref{teo.expimpliesTAF}). Our definition is equivalent to this, but we express it in a way that is convenient for the statements of our results, in particular, as we will later explain, it is equivalent to the definition used in \cite{Shannon}. 
\end{remark}

An \emph{orbit equivalence} between (topological) Anosov flows $\phi_t^1\colon M \to M$ and $\phi_t^2\colon N \to N$ is a homeomorphism $\beta\colon M \to N$ sending orbits of $\phi_t^1$ to orbits of $\phi_t^2$ preserving the orientation.
 In other words, there exists a reparametrization $\phi^2_{u(x,t)}$\footnote{In order for $\phi^2_{u(x,t)}$ to be a flow, the function $u$ must satisfy the following cocycle condition:  $u(x,t+s)= u(\phi^1_t(x),s)+ u(x,t)$. See, e.g., \cite[\S 2.2]{KH}.} of $\phi^2_t$ such that $\beta$ is a conjugation, i.e., $\beta(\phi_t^1(x)) = \phi^2_{u(t,x)}(\beta(x))$, where $u(t,x)$ is monotone increasing for fixed $x$.

\begin{remark} It has been recently proved that every \emph{transitive} topological Anosov flow is \emph{orbit equivalent} to a (smooth) Anosov flow \cite{Shannon}.

There are several reasons why we choose to consider topological Anosov flows, despite Shannon's result. First, from a philosophical stand point, this article aims to relate the topological classification problem for partially hyperbolic diffeomorphisms to that of the topological classification of Anosov flows, making topological Anosov flow the more natural setting. But, more importantly, in some of our results (specifically in Theorem \ref{teo.main5}) we only obtain a \emph{topological} Anosov flow. For the generic partially hyperbolic diffeomorphisms that we consider, the associated topological Anosov flow has no reason to be transitive, thus it is yet unknown whether the topological Anosov flows that we obtain can be taken to be orbit equivalent to a smooth Anosov flow. We will however show that under some general assumptions on the partially hyperbolic diffeomorphism $f$ (for instance $f$-minimality of one of the foliations, see \S\ref{ss.aplicationGT}), the corresponding Anosov flow is transitive, hence can be taken to be a smooth Anosov flow up to orbit equivalence.
\end{remark} 

\begin{defi} A \emph{self orbit equivalence} of an Anosov flow $\phi_t$ is an orbit equivalence between $\phi_t$ and itself. 
\end{defi}

Self orbit equivalences homotopic to the identity have been studied in \cite{BaG} to understand fiberwise Anosov dynamics, but in fact
there are self orbit equivalences of certain Anosov flows which are not homotopic to the identity.

\begin{defi}\label{def.trivialsoe_and_equivalentsoe}
We say that a self orbit equivalence $\beta$ is \emph{trivial} if there exists a  continuous function $\tau\colon M \to \RR$ such that $\beta(x)= \phi_{\tau(x)}(x)$.
Two self orbit equivalences $\alpha,\beta$ are said to be \emph{equivalent} (or that they belong to the same class) if $\alpha \circ \beta^{-1}$ is a trivial self orbit equivalence.  
\end{defi}

\subsection{Partially hyperbolic diffeomorphisms}

A \emph{partially hyperbolic diffeomorphism} is a diffeomorphism $f\colon M \to M$ such that $Df$ preserves a splitting $TM = E^s \oplus E^c \oplus E^u$ into non-trivial bundles such that there is $\ell>0$ verifying that for every $x\in M$ and unit vectors $v^\sigma \in E^\sigma(x)$ ($\sigma = s,c,u$) one has:

\begin{equation*}
 \|Df^\ell v^s \| < \frac{1}{2} \min\{ 1, \|Df^\ell v^c \| \},  \text{ and } \|Df^\ell v^u\| > 2 \max \{1, \|Df^\ell v^c \| \}.
\end{equation*}

As in the Anosov flow case, this condition implies that the bundles are continuous. It also implies unique integrability of the bundles $E^s$ and $E^u$ into foliations $\cW^s$ and $\cW^u$, called \emph{strong stable} and \emph{strong unstable} foliations respectively (see \cite{CP}). 
We denote by $E^{cs}=E^s \oplus E^c$ and $E^{cu}=E^c \oplus E^u$. 

\begin{remark} It follows that given an Anosov flow, its
 time one map is partially hyperbolic and $E^c$ coincides with the bundle generated by the vector field tangent to the flow. 
\end{remark}

When necessary, we will denote the dependence of bundles or foliations on the maps with a subscript, e.g., $E^s_f$ or $\cF^{s}_\phi$. 

\subsection{Collapsed Anosov flows and strong collapsed Anosov flows} 

We are now ready to give the formal definition of a collapsed Anosov flow.

\begin{defi}[Collapsed Anosov Flow]\label{defi1w} 
A partially hyperbolic diffeomorphism $f$ of a closed 3-dimensional manifold $M$ is said to be a \emph{collapsed Anosov flow} if there exists a topological Anosov flow $\phi_t$, a continuous map $h\colon M \to M$ homotopic to the identity and a self orbit equivalence $\beta\colon M \to M$ of $\phi_t$ such that: 
\begin{enumerate} [label= (\roman*)]
\item \label{item.Center_direction_preserved} The map $h$ is differentiable along the orbits of $\phi_t$ and maps the vector field tangent to $\phi_t$ to non-zero vectors tangent to $E^c$.
\item For every $x \in M$ one has that $f \circ h(x)= h\circ \beta(x)$. 
\end{enumerate}
\end{defi}

As noted earlier, \emph{discretized Anosov flows} (as defined in \cite{BFFP_part1}, see \S\ref{ss.DAF}) are collapsed Anosov flows, where $h$ can be taken to be the identity and $\beta$ is a trivial self orbit equivalence.

\begin{remark} The name of this class was chosen as a natural extension of the class of discretized Anosov flows. We stress that these collapsed Anosov flows are diffeomorphisms, and not flows, and apologize in advance if it leads to any confusion. Other possible names that we considered, but decided against, were ``of collapsed Anosov flow type'' or ``collapsed self orbit equivalences''.
\end{remark}

Another case, discussed previously, that is easily seen to be a collapsed Anosov flow is when an Anosov flow $\phi_t$ commutes with a smooth map $\beta$ (as in \cite[Proposition 4.5]{BW} for instance), then $\beta \circ \phi_1$ is a collapsed Anosov flow (with self orbit equivalence 
$\beta \circ \phi_1$ and $h$ the identity).
However, we show (Proposition \ref{prop.C1soe}) that these examples are always ``periodic'' in the sense that a power of the diffeomorphism is a discretized Anosov flow (or, equivalently, a power of the self orbit equivalence is trivial).

In contrast, the examples of \cite{BGHP} will give, thanks to Theorem \ref{teo.main1}, collapsed Anosov flows associated with self orbit equivalences of infinite order. Indeed, from a topological point of view, the examples built in \cite{BGHP} are of two forms: Either the manifold $M$ is toroidal with a torus $T$ transverse to an Anosov flow and the examples are in the isotopy class of a Dehn twist on $T$ (or a composition of such). Or the manifold is the unit tangent bundle $T^1S$ of a surface and the examples are in the isotopy class of the differential of any diffeomorphism of the base $S$.

\begin{remark}\label{rem.corientable}
The definition of a collapsed Anosov flow
forces the center direction of $f$ 
to be orientable, since we can induce an orientation from the orientation of the flow direction via $h$. 
To see this, suppose that $E^c$ is not orientable, and suppose that
$\alpha$ is a closed curve starting at $x$ in $M$ that 
reverses the local orientation of $E^c$.
Let $\gamma$ be the deck transformation associated to $\alpha$.
Lift $x$ to $\tild{x}$ in $\mt$. 
Let $\tild{h}$ be a lift of $h$ which is a lift
of a homotopy of $h$ to the identity.
Then $\tild{h}$ commutes with every deck transformation.
Since $h$ is homotopic to the identity it is degree one,
so there is $\tild{y}$ in $\mt$ such that $\tild{h}(\tild{y})
= \tild{x}$. 
Let $\eta$ be a curve in $\mt$ from $\tild{y}$ to 
$\gamma(\tild{y})$. Along $\tild{h}(\eta)$ the projection
of the flow lines of $\widetilde{\phi_t}$ by 
$\tild{h}$ induces a non zero vector tangent to $E^c$.
Since $\gamma$ commutes with $\tild{h}$ the final vector
is the tangent to $E^c$ in the direction induced
by $\gamma$. But $\gamma$ was supposed to reverse the direction
of $E^c$ so this is a contradiction to the fact that
the tangent vectors to $E^c$ are changing continuously
along the curve $\tild{h}(\eta)$.
\end{remark}

Note that we require more from a collapsed Anosov flow than just being semi-conjugated to a self orbit equivalence of an Anosov flow. Indeed condition \ref{item.Center_direction_preserved} of Definition \ref{defi1w} asks that the semi-conjugacy $h$ at least sends the flow direction to the center direction. 

There are several reasons for this condition:
 First, going back to at least \cite{HPS}, an overarching idea has been that any kind of classification for partially hyperbolic diffeomorphisms should be ``up to center dynamics'' (where the precise meaning of this can be taken to be more or less strong, and somehow has to be adapted to the particular situation of study). Therefore we see condition \ref{item.Center_direction_preserved} as the minimal requirement in order to keep to the spirit of this paradigm.
A second, less philosophical, reason is that a collapsed Anosov flow thus defined is a natural extension of the concept of discretized Anosov flow introduced in \cite{BFFP_part1} (see \S\ref{ss.DAF}, in particular Proposition \ref{p.daf2}).
Finally, Definition \ref{defi1w} provides us with a model of the dynamical behavior of the partially hyperbolic diffeomorphism to compare it with. Moreover, since some features of the dynamics of a self orbit equivalence $\beta$ can be readily understood, we hope that Definition \ref{defi1w} is enough to understand some of the dynamical properties of a collapsed Anosov flow, as has been the case for discretized Anosov flows.

There is an issue one quickly runs into when one wants to extract more geometrical or topological information about the partially hyperbolic diffeomorphism from the collapsed Anosov flow definition: The map $h$ sends orbits of the Anosov flow to center curves (i.e., curves tangent to the center direction) of the diffeomorphism $f$. However, it is usually difficult in partially hyperbolic dynamics to extract much knowledge about the behavior of center stable and center unstable (branching) foliations from coarse information about center curves. In fact, the center curves obtained via $h$ may, a priori, not even be inside the intersection of a center stable and center unstable leaf\footnote{It is not even known if a collapsed Anosov flow necessarily admits invariant center stable or center unstable branching foliations, as the existence result of Burago--Ivanov \cite{BI}, see \S\ref{s.branching}, requires some orientability conditions.}.

One can resolve this issue, while preserving the interest of the class of partially hyperbolic diffeomorphisms thus defined, by requiring that the semi-conjugacy $h$ somehow sends the weak stable and unstable directions of the flow to the center stable and unstable directions of the diffeomorphism. This leads us to our next definition.

\begin{defi}[Strong Collapsed Anosov Flow]\label{defi1}
A partially hyperbolic diffeomorphism $f$ of a closed $3$-manifold $M$ is called a \emph{strong collapsed Anosov flow} if there exists a topological Anosov flow $\phi_t$, a continuous map $h\colon M \to M$ homotopic to identity and a self orbit equivalence $\beta\colon M \to M$ of $\phi_t$ such that: 
\begin{enumerate}
\item The map $h$ is differentiable along orbits of $\phi_t$ and maps the vector field tangent to $\phi_t$ to non vanishing vectors tangent to $E^c$.
\item  The image by $h$ of a leaf of $\cF^{ws}_\phi$ (resp. $\cF^{wu}_\phi$) is a $C^1$-surface tangent to $E^{cs}$ (resp. $E^{cu}$).  
\item The map $h$ is \emph{transversaly collapsing}: Given a lift $\tilde h$ of $h$ to the universal cover $\mt$, then for any leaf $\wt F$ of $\wt\cF^{ws}_{\phi}$ (or $\wt\cF^{wu}_{\phi}$) and any orbit $\gamma$ of $\wt \phi_t$ on $F$, the map $\tilde h$ sends $\gamma$ to a curve $c$ in $\tilde h (F)$ that separates $\tilde h (F)$ in two. Moreover, $\wt h$ sends the open half-spaces inside the closed half-spaces. More precisely, the image by $\wt h$ of one of the connected component of $F\smallsetminus \gamma$ is contained in the closure (in $\tilde h(F)$) of exactly one connected component of $\tilde h(F)\smallsetminus c$.
\item For every $x \in M$ one has that $f \circ h(x)= h\circ \beta(x)$. 
\end{enumerate}
\end{defi}

By being a $C^1$-surface tangent to $E^{cs}$ we mean that if $\tild{h}$ is a lift of $h$ to $\mt$ and
$L$ is a leaf of $\widetilde{\cF}^{ws}_\phi$ then $\tild{h}(L)$
is a $C^1$, properly embedded plane in $\mt$ tangent to $E^{cs}$.

Clearly, a strong collapsed Anosov flow is a collapsed Anosov flow, but we do not know whether those definitions are distinct or equivalent.

Notice that a strong collapsed Anosov flow automatically admits a pair of invariant center stable and center unstable branching foliations (see \S\ref{s.branching} for the precise definition) by looking at the image under $h$ of the weak foliations of the Anosov flow. Indeed, the ``transversaly collapsing'' condition in the definition ensures that the images of the leaves under $h$ may merge but do not topologically cross.
Definition \ref{defi1w} on the other hand does not directly require the existence of such branching foliations. But even if one assumes that a collapsed Anosov flow has branching foliations (or even true foliations), it is not clear that it is enough to make it a strong collapsed Anosov flow. Part of the issue arising here is that, in general, the center direction of a partially hyperbolic diffeomorphism is not uniquely integrable (even when it integrates to a foliation, see \cite{HHU}).

Let us mention here, that we obtain some results about unique integrability of the center direction in \S\ref{ss.uniqueness_center_curves}.

\subsection{First results and examples} 

In \cite{BGHP} a notion of transversality was introduced that allows to produce new examples of partially hyperbolic diffeomorphisms. This encompasses results proved in previous papers \cite{BPP, BGP, BZ}.   

\begin{defi}\label{def.phi-transversality}
Let $\phi_t \colon M \to M$ be an Anosov flow generated by a vector field $X$ in a closed $3$-manifold and preserving a splitting $TM = E^{s} \oplus \RR X \oplus E^u$ and  $\varphi\colon M \to M$ a diffeomorphism. We say that $\phi_t$ is $\varphi$-\emph{transverse to itself} if $D\varphi(E^{u})$ is transverse to $E^s \oplus \RR X$ and $D\varphi^{-1}(E^s)$ is transverse to $\RR X \oplus E^u$. 
\end{defi}
Note that this notion makes sense more generally when considering any partially hyperbolic diffeomorphism instead of an Anosov flow $\phi_t$, see \cite{BGHP}. 

Using this notion, \cite{BGHP} proves:

\begin{prop}[Proposition 2.4 \cite{BGHP}]\label{prop.BGHP_example}
If an Anosov flow $\phi_t$ is $\varphi$-transverse to itself, then there exists $T>0$ such that , for all $t>T$, the map  $f_t := \phi_t \circ \varphi \circ \phi_t$ is\footnote{Note that we wrote it this way for convenience, since $f_t$ is smoothly conjugate to $\phi_{2t} \circ \varphi$ and to $\varphi \circ \phi_{2t}$.} partially hyperbolic.
\end{prop}

Not only does \cite{BGHP} give that criterion for building partially hyperbolic diffeomorphisms, but it also gives many examples (using results of \cite{BPP, BGP, BZ}) of maps $\varphi$ and Anosov flows $\phi_t$ that are $\varphi$-transverse to themselves.

The proof of this criterion is an almost immediate application of the classical cone criterion: Given a, say, strong unstable cone $\mathcal{C^uu}$ for $\phi_t$, the transversality of $D\varphi(E^{u})$ with $E^s \oplus \RR X$ ensures that, for some large enough $t$, $D\phi_t \mathcal{C^uu}$ will also be a cone family for the map $f_t$.
A drawback of this proof is that one does not get any precise information about the dynamics.

Hence, while great at providing examples, this criterion fails, at least directly, to give any concrete understanding of the structure that these maps may enjoy. Indeed, it is a priori not obvious, and it may even 
seem contradictory, how these examples may act:
On the one hand, many of them are not homotopic to the identity, while when $t$ is large, the dynamics seems to be governed by the Anosov flow $\phi_t$.

Our first main result gives an understanding of how the behaviors of $\varphi$ and $\phi_t$ must play together and makes clear the structure of these examples.

\begin{thmintro}\label{teo.main1}
Let $\phi_t\colon M \to M$ be an Anosov flow on a closed 3-manifold and $\varphi\colon M \to M$ a diffeomorphisms such that $\phi_t$ is $\varphi$-transversal to itself. Then, there exists $t_0>0$ such that for all $t>t_0$ the diffeomorphism $f_t = \phi_t \circ \varphi \circ \phi_t$ is a strong collapsed Anosov flow of the flow $\phi_t$. 
\end{thmintro}

With the help of Theorem \ref{teo.main1} one can prove that
all the partially hyperbolic diffeomorphisms built in \cite{BGHP} are
collapsed Anosov flows.\footnote{To be precise, one proves
that all examples \textit{\`a la} \cite{BGHP}, understood as any example obtained via Proposition \ref{prop.BGHP_example}, are collapsed Anosov flows by applying Theorems \ref{teo.main1} and \ref{teo.main6} together. 
See Remark \ref{rem:thmA}.}
This not only gives
a wealth of examples, but also shows that all known constructions of partially hyperbolic diffeomorphisms on $3$-manifolds with non virtually solvable fundamental group are collapsed Anosov flows. Note that the non virtually solvable fundamental group assumption is necessary: Aside from the examples of partially hyperbolic diffeomorphisms on $\TT^3$, where no Anosov flow can exist, there are, as pointed out by a referee, some examples built similarly as in \cite{HHU} which are partially hyperbolic diffeomorphisms on the mapping torus of an Anosov automorphism, but which are not collapsed Anosov flows, see \S\ref{ss.open}.

In the examples that we advertised earlier, i.e., the discretized Anosov flows and the examples of \cite{BW}, the map $h$ of Definition \ref{defi1} could always be taken to be a homeomorphism (in fact, the identity). Now, some of the collapsed Anosov flows obtained through Theorem \ref{teo.main1} show why we cannot always ask for the collapsing map $h$ to be a homeomorphism: Indeed, if $h$ is injective, then the image by $h$ of the weak stable and weak unstable foliations of the Anosov flow $\phi_t$ are center stable and center unstable foliations of the strong collapsed Anosov flow $f$. In particular, $f$ must be \emph{dynamically coherent}\footnote{A partially hyperbolic diffeomorphism $f$ is called \emph{dynamically coherent} if it preserves a pair of foliations tangent to respectively $E^s \oplus E^c$ and $E^c \oplus E^u$.}. Since some examples built in \cite{BGHP} are shown to be non dynamically coherent, the associated map $h$ must be non-injective.

It is an interesting question to try to determine when the map $h$ can be a homeomorphism, or equivalently when a strong collapsed Anosov flow may be dynamically coherent. Some examples built in \cite{BPP,BZ} are dynamically coherent and associated with a non-periodic self orbit equivalence. But the associated Anosov flow is \emph{non transitive}.

So far, the only collapsed Anosov flows associated with a \emph{transitive} Anosov flow that are known to be dynamically coherent are such that the self orbit equivalent is periodic (i.e., such that a power is a trivial self orbit equivalence).

\subsection{Leaf space collapsed Anosov flows}

Although not explicit, the definition of a strong collapsed Anosov flow implies the existence of center stable $\cs$ and center unstable $\cu$ branching foliations (we defer their precise definitions to \S\ref{s.branching}) that are invariant under $f$.
By taking the intersection of these branching foliations (in an appropriate way), one gets an invariant $1$-dimensional \emph{center} branching foliation $\cW^c$. 

It is possible to generalize the definition of a leaf space of a true foliation to the branching case (see \S\ref{s.branching} or \cite{BFFP_part2}), and we thus obtain the center leaf space $\lc$, on which any lift $\wt f$ of $f$ to the universal cover acts naturally.

For a collapsed Anosov flow which preserves branching foliations, this center leaf space $\lc$ should be the same as the orbit space of an Anosov flow, and the action of $\wt f$ should correspond to the action of a lift of a self orbit equivalence. This idea is made precise in the next definition of a \emph{leaf space collapsed Anosov flow}.

For a topological Anosov flow $\phi_t\colon M \to M$ we denote by $\cO_\phi$ the orbit space of the flow $\wt\phi_t$ which is the lift of $\phi_t$ to $\mt$. We recall that $\cO_{\phi}$ is homeomorphic to $\RR^2$ and $\pi_1(M)$ acts naturally on $\cO_\phi$, see \cite{Barbot,FenleyAnosov}\footnote{Technically, the references \cite{Barbot,FenleyAnosov} only deal with smooth Anosov flows, however the proofs rely only on the existence of weak foliations and the behavior of orbits inside them, so apply directly to the topological Anosov setting. Formally, a proof for topological Anosov flow is contained in \cite{FenleyMosher}, where that result is proven for any pseudo-Anosov flow.}. 

\begin{defi}[Leaf space collapsed Anosov flow]\label{defiw2}
We say that a partially hyperbolic diffeomorphism $f$ of a closed 3-manifold is a \emph{leaf space collapsed Anosov flow} if it preserves center stable and center unstable branching foliations $\cs$ and $\cu$ and there exists a topological Anosov flow $\phi_t$ and a homeomorphism $H\colon \cO_\phi \to \lc$ which is $\pi_1(M)$-equivariant. 
\end{defi}

That $H$ is $\pi_1(M)$-equivariant means that if $\gamma \in \pi_1(M)$ is a deck transformation, then, $H(\gamma o) = \gamma H(o)$, for any $o \in \cO_\phi$. 

\begin{remark}
Notice that the branching foliations $\cs$ and $\cu$ are an integral part of the data needed to define a leaf space collapsed Anosov flow: The center leaf space $\lc$ is defined \emph{using} the branching foliations $\cs$ and $\cu$ (see \S\ref{s.branching}). In fact, we do not know the answer to the following questions in full generality: If $f$ preserves two pairs of branching foliations $\cW_1^{cs,cu}$ and $\cW_2^{cs,cu}$, is the center leaf space obtained from $\cW_1^{cs,cu}$ homeomorphic to that build from $\cW_2^{cs,cu}$? And if $f$ is a leaf space collapsed Anosov flow for  $\cW_1^{cs,cu}$, is it also a leaf space collapsed Anosov flow for $\cW_2^{cs,cu}$? Even assuming that $f$ is a leaf space collapsed Anosov flow for $\cW_1^{cs,cu}$ and for $\cW_2^{cs,cu}$, are the two associated center leaf spaces homeomorphic? Are the two associated Anosov flows orbit equivalent?

 We further emphasize that the map $H \colon \cO_\phi \to \lc$ in the definition above is a \emph{homeomorphism}, and not just a surjective continuous map as $h\colon M \to M$ was in Definitions \ref{defi1w} and \ref{defi1}. We can require this because, although distinct center leaves may merge, they always represent different points in the center leaf space $\lc$.
\end{remark}

\begin{remark}\label{r.lscafsoe}
Note that Definition \ref{defiw2} does not involve a self orbit equivalence explicitly. However, it is easy to note that there is a self orbit equivalence class associated to a leaf space collapsed Anosov flow since the action of a lift $\ft$ of $f$ to $\mt$ induces a permutation of leaves of $\lc$ which via $H$ induces a permutation of orbits of $\phi_t$. The fact that from this one can actually construct a self orbit equivalence follows from a standard averaging argument, see for instance \cite[Theorem 3.4]{Barbot}.
\end{remark}

The homeomorphism $H$ of Definition \ref{defiw2} identifies the center leaf space of a leaf space
collapsed Anosov flow $f$ with the orbit space of an Anosov flow $\phi_t$. The difficulty to go from there to a strong collapsed Anosov flow (Definition \ref{defi1}) is to build a map $h$ on the manifold from the map $H$ which is only on the orbit/center leaf space. This is done (in \S\ref{s.leafspace_implies_strong}) using a standard averaging argument (although made harder by the existence of branching).

There is however a wrinkle to smooth out before this: Suppose $f$ is a partially hyperbolic diffeomorphism, preserving two branching foliations $\cs$ and $\cu$, that is a leaf space collapsed Anosov flow. Then the map $H$ of Definition \ref{defiw2} is not explicitly required to behave well with respect to the center stable and unstable (branching) foliations. That is, $H$ is not assumed to identify the weak (un)stable leaf space of $\phi_t$ with the leaf space of $\cs$ ($\cu$). However, thanks to the fact that pairwise transverse foliations invariant by an Anosov flow are unique (see Proposition \ref{p.nontransitiveAF}), $H$ will automatically identify the weak (un)stable leaf space of the Anosov flows with the center (un)stable leaf space of the diffeomorphism (Proposition \ref{p.wcaf2impliescaf2}). 

Thus we obtain the following equivalence:

\begin{thmintro}\label{teo.main3} 
If a diffeomorphism $f$ is a strong collapsed Anosov flow then it is a leaf space collapsed Anosov flow. Moreover, if $E^s$ or $E^u$ are orientable, the converse also holds. 
\end{thmintro}

In order to prove Theorem \ref{teo.main1}, we will prove that the examples are leaf space collapsed Anosov flows and use Theorem \ref{teo.main3} (some additional work allows to bypass the orientability condition in Theorem \ref{teo.main3}).

\subsection{The space of collapsed Anosov flows}\label{ss.spaceofcollapsed}

It is classical, thanks to the cone criterion, that being partially hyperbolic is a $C^1$-open condition among diffeomorphisms. Based on a result of \cite{HPS}, that we expand upon in Appendix \ref{app.graphtransform}, we are able to obtain a global stability result for collapsed Anosov flows.

\begin{thmintro}\label{teo.main6}
The space of collapsed Anosov flows for a given Anosov flow $\phi_t$ and self orbit equivalence class $\beta$ is open and closed among partially hyperbolic diffeomorphisms on a given $3$-manifold. Similarly, the space of leaf space collapsed Anosov flows is open and closed among partially hyperbolic diffeomorphisms. 
\end{thmintro}

Similar statements for other classes of systems have been obtained in \cite{Pot-T3, FPS}. This result has also been announced for discretized Anosov flows in \emph{any} dimension in \cite{Martinchich}.

In terms of classification, Theorem \ref{teo.main6} gives us that any partially hyperbolic diffeomorphism in a connected component of a collapsed Anosov flow (in the space of partially hyperbolic diffeomorphisms) is also a collapsed Anosov flow, for the same flow and the same self orbit equivalence class. In particular, two leaf space collapsed Anosov flows in the same connected component have center leaf spaces that can be chosen to be homeomorphic (i.e., there exists a pair of branching foliations making the associated center leaf space homeomorphic) and act equivariantly on them.

However, to stay even closer to the spirit of the first efforts at a classification of partially hyperbolic diffeomorphisms, as in \cite{HPS} or Pujals' conjecture \cite{BW}, we may want to ask more: One may hope that inside a connected component, not only are the center leaf spaces homeomorphic, but so is the structure of \emph{branching of center leaves}.
More precisely, suppose that $f_1$ and $f_2$ are two leaf space collapsed Anosov flows associated with an Anosov flow $\phi_t$ and the same self orbit equivalence class $\beta$. Then $\lc_1$ the center leaf space of $f_1$ is homeomorphic to $\lc_2$, via the composition $H_2\circ H_1^{-1}$, where the $H_i$ are as in Definition \ref{defiw2}. However it may a priori happen that two center leaves $c_1,c_1'\in \lc_1$ merge (i.e., have a non empty intersection in $\wt M$), while their images by $H_2\circ H_1^{-1}$ do not.

We show (in \S\ref{ss.DAF}) that this issue does not arise for discretized Anosov flows (or, as a consequence, for collapsed Anosov flows for which the self orbit equivalence class is periodic), but, in general, we do not know whether the branching structure is completely determined by the Anosov flow and the self orbit equivalence class.

\begin{quest}\label{q.branch}
Is the branching locus\footnote{To be precise, consider $\cO_\phi$ to be the orbit space of the lift $\widetilde{\phi_t}$ of $\phi_t$ to the universal cover. We can define the \emph{branching locus} as a function $B\colon \cO_\phi \times \cO_\phi  \to \{0,1\}$ such that $B(o_1,o_2)=1$ if and only if the corresponding center leaves intersect in $\mt$. One could define more refined notions taking into account how many connected components of intersection they have, or the direction on which center curves branch, etc.... All these things could a priori be determined by the data of the flow and the self orbit equivalence class and be independent of the partially hyperbolic diffeomorphism that realizes this as a collapsed Anosov flow.} of a collapsed Anosov flow determined by the Anosov flow $\phi_t$ and the self orbit equivalence class (or, at least, is the branching locus constant in a connected component of partially hyperbolic diffeomorphisms)?
\end{quest}
Related questions can be found in \cite[\S 7]{HPS}.

A first step towards Question \ref{q.branch} could be to prove that, if a collapsed Anosov flow is dynamically \emph{incoherent}, then all collapsed Anosov flows in its connected component are also dynamically incoherent. This is true in certain manifolds, or classes of partially hyperbolic diffeomorphisms (e.g., hyperbolic manifolds \cite[Theorem B]{BFFP_part2}, or Seifert manifolds when the action on the base is pseudo-Anosov \cite{BFFP2}).  One natural, seemingly simple but far from well-understood, class of examples where this is not known is for partially hyperbolic diffeomorphisms in Seifert manifolds which act as a Dehn-twist on the base.

\subsection{Quasigeodesic partially hyperbolic diffeomorphisms}

The last definition we introduce describes a class of partially hyperbolic diffeomorphism that are, in some sense, geometrically well-behaved.

As before, we consider $f\colon M \to M$ a partially hyperbolic diffeomorphism which preserves branching foliations $\cs$ and $\cu$ tangent respectively to $E^{cs}$ and $E^{cu}$ (cf.~\S \ref{s.branching}).

We say that a curve in a leaf $L$ of the lifted branching foliation $\wcs$ (or $\wcu$) is a \emph{quasigeodesic} if it admits a parametrization
$\eta\colon \RR \to L$ such that $c|t-s| + c' \geq d_L(\eta(t), \eta(s)) \geq c^{-1} |t-s| - c'$ for some $c >1, c' >0$ where $d_L$ is the path 
metric induced on $L$ by the pullback metric from $\mt$. A family of curves is \emph{uniformly quasigeodesic} if the constants $c,c'$ can be chosen independently of the curve and the leaf.
Following common usage in the field, we say that a curve $\alpha$ in a leaf $L$ of $\cs$ or $\cu$ is a quasigeodesic, if a lift $\tilde{\alpha}$ to a leaf $\widetilde{L}$ in $\mt$ is a quasigeodesic.

\begin{defi}[Quasigeodesic partially hyperbolic diffeomorphism]\label{defiQG}
Consider $f\colon M \to M$ a partially hyperbolic diffeomorphism. We say that $f$ is \emph{quasigeodesic} if it preserves center stable and center unstable branching foliations with Gromov-hyperbolic leaves such that the center curves are uniform quasigeodesics inside center stable and center unstable leaves.  
\end{defi}

Note that this definition is independent of the choice of Riemannian metric on $M$ (cf.~Proposition \ref{prop-GH}).

For true foliations, deciding whether their leaves are Gromov-hyperbolic can be done via Candel's uniformization theorem (see \cite[\S I.12.6]{CanCon}, although Candel's Theorem is stated for foliations with smooth leaves, it extends to foliation with less regular leaves thanks to \cite{Cal-smoothing}). While Candel's theorem does not directly apply to branching foliations, it may be used when these branching foliations are \emph{well-approximated} by true foliations (see Appendix \ref{ss.Candelmetric}), as occurs
for example in the existence theorem of Burago--Ivanov (Theorem \ref{teo-openbranch}). In particular, one can prove that, when the branching foliations are \emph{minimal} and the manifold has fundamental group with exponential growth, then the leaves are Gromov-hyperbolic \cite[\S 5.1]{FP-acc}. Notice that the weak stable and weak unstable foliations of Anosov flows always have Gromov-hyperbolic leaves \cite{FenleyAnosov}.

It turns out that quasigeodesic partially hyperbolic diffeomorphisms and leaf space collapsed Anosov flows are one and the same class (at least under some orientability conditions), thereby giving a nice geometrical description of (strong) collapsed Anosov flows.

\begin{thmintro}\label{teo.main5}
A leaf space collapsed Anosov flow is a quasigeodesic partially hyperbolic diffeomorphism. Moreover, if the bundles $E^s$ and $E^u$ are orientable, the converse holds.
\end{thmintro}

Although we do not use it to prove Theorem \ref{teo.main1}, this characterization can be used to prove that some partially hyperbolic diffeomorphism are collapsed Anosov flows, as is done in \cite{FP-2}. (In fact, \cite{FP-2} motivated some of the results in this article, including Theorem \ref{teo.main5}.)

\begin{remark}
 Use of leafwise quasigeodesic behavior and the Morse Lemma to prove stability results have a long and fruitful history. While not directly applicable, some of our techniques share a similar philosophy as those previously used by, for instance, Ghys \cite{Ghys}, or, more recently, by Bowden and Mann \cite{BowdenMann}.
\end{remark}

The geometric description we obtain for collapsed Anosov flows is in fact more precise than this: We show that the center leaves of a quasigeodesic partially hyperbolic diffeomorphism must form a \emph{quasigeodesic fan} inside each center stable or unstable leaf, as is the case for orbits of Anosov flows (see Theorem \ref{teo.QG}). In addition, we prove that the branching of center leaves, if it exists, is fairly well-behaved,
see Lemma \ref{l.semiflow}.

A parallel can be made between the cases studied here and the classification of partially hyperbolic diffeomorphisms on $3$-manifolds with (virtually) nilpotent fundamental group: On these manifolds, while the branching foliations are not Gromov hyperbolic, the center leaves may be quasigeodesics inside their branching center (un)stable leaves. Determining when this is the case turned out to be a successful strategy for the classification, see \cite{HPsurvey,HP-cstori}.

\begin{remark}
 Both Theorem \ref{teo.main6} and Theorem \ref{teo.main5}, giving the equivalence between strong collapsed Anosov flows, leaf space collapsed Anosov flows and quasigeodesic partially hyperbolic diffeomorphism require some orientability conditions for one of their directions.
 The knowledgeable reader might surmise that this is linked to the theorem of Burago--Ivanov (Theorem \ref{teo-openbranch}), giving the existence of branching foliations under some orientability conditions of the bundles. This is only partly true: each of the Definitions \ref{defi1}, \ref{defiw2} and \ref{defiQG} assumes already the existence of branching foliations, but what we do need for some arguments from Burago--Ivanov Theorem is that these branching foliations are well approximated by true foliations.
 
 While we think it likely that both Theorem \ref{teo.main6} and Theorem \ref{teo.main5} would hold without the orientability assumptions, we are not able to prove it at this time.

 In particular, one step that would be very helpful to solve this problem, would be to prove uniqueness of the invariant branching foliations tangent to the center stable and center unstable bundles for partially hyperbolic diffeomorphisms (see Question \ref{q.branchingunique}).
 
 The uniqueness question, which has a very wide scope of potential applications, is completely open in general. Here we prove it for the examples of Theorem \ref{teo.main1} in Proposition \ref{prop.uniquebranchingBGHP}.
\end{remark}

\subsection{Realization of self orbit equivalences}

One way of looking at the definition of collapsed Anosov flows is as a partially hyperbolic realization of a self orbit equivalence of an Anosov flow.

Quite clearly, not every self orbit equivalence of an Anosov flow
can be a partially hyperbolic diffeomorphism: Just consider a trivial self orbit equivalence $\phi_{h(\cdot)}$ of an Anosov flow $\phi_t\colon M \to M$ where $h\colon M \to \RR$ is such that $h(x_0) = 0$ for some $x_0\in M$, which therefore cannot be partially hyperbolic. However, if we consider the \emph{equivalence class} of a trivial self orbit equivalence, then that class has an element that can be represented as a partially hyperbolic diffeomorphism.

Therefore, the following natural problem presents itself.

\begin{quest}\label{q.realization_soe}
 Is every self orbit equivalence class of an Anosov flow realized by a collapsed Anosov flow?
\end{quest}

Notice that a positive answer would in particular imply that
there exists examples of partially hyperbolic diffeomorphisms
in hyperbolic $3$-manifolds which are not discretized Anosov flows,
even up to finite powers, see \cite[Theorem B]{BFFP_part2}.

While not complete enough to fully answer Question \ref{q.realization_soe}, the constructions of \cite{BGHP} lead, via Theorem \ref{teo.main1}, to the realization of many classes of self orbit equivalences. In fact, for some Anosov flows, their construction is enough to realize (virtually) all self orbit equivalence classes.

On the other hand, a basic understanding of the self orbit equivalences of Anosov flows, such as the one obtained in \cite{BaG} for those homotopic to the identity, directly leads to restrictions on possible collapsed Anosov flows (up to isotopy).

We choose to illustrate both of these principles on three specific, but important, examples.

In Theorem \ref{teo.CAF_homotopic_to_identity}, we completely describe strong collapsed Anosov flows that are homotopic to identity (this result is obtained from \cite[Theorem 1.1]{BaG}, which describes the self-orbit equivalences of Anosov flows that are homotopic to identity).

In Theorem \ref{teo.CAFonT1S} and Theorem \ref{teo-SOE_on_FranksWilliams}, we show that, on the unit tangent bundle of a hyperbolic surface and when considering the Franks--Williams example (or some generalizations of it) the answer to Question \ref{q.realization_soe} is (virtually) positive and we describe all possible collapsed Anosov flows up to isotopy.

\subsection{Organization of the paper}

In \S\ref{s.branching}, we recall the definition of branching foliation and the existence theorem of Burago--Ivanov. We also state a more precise existence theorem for true foliations that approximate branching foliations. This precision can be extracted from the original proof of Burago--Ivanov and we explain how to do that in Appendix \ref{app.branching}.

In \S\ref{s.openclosed}, we prove Theorem \ref{teo.main6}. To prove it, we first recall some results that can be extracted from \cite{HPS}, as explained in Appendix \ref{app.graphtransform}.

In \S\ref{s.AF}, we prove some general facts on topological Anosov flows and show that discretized Anosov flows (in the sense of \cite{BFFP_part1,BFFP_part2}) are (strong) collapsed Anosov flows.

In \S\ref{s.QG}, we prove (Theorem \ref{teo.QG}) that the center leaves of a quasigeodesic partially hyperbolic diffeomorphism must make a \emph{quasigeodesic fan} in each center (un)stable leaf. To prove this, we study general subfoliations by quasigeodesic leaves of a foliation and obtain some results that apply in the general case.

In \S\ref{s.criteria}, we prove Theorem \ref{teo.main5}.

In \S\ref{s.strong_implies_leafspace} and \S\ref{s.leafspace_implies_strong}, we prove both directions of Theorem \ref{teo.main3}.

In \S\ref{s.examples}, we prove Theorem \ref{teo.main1}. We also (see \S\ref{ss.uniqueness_center_curves}) prove a result about the uniqueness of center stable and center unstable branching foliations, in the setting of the examples of Theorem \ref{teo.main1}, as well as a (branching) version of unique integrability of their center curves (Proposition \ref{prop.uniquecenters}).

Finally, in \S\ref{s.classification_results}, we prove some classification results about collapsed Anosov flows and self orbit equivalences.

\subsection{On some simplifying assumptions}\label{ss.simplifying_assumption}

Some elements of this article are technical in nature, in part because of the broad generality of our context. So we mention here a simplifying assumption that the reader could make.

If all partially hyperbolic diffeomorphisms are assumed to preserve branching foliations that are \emph{$f$-minimal} (that is, the only closed, non-empty, $f$-invariant set that is foliated by center stable or center unstable leaves is the whole manifold), then the following occurs:
\begin{enumerate}
 \item If one of the branching foliations is $f$-minimal and $f$ is any version of collapsed Anosov flow, then the associated topological Anosov flow will be transitive (see Remark \ref{r.transitiveAF}). Hence, Shannon's result \cite{Shannon} applies and, up to an orbit equivalence, the topological Anosov flow can then be assumed to be a smooth Anosov flow.
 \item In \S\ref{s.QG}, which contains the key for the proof of Theorem \ref{teo.main5}, subsection \ref{ss.general_result_laminations} can be skipped, as Proposition \ref{prop.WQGF} then follows trivially from Proposition \ref{prop.closedfans}.
\end{enumerate}

The $f$-minimality of branching foliations follows readily from any of the following assumptions on $f$: $f$ is (chain)-transitive, $f$ is volume preserving, the manifold $M$ is hyperbolic (\cite{BFFP_part2}), or if $f$ is in the same connected component (among partially hyperbolic diffeomorphisms) from a transitive one (see Proposition \ref{p.fminimalopenclosed}).


\section{Branching foliations and leaf spaces}\label{s.branching} 

In this section we review the notion of branching foliations introduced in \cite{BI} and their leaf spaces. Under some orientability assumptions, partially hyperbolic diffeomorphisms always preserve branching foliations which are well approximated by foliations, so it makes sense to consider partially hyperbolic diffeomorphisms preserving some branching foliations. We assume basic familiarity with foliations in 3-manifolds, see, e.g., \cite[Appendix B]{BFFP_part1} and references therein.

Given a plane field $E$ in a 3-manifold $M$ we call \emph{complete surface tangent to $E$} a $C^1$-immersion $\varphi\colon U \to M$ from a simply connected domain $U \subset \RR^2$ into a 3-manifold $M$ which is complete for the pull-back metric and such that $D_p\varphi(\RR^2) = E(\varphi(p))$ at every $p \in U$. 

\begin{defi}[Branching foliation]\label{def.branching}
A branching foliation $\cF$ of a 3-manifold $M$ tangent to $E$ is a collection of complete surfaces tangent to $E$ such that:
\begin{enumerate}[label= (\roman*)]
\item\label{item.BF_covereverything} every point $x \in M$ belongs to (the image of) some surface, 
\item\label{item.BF_nocrossing} the surfaces pairwise do not topologically cross (see below),
\item\label{item.BF_complete} the family is \emph{complete}, i.e., it contains all the limits of its leaves in the pointed compact open topology (see below),
\item\label{item.BF_minimal} it is minimal in the sense that one cannot remove any surface from the collection and still satisfy properties \ref{item.BF_covereverything} to \ref{item.BF_complete}.  
\end{enumerate}
\end{defi}

The condition of no topological crossing is quite subtle, since the crossing may take place far in the manifold (it cannot be defined locally, and it is part of the reason surfaces are defined in terms of simply connected domains). Following \cite[Section 4]{BI}, given two complete surfaces $\varphi\colon U \to M$ and $\psi\colon V \to M$ tangent to $E$ we say that they \emph{topologically cross} if there is a curve $\gamma\colon (0,1) \to U$ a $C^1$-immersion $\Psi\colon V \times (-\eps,\eps) \to M$ such that $\Psi(\cdot,0)= \psi$ and a map $\tilde\gamma \colon (0,1) \to V \times (-\eps,\eps)$ whose image intersects both $V \times (0,\eps)$ and $V\times (-\eps,0)$ such that $\varphi \circ \gamma = \Psi \circ \tilde \gamma$. This notion is well defined and symmetric on the surfaces. 

\begin{remark}
The key difference between branching foliations defined above and the \emph{branched laminations} introduced in \cite[\S 6.B]{HPS} is the additional assumption \ref{item.BF_nocrossing} that there are no topological crossing between surfaces. Of course, this added notion only makes sense in the codimension one setting. 
\end{remark}

The notion of completeness of the family stated in the definition of branching foliation should be understood in the following sense: Let $\varphi_n\colon U_n \to M$ be a sequence of complete surfaces tangent to $E$ in the family. Suppose that $\varphi_n(p_n) \to x$ for some points $p_n \in U_n$. Then, there exists a surface $\varphi\colon U \to M$ in the family that verifies the following. Given a point $p \in U$, there is an arbitrarily large ball around $p$ and large balls around each $p_n\in U_n$ on which the map $\varphi$ is $C^1$-close to some reparametrization (see
next paragraph) of the maps $\varphi_n$ (see also \cite[Lemma 7.1]{BI}). 

Note that condition \ref{item.BF_minimal} above is not stated explicitly in \cite{BI}, but can be easily deduced by choosing one leaf in each equivalence class (up to topological reparametrization). There is a lot of ambiguity for the choice of parametrizations and since we want to focus on their images,
we want to avoid it. For that, we will say that two complete surfaces $\varphi\colon U \to M$ and $\psi\colon V \to M$ tangent to $E$ are the \emph{same up to reparametrization} if there is a homeomorphism $h\colon U \to V$ such that $\varphi = \psi \circ h$. 

It is standard in the literature to abuse notation and talk about leaves of a branching foliation $\cF$ to refer either to the complete surface $\varphi\colon U \to M$ up to reparametrization, or to its image. In this article, we will try to avoid this for clarity. In fact, some of our results a posteriori help justifying this classical abuse (see the end of Remark \ref{rem-novikov}). 

In \cite[Theorem 7.2]{BI} it is shown that branching foliations can be approximated arbitrarily well by true foliations. The statement of \cite[Theorem 7.2]{BI} does not state explicitly some properties of the approximation that we will need.  Precisely, item \ref{item.ApporxFol_uniqueness} below is not stated in \cite[Theorem 7.2]{BI}, but we explain in Proposition \ref{prop-uniqueleaf} of the appendix how it follows from its proof.

\begin{teo}\label{teo-openbranch}
Let $\cF$ be a branching foliation tangent to a transversely orientable distribution $E$ on a closed 3-manifold $M$. Then, for every $\eps>0$ there exists a foliation $\cF_\eps$ with smooth leaves and a continuous map $h_\eps\colon M \to M$ such that the following conditions hold:
\begin{enumerate}[label=(\roman*)]
\item\label{item.ApporxFol_smallangle} the angle between $E$ and $T\cF_\eps$ at every point is smaller than $\eps$, 
\item\label{item.ApporxFol_uniqueness} for every surface $\varphi\colon U \to M$ in $\cF$ there is a unique leaf $L$ of the foliation $\cF_\eps$ such that $h_\eps$ is a local $C^1$ diffeomorphism from $L$ to the surface: That is, for every $x \in L$ there is a neighborhood $V$ of $x$ in $L$ and an open subset $W \subset U$ such that $\varphi^{-1} \circ h_{\eps}\colon V \to W$ is a diffeomorphism,
\item\label{item.ApporxFol_smalldistance} $d(h_\eps(x),x)< \eps$ for every $x\in M$. 
\end{enumerate}
\end{teo}

\begin{remark}
 While the theorem of Burago and Ivanov, a priori only gives an approximating foliation with $C^1$ leaves, thanks to work of Kazez and Roberts \cite{KazezRoberts}, any foliation with $C^1$ leaves can be approximated by a foliation with smooth ($C^{\infty}$) leaves with the same properties. Hence, one can then obtain a family of foliations with smooth leaves that approximates the branching foliations.
\end{remark}

The uniqueness of the correspondence between leaves of the true and branching foliations, given by item \ref{item.ApporxFol_uniqueness} above, allows to simplify the definition of the leaf spaces of the center stable, center unstable and center (branching) foliations given in \cite[Section 3]{BFFP_part2}. 

Let $f\colon M \to M$ be a partially hyperbolic diffeomorphism and assume that $f$ preserves branching foliations $\cs$ and $\cu$ tangent respectively to $E^{cs}=E^s \oplus E^c$ and $E^{cu}=E^c \oplus E^u$. That $f$ preserves the branching foliation means that each surface of the collection is mapped, up to reparametrization, to another surface in the collection. 

\begin{remark}
 We pause here to emphasize the different assumptions about the respective smoothness of the foliations we consider in this article. While all our foliations and branching foliations are assumed to be only continuous transversaly, the assumed regularity of their leaves is different depending on the context:
 \begin{itemize}
  \item If $\cF^{ws}$ is the weak (un)stable foliation of a topological Anosov flow, then the leaves of $\cF^{ws}$ are only assumed to be $C^0$. In particular, they may not even be rectifiable, which leads to some technicalities in \S\ref{s.AF}.
  \item If $\cs$ is a branching foliation of a partially hyperbolic diffeomorphism, then the leaves of $\cs$ are assumed to be at least $C^1$.
  \item If $\cF_\eps$ is an approximating foliation for a branching foliation $\cs$, then the leaves of $\cF_\eps$ are assumed to be smooth ($C^{\infty}$).
 \end{itemize}
 The reason for these different regularity assumptions, is that we chose to take the broadest, and most ``natural'', definitions for each categories of objects.
\end{remark}

\begin{obs}\label{rem-novikov} 
By \cite[Lemma 2.3]{BI}, there is no closed contractible curve everywhere transverse to $E^{cs}$. Indeed, thanks to Novikov theorem, a closed transversal would imply the existence of a Reeb component for some of the approximating foliations. Now, a Reeb component would force disks inside to be sent into themselves by the unstable holonomy (see \cite[Lemma 2.2]{BI}) and thus one would obtain a closed curve tangent to $E^u$ which is impossible. This also has the important consequence that the approximating foliation $\fes$ given by Theorem \ref{teo-openbranch}, for small $\eps$, is \emph{Reebless}. The same holds for $E^{cu}$ and $\cu$ and $\feu$. Notice that once one knows that the branching foliation is Reebless, one can simplify a bit its treatment, in particular, when lifting to the universal cover, there is no ambiguity in identifying surfaces up to reparametrization with their images. 
\end{obs}

Remark \ref{rem-novikov} and Palmeira's theorem imply that the universal cover $\mt$ of $M$ is homeomorphic to $\RR^3$ and that the leaf space of the lifted foliations 
$\wfes$ and $\wfeu$ of $\fes$ and $\feu$ respectively are $1$-dimensional simply connected (possibly non-Hausdorff) manifolds. Theorem \ref{teo-openbranch} implies that all of these leaf spaces 
$\wfes$ for $\eps > 0$ are independent of $\eps>0$ and
are naturally bijective with the leaf space of $\wcs$.
It allows one to put a topology on the leaf space $\lcs = \mt/_{\wcs}$ 
which is the same as the topology on 
 $\mt/_{\wfes}$ independently of $\eps > 0$.
In the same way we define a topology on $\lcu = \mt/_{\wcu}$. 
It also allows to define the action of $\ft$, a lift of $f$ on these spaces, since $\ft$ preserves the lifted branching foliations $\wcs$ and $\wcu$.

The same holds for every deck transformation $\gamma \in \pi_1(M)$ that acts on these leaf spaces canonically. 
Using these identifications there is a canonical action on the
leaf spaces of $\wfes, \wfeu$ by either lifts of $f$ or deck 
transformations.

We obtain also a way to define a leaf space $\lc$ for the center branching foliation. A center leaf in $\mt$ is a connected component of the intersection
of a leaf $L$ of $\wcs$ and a leaf $U$ of $\wcu$. The center leaf space
is this set with the natural topology induced from the quotient of the
subset of the Cartesian product of the two original leaf spaces.
Another way to see this is using the identification of leaf spaces of
$\wcs, \wfes$;  and $\wcu, \wfeu$ to define the center leaf space 
 as the leaf space of the foliation $\widetilde{\cF^c_\eps}$ obtained as the intersection of the foliations $\wfes$ and $\wfeu$. This is again well defined independently of $\eps$ and there is a well defined action of $\pi_1(M)$ on this leaf space as well as an action of any lift, $\ft$, of $f$ to $\mt$. 

\begin{remark}\label{rem.leafspaces}
The notions of leaf spaces of $\wcs$ and $\wcu$ coincide with the ones studied in \cite[Section 3]{BFFP_part2} where we did not rely on the approximating foliations.
The definition of the center leaf space $\cL^c$ taken here may however differ slightly\footnote{We do not know whether there exists examples where the two definitions are actually different, but, at least formally, they are not the same.} from the one defined in \cite[Section 3]{BFFP_part2} which is a quotient of this definition:
In \cite[Section 3]{BFFP_part2} if two connected components $c_1$ of $L_1 \cap U_1$ and $c_2$ of $L_2 \cap U_2$ ($L_i$ in $\wcs$, $U_i$ in $\wcu$) are the
same set in $\wt M$, then they produced a single center leaf. Here
we do not identify them. So the center leaf space defined in \cite{BFFP_part2}
is a quotient of the one we define here.
\end{remark}

In the cases we will be interested in, there will be a nice topological structure in the leaf space $\lc$ which will be homeomorphic to $\RR^2$. Notice that in this setting, the foliations induced in $\lc$ by $\wcs,\wcu$ (or by $\wfes,\wfeu$ using the above identifications) are (topologically) transverse and invariant under the action of $\pi_1(M)$ and $\ft$. 

The assumption that a partially hyperbolic diffeomorphism of a 3-manifold preserves branching foliations is justified, since it always holds up to finite cover and iterate as the following fundamental result of \cite{BI} shows. 

\begin{teo}[Burago--Ivanov]\label{BI}
Let $f\colon M \to M$ be a partially hyperbolic diffeomorphism with splitting $TM =E^s \oplus E^c \oplus E^u$ such that the bundles are oriented and $Df$ preserves their orientation. Then, there are $f$-invariant branching foliations $\cW^{cs}$ and $\cW^{cu}$ tangent respectively to $E^{cs}=E^s \oplus E^c$ and $E^{cu}=E^c \oplus E^u$. 
\end{teo}

To be precise, the invariance by $f$ of the branching foliations means the following:
If $(\varphi,U)$ is a leaf of $\cs$ then $(f \circ \varphi, U)$ is also a leaf of $\cs$ modulo reparametrization.

Notice that the branching foliations constructed in \cite{BI} are invariant under \emph{every} diffeomorphism that preserves the bundles $E^{cs}$ and $E^{cu}$ and preserves orientations of the
bundles $E^c, E^s$ and $E^u$. Some other consequences of their construction is explored in Appendix \ref{app.branching}. One goal  being to better understand the uniqueness properties these foliations may have. 

\begin{notation}\label{not.branching}
Given a branching foliation $\cF$ on $M$ we will use the notation $(\varphi,U)\in \cF$ to refer to the surface $\varphi\colon U \to M$. If $f\colon M \to M$ is a diffeomorphism that preserves the branching foliation $\cF$, then we write $f(\varphi,U)$ for the leaf $(\psi,V) \in \cF$ which is a reparametrization of $(f\circ \varphi, U)$. 
\end{notation}

\section{The space of collapsed Anosov flows}\label{s.openclosed}

We want to show Theorem \ref{teo.main6}. First we recall some results from \cite{HPS} that we need. 

\subsection{Graph transform method}\label{ss.graphtransform} 
The structural stability results of Hirsch, Pugh and Shub \cite{HPS} provide conditions implying that perturbations of a partially hyperbolic diffeomorphisms preserving a foliation tangent to the center direction are \emph{leaf conjugate} to the original one. Their classical stability result (see \cite[\S 7]{HPS}) requires the center bundle to be integrable plus a technical condition called \emph{plaque expansivity}. We refer the reader to \cite{HPS} for the precise definitions of these notions since we do not use them here. 

In \cite[\S 6]{HPS} the authors develop a more general theory that permits
leaves to merge (see also \cite[Theorem 4.26]{CP}). The more general
theory allows one to remove the technical conditions at the expense of not knowing if centers remain disjoint after perturbation. Since in our case this is what usually happens, this is precisely what we need. 

We will use a couple of variants of \cite[Theorem 6.8]{HPS} (which is part of \cite[Theorem 6.1]{HPS}). The version we need here is a uniform version of the results there. The key observation is that  the proof given in \cite{HPS} provides uniform estimates that depend, rather than on the diffeomorphism, on some of its properties which hold in uniform neighborhoods of a given partially hyperbolic diffeomorphism. We will provide a sketch of the proof in  Appendix \ref{app.graphtransform}. We state the results in dimension 3  where we will use them, but similar results should hold in any dimension (see \cite{Martinchich}). 

We first need some definitions from \cite{HPS}.  A \emph{$C^1$-leaf immersion} is a $C^1$-immersion, $\imath\colon V\to M$, of a manifold $V$ (which is typically a disjoint union of possibly uncountably many connected complete manifolds) to $M$ whose image is a closed set in $M$. To give a metric to $V$ we consider the metric in each connected component of $V$ and declare distance equal to $\infty$ between points in different connected components\footnote{Technically this is not a metric, but it serves well for notions such as being uniformly close, etc.}.

For a diffeomorphism $g\colon M \to M$, a $C^1$-leaf immersion $\imath\colon V\to M$ is said to be \emph{$g$-invariant} if there exists a $C^1$-diffeomorphism $\imath_\ast g\colon V \to V$ verifying $\imath \circ \imath_\ast g = g \circ \imath$. 
Two $C^1$-leaf immersions $\imath, \imath'$ from $V$ to $M$ are said to be \emph{$C^1$-close} if they are uniformly $C^1$-close, meaning that there exists $\eps>0$ such that for every $x \in V$ we have $d(\imath(x),\imath'(x))< \eps$ and\footnote{To make sense of the difference of derivatives, one can for instance, embed $M$ in some $\RR^k$ with large $k$.}  $\|D_x\imath - D_x\imath'\|< \eps$. 

\begin{teo}\label{teo.uniformHPS}
Let $f_0\colon M \to M$ be a partially hyperbolic diffeomorphism. There exists  a $C^1$-neighborhood $\cU$  of $f_0$ such that if $g,g'\in \cU$ and $\imath_g\colon V \to M$ is a $g$-invariant $C^1$-leaf immersion tangent to $E^c_g$, then there exists $\imath_{g'}\colon  V \to M$ a $g'$-invariant $C^1$-leaf immersion tangent to $E^c_{g'}$ and $C^1$-close to $\imath_g$ and a homeomorphism $\tau\colon V \to V$ which is $C^0$-close to the identity\footnote{In particular, it preserves connected components.} verifying that $(\imath_g)_\ast g(x) = (\imath_{g'})_\ast g' (\tau(x))$ for every $x \in V$. 
\end{teo}

We will also need a version of the result above for branching foliations\footnote{See \cite[\S 6.B]{HPS} for the related notion of branched lamination (which is also intimately  related to the statement of Theorem \ref{teo.uniformHPS}) which differs from the notion of branching foliations we use in this paper.  The latter has to do with codimension one phenomena and features the non-topological crossing condition that makes no sense in the setting of \cite[\S 6.B]{HPS}.}.

\begin{defi}\label{defi-epseq} If $g\colon M \to M$ and $g'\colon M \to M$ are partially hyperbolic diffeomorphisms preserving respectively branching foliations $\cs_g$ and $\cs_{g'}$ tangent to $E^{cs}_g$ and $E^{cs}_{g'}$. We say that $\cs_g$ and $\cs_{g'}$ are $\eps$-\emph{equivalent} if:
\begin{enumerate}
\item There exists a $\pi_1(M)$-invariant homeomorphism $H$ from $\cL^{cs}_g$ to $\cL^{cs}_{g'}$, the leaf spaces of $\widetilde{\cs_g}$ and $\widetilde{\cs_{g'}}$ in $\mt$ respectively.
\item There are lifts $\tilde{g}$ and $\tilde{g'}$ of $g$ and $g'$ respectively such that the actions on $\cL^{cs}_g$ and $\cL^{cs}_{g'}$ are conjugate via $H$, that is, $H \circ \tilde{g} = \tilde{g}' \circ H$. 
\item\label{it.closeleafuc} Given $L =(\varphi, U) \in \cL^{cs}_g$ a leaf of $\widetilde{\cs_{g}}$ we have that the leaf $H(L) = (\psi, V)$ of $\widetilde{\cs_{g'}}$ is uniformly $\eps$-$C^1$-close to $L$. This means that there exists a diffeomorphism $\eta\colon U \to V$ such that the maps $\varphi$ and $\psi \circ \eta$ are uniformly $\eps$-close as well as their derivatives. 
\end{enumerate}
 \end{defi}
 We can now state the result we will need. 
 
 \begin{teo}\label{teo.graphtransform}
 Let $f_0\colon M \to M$ be a partially hyperbolic diffeomorphism of a closed 3-manifold $M$. There exists  an open neighborhood $\cU$ of $f_0$ in the $C^1$ topology and $\eps>0$ with the property that every $g \in \cU$ is partially hyperbolic and if $\cs_g$ is a branching foliation tangent to $E^{cs}_g$ and invariant under $g$, then, for every $g' \in \cU$ there is a branching foliations $\cs_{g'}$, invariant under $g'$ and $\eps$-equivalent to $\cs_g$. 
 \end{teo}

The proof of both Theorem \ref{teo.uniformHPS} and Theorem \ref{teo.graphtransform} are the same as the ones given in \cite{HPS} with some simplifications (due to the fact that we are only interested in part of their statement). The key difference is that we are claiming that the size of the neighborhoods where the results of \cite[\S 6]{HPS} hold are uniform and depend only on some constants of the diffeomorphism (like the $C^1$-norm, the angle between the bundles and the contraction/expansion rates) and not the diffeomorphism itself. 

For the convenience of the reader, we will include a short sketch of the proof of Theorem \ref{teo.graphtransform} in Appendix \ref{app.graphtransform} (part of the justification for this appendix is the fact that \cite[\S 6]{HPS} proves many other results and what we need is not always easily separated from what we do not need).  The sketch will also serve to show how the uniform estimates, which are the main tool we will use follow from the same arguments (and how the non-crossing condition is automatically satisfied). This will also allow us to explain how Theorem \ref{teo.uniformHPS} follows from \cite{HPS}.

\begin{remark}
As remarked by a referee, it is possible that some of the arguments in the proofs of Theorem \ref{teo.uniformHPS} and \ref{teo.graphtransform} can be made to show that the maps obtained for the stability vary continuously with respect to parameters if one makes a deformation of the original partially hyperbolic map (that is, a curve $f_t$ of partially hyperbolic diffeomorphisms so that $f_0$ is the original map). There is some difficulty in the fact that the leaf conjugacies are not really canonical, but it is true that understanding this further could provide useful information on how the collapsing maps vary with respect to the map. We do not enter into this problem in this paper. 
\end{remark}

\subsection{Proof of Theorem \ref{teo.main6}}\label{ss.open}

Recall (cf.~Remark \ref{r.lscafsoe}) that a leaf space collapsed Anosov flow induces naturally a self orbit equivalence class (i.e., a self orbit equivalence up to trivial self orbit equivalences).  

\begin{prop}\label{prop.open}
Let $f\colon M \to M$ be a partially hyperbolic diffeomorphism. There exists a neighborhood $\cU$ of $f$ such that if there is $g \in \cU$ which is a leaf space collapsed Anosov flow associated to a topological Anosov flow $\phi_t\colon M \to M$ and a self orbit equivalence class $[\beta]$ the following holds: every $g' \in \cU$ is a leaf space 
collapsed Anosov with respect to $\phi_t$ and $[\beta]$. 
\end{prop}

\begin{proof}
Let $\cU$ be a neighborhood given by Theorem \ref{teo.graphtransform}. Then, for any $g'\in \cU$, we obtain a pair of branching foliations $\cs_{g'}$ and $\cu_{g'}$ with the same dynamics as $g$ in their leaf spaces. 

Let  $\tild{\cs_{g'}}$ and $\tild{\cu_{g'}}$ be the lifted foliations to the universal cover.
For each leaf $F$ of $\tild{\cs_{g}}$ there is a unique leaf $F' = H(F)$
of $\tild{\cs_{g'}}$ which is $\eps$ close to it, and vice versa. (Recall Definition \ref{defi-epseq}(\ref{it.closeleafuc}).)

Let $c$ be a center leaf of $g$. It is a connected component of the intersection
of a leaf $F$ of $\tild{\cs_g}$ and $L$ of $\tild{\cu_g}$. Hence, there is a unique component $c'$ of the intersection of 
$H(F)$ and $H(L)$ which is $\eps$-close to $c$. So the 
center leaf spaces of $g$ and $g'$ are equivariantly homeomorphic. 
So one gets Definition \ref{defiw2} for $g'$ which implies that $g'$ 
is a leaf space collapsed Anosov flow with respect to the flow $\phi_t$. Moreover, $H$ conjugates the respective actions of lifts of $g$ and $g'$ on their center leaf spaces. Hence, their corresponding self orbit equivalence are equivalent (cf.~Remark \ref{r.lscafsoe}).  
 \end{proof}
 
 This proposition implies Theorem \ref{teo.main6} for leaf space collapsed Anosov flows. 
The open property is immediate.
In order to see that it is a closed property consider $f_n\colon M \to M$ leaf space collapsed Anosov flows converging to a partially hyperbolic diffeomorphism $f\colon M \to M$. If we apply Proposition \ref{prop.open} to $f$ we get a neighborhood $\cU$ such that if $g \in \cU$ is leaf space collapsed Anosov flow, then every $g' \in \cU$ is leaf space collapsed Anosov flow. Since $f_n \to f$ it follows that for large $n$ we have that $f_n \in \cU$ and so we can apply the proposition with $g=f_n$ and $g'=f$. 

An analogous proof, below, will give Theorem \ref{teo.main6} for collapsed Anosov flows using Theorem \ref{teo.uniformHPS} instead of Theorem \ref{teo.graphtransform}.
This case is however much more involved because we need to construct
a map in the manifold and not just on the leaf space level.

\begin{prop}\label{p.CAFoc} The space of collapsed Anosov flows is open and closed
among partially hyperbolic diffeomorphisms.
\end{prop}

\begin{proof}
Let $f_0\colon M \to M$ be a partially hyperbolic diffeomorphism. We will show that there is a neighborhood $\cU$ of $f_0$ satisfying that if there is $g \in \cU$ which is a collapsed Anosov flow, then every $f \in \cU$ is a collapsed Anosov flow. This shows that being collapsed Anosov flow is an open and closed property among partially hyperbolic diffeomorphisms as explained above. 

For such a $f_0\colon M \to M$ we will take $\cU$ to be the neighborhood given by Theorem \ref{teo.uniformHPS} and assume that there is $g \in \cU$ which is a collapsed Anosov flow. That is, there exists an Anosov flow $\phi_t\colon M \to M$, a continuous map $h\colon M \to M$ homotopic to the identity as in Definition \ref{defi1w} and a self orbit equivalence $\beta$ such that $g \circ h = h \circ \beta$. We want to construct, for $g' \in \cU$ a map $h'\colon M \to M$ and a self orbit equivalence $\beta'$ of $\phi_t$ which verify Definition \ref{defi1w}. 

First, we will consider a leaf immersion $\imath_g\colon V \to M$ induced by $h$ and $\phi_t$. This is defined as follows: Consider $V$ to be the disjoint union of orbits of $\phi_t$, each one with the smooth structure induced by the length of the curves in $M$. Note that even if $V$ is a disjoint union, we can think of points in $V$ as points of $M$ so we can apply both $h$ and $\beta$ to these leaves. We define $\imath_g(x)= h(x)$. This is a well defined $C^1$-leaf immersion since leaves can be lifted to the universal cover where the lift of $g$ acts and induces a map from $V$ to $V$ which is exactly $\beta$. In particular, we get that $(\imath_g)_\ast g = \beta$. 

We now consider $g' \in \cU$ and Theorem \ref{teo.uniformHPS} gives us a $C^1$-leaf immersion $\imath_{g'}\colon V\to M$ and a homeomorphism $\tau\colon V\to V$ which is globally $C^0$-close to the identity such that $(\imath_g)_\ast g(x) = (\imath_{g'})_\ast g' (\tau(x))$. We need to construct $h'$ and $\beta'$ using this map. 

Note that for $\varphi\colon V \to V$ a $C^1$-diffeomorphism we get that $\imath_{g'} \circ \varphi$ is also a $C^1$-leaf immersion with the same properties, so we need to show that there is a choice of $\varphi$ which makes $h'\colon M \to M$ continuous when defined as $h'(x) = \imath_{g'} \circ \varphi(x)$ where we identify $V$ with $M$ as a set. 
The subtlety here is that even though $V$ and $M$ are identified as
sets, their topologies are completely different. In particular
$V$ has many more open sets.

To obtain $h', \beta'$ we will take advantage of the fact that $\imath_g$ was defined using $h$ which is continuous and that $\imath_{g'}$ and $\imath_g$ are uniformly $C^1$-close. 
First some local considerations.
The curves $i_g(\alpha)$ where $\alpha$ is a component of $V$ are
all integral curves of $E^c_g$ and likewise those of $i_{g'}$ 
are integral curves of $E^c_{g'}$.
In a fixed small scale one can choose local coordinates $(x,y,z)$
so that the curves $i_g(\alpha)$ are all $\eps_0$ \ $C^1$-close
to vertical curves,
with fixed $\eps_0$. The same happens for $i_{g'}(\alpha)$. 
Hence in a local box, for a fixed point $z$ in $i_{g'}(\alpha)$ 
there is a unique point denoted by $\eta(z)$
 in the corresponding local sheet
of $i_{g}(\alpha)$ which is the closest point to
$z$. This defines a function $\eta$. Switching the
roles, this implies that this function is locally injective.
Finally this function has
derivative which is non zero everywhere.

So, given $x \in V$ we consider $I_x$ the $\eps$-neighborhood around $x$ with the metric of $V$ (induced by $M$), in particular this is
contained in the same component of $V$.
Consider $\ell_x = \imath_g(I_x)$.  
Take $\varphi_0(x)$ to be the preimage by $\imath_{g}$ restricted to $\ell_x$ of the closest point in $\ell_x$ to $\imath_{g'}(x)$. The map $\varphi_0\colon V \to V$ is continuous and close to the identity. By integrating in $I_x$ and using the orientation, one can make it to be a $C^1$-diffeomorphism $\varphi$
of $V$. We claim that $\imath_{g'} \circ \varphi^{-1}$ works as a choice of $h'$. Consider $x_n \to x$ a converging sequence in $M$. It follows that $\ell_{x_n} \to \ell_x$ uniformly in the $C^1$-topology. We get that $\imath_{g'}(\varphi^{-1}(x_n))$ is the closest point in average to $\ell_{x_n}$ in $\imath_{g'}(I_{x_n})$ and this point varies continuously. 
This shows that $\imath_{g'} \circ \varphi^{-1}$ is continuous seen
as a map from $M$ to $M$.

Once we have the leaf immersion $\imath_{g'} \circ \varphi^{-1}$, that we 
will now just rename as $\imath_{g'}$ to simplify notation,
 we can also define $\beta'\colon M \to M$ viewing $M$ as the disjoint union of the components of $V$. We set $\beta'(x) = (\imath_{g'})_\ast g' (x) \in V$ and consider this to be in $M$. The map $\beta'$ is bijective since it is bijective in each component of $V$ and maps components to components bijectively. We need to check that $\beta'$ is continuous with the topology of $M$ (which is weaker than the one of $V$). But the continuity of $\beta'$ is a direct consequence of the continuity of $g'$ which forces the maps $(\imath_{g'})_\ast g'$ in different components of $V$ to be close when the components are close in $M$. 

The equation $g' \circ h' = h' \circ \beta'$ is automatically verified. 
\end{proof}

\begin{remark}
We can also show that being a quasigeodesic partially hyperbolic diffeomorphism is an open and closed property:
If $f$ is a quasigeodesic partially hyperbolic diffeomorphism, in a finite cover, an iterate of $f$ is a leaf space collapsed Anosov flow (cf.~Theorem \ref{teo.main5}). 
Suppose that $f_n \to f$ is a sequence of quasigeodesic partially hyperbolic diffeomorphisms converging to a partially hyperbolic diffeomorphism $f$.
Using the neighborhood $\mathcal{U}$ of $f$ given 
by Theorem \ref{teo.graphtransform}, it follows that
there are 
$f$-invariant
branching foliations tangent to $E^{cs}, E^{cu}$.
Let $g$ be a lift of a finite iterate $f^i$ of $f$ to a finite cover
$M_1$ of $M$ so that the lifted bundles $E^c, E^s, E^u$ in $M_1$
are orientable and $g$ preserves the orientations. Let $g_n$ 
be the lifts to $M_1$ of $f_n^i$ which converge to $g$.
Since $g_n$ converges to $g$ and $g$ preserves orientations of
the bundles then the same happens for $g_n$ for $n$ big enough. We assume it is true for all $n$. It now follows
from Theorem \ref{teo.main5} that the $g_n$ are leaf space collapsed Anosov flows. By Theorem \ref{teo.main6} it follows that $g$ is a
leaf space collapsed Anosov flow.
Hence using Theorem \ref{teo.main5} again, it follows that 
$g$ is a quasigeodesic partially hyperbolic diffeomorphism.
Since $f$ itself preserves branching foliations, it now follows
that $f$ is a quasigeodesic partially hyperbolic diffeomorphism,
because the foliations of
$g$ obtained, up to taking subsequences, as limits of the branching foliations of $g_n$, are lifts of foliations of $f$ (which are limit of the branching foliations of $f_n$).
This proves that being a quasigeodesic partially hyperbolic diffeomorphism
is a closed property among partially hyperbolic diffeomorphisms. 
The open property is proved analogously.
\end{remark}

As mentioned in \S\ref{s.intro}, we may wonder whether a collapsed Anosov flow is automatically a strong, or leaf space, collapsed Anosov flow. Notice that, if not, then Theorem \ref{teo.main6} implies that there is at least one entire connected component of partially hyperbolic diffeomorphisms on which all maps are collapsed Anosov flows, but none are leaf space collapsed Anosov flows.

To try to decide whether all collapsed Anosov flows are leaf space collapsed Anosov flows, one tool that would greatly help is if the following was true:

\begin{quest}\label{q.branchingunique}
Let $M$ with non virtually solvable fundamental group and $f \colon M \to M$ a collapsed Anosov flow. Suppose that the bundles $E^c, E^s, E^u$ are orientable.
 Is the invariant branching foliation of $f$
tangent to the center stable (resp. the center unstable) bundle unique?
\end{quest}

Notice that this question also naturally arises in the existence theorem of Burago--Ivanov (Theorem \ref{teo-openbranch}), as their construction yields two, a priori distinct, center (un)stable branching foliations (see also Appendix \ref{app.branching}). For virtually solvable fundamental group, there are examples where this question admits a negative answer \cite{HHU} (i.e. one can create dynamically coherent examples which also admit other branching foliations\footnote{In fact, as pointed out by a referee, a very interesting example can be made in the spirit of \cite{HHU} showing that a partially hyperbolic diffeomorphism may leave invariant a foliation with $C^0$-leaves which admits a flow orbit equivalent to an Anosov flow, while not being a collapsed (nor discretized) Anosov flow in our definition (in particular, the center bundle of this example is not orientable). These examples are not known to exist if the fundamental group of $M$ is not virtually solvable.}). In the other extreme, for hyperbolic 3-manifolds, we know that the answer to the question is affirmative (see \cite[\S 10]{FP-2}). More evidence that could indicate a positive answer to Question \ref{q.branchingunique} in manifolds with non-virtually solvable fundamental group is in Section \ref{s.examples}. 

But the scope of potential use, if Question \ref{q.branchingunique} were to be true, is much greater:
When studying partially hyperbolic diffeomorphisms in dimension $3$, if one wants to use branching foliations (which so far have been the main tool to understand partially hyperbolic diffeomorphisms geometrically or topologically), then one has to use the existence result of Burago--Ivanov. Now that result comes with an orientability condition, thus forcing one to take a finite lift and finite power to ensure the existence of such foliations. Knowing uniqueness of such foliations would then allow to prove that they can project to the original manifold. Hence, one may hope to obtain geometrical consequences for the original map as well as for its lifts and powers.

\subsection{Another application of Theorem \ref{teo.graphtransform}}\label{ss.aplicationGT}

We will give another application of Theorem \ref{teo.graphtransform} that may be useful to simplify some of the arguments of the rest of the paper in some particularly relevant cases. We recall from \cite{BFFP_part2} the following notion for an $f$-invariant branching foliation. If $f \colon M \to M$ is a diffeomorphism preserving a branching foliation $\cF$, we say $\cF$ is $f$-minimal if every non empty, closed and $f$-invariant $\cF$-saturated set is all of $M$ (see \cite[Definition 3.23]{BFFP_part2}).

A direct consequence of Theorem \ref{teo.graphtransform} is the following: 

\begin{prop}\label{p.fminimalopenclosed}
Let $f \colon M \to M$ be a partially hyperbolic diffeomorphism admitting an $f$-invariant branching foliation $\cs_f$ tangent to $E^{cs}_f$ which is $f$-minimal. Then, for every $g\colon M \to M$ that can be connected to $f$ by a path of partially hyperbolic diffeomorphisms, the map $g$ admits a $g$-invariant branching foliation $\cs_g$ tangent to $E^{cs}_g$ which is $g$-minimal. 
\end{prop}

\begin{proof}
Fix a path $f_t\colon M \to M$ of partially hyperbolic diffeomorphisms such that $f_0=f$ and $f_1=g$. Let $A \subset [0,1]$ be the set of $t\in [0,1]$ such that $f_t$ verifies that it admits a $f_t$-invariant branching foliation which is $f$-minimal. The set $A$ contains $0$ and thus is non-empty. It is open thanks to Theorem \ref{teo.graphtransform} and \cite[Lemma B.1]{BFFP_part2}. 

To see that it is closed, fix an accumulation point $t$ of $A$. It follows that $f_t$ has a neighborhood $\cU$ where Theorem \ref{teo.graphtransform} holds. There is $t'\in A$ such that $f_{t'} \in \cU$, and thus we can apply Theorem \ref{teo.graphtransform} to deduce that $t \in A$. This shows $A$ is closed and thus $A=[0,1]$, in particular $g=f_1$ verifies the desired property. 
\end{proof}

\begin{remark}\label{r.transitiveAF}
A \emph{leaf space collapsed Anosov flow} $f$ will have its branching foliations $f$-minimal if and only if the corresponding Anosov flow is transitive, so the previous proposition implies that this will be the case for all partially hyperbolic diffeomorphisms in the connected component of $f$. In particular, if $f$ is a leaf space collapsed Anosov flow with respect to an Anosov flow which is not transitive, we deduce that $f$ cannot be in the same connected component as a partially hyperbolic diffeomorphism which is chain recurrent\footnote{We recall that being chain recurrent means that there is no proper open subset $U \subset M$ such that $f(\overline{U})\subset U$. This is implied for instance when $f$ is volume preserving, or transitive.}. Since transitive topological Anosov flows are orbit equivalent to true Anosov flows \cite{Shannon}, one can ignore the distinction between topological Anosov flows and smooth Anosov flows when working in the connected component of a partially hyperbolic diffeomorphism which is transitive or volume preserving, for instance. 
\end{remark}


\section{Some results about topological Anosov flows}\label{s.AF}

\subsection{Foliations of Anosov flows}
Let $\phi_t\colon M \to M$ be a topological Anosov flow on a closed 3-manifold $M$. We study here the $\phi_t$-invariant foliations saturated by orbits. We say that a foliation $\cF$ is $\phi_t$-\emph{saturated} if for every leaf $L \in \cF$ and $x \in L$ we have that $\phi_t(x) \in L$ for all $t \in \RR$. 

\begin{prop}\label{prop.uniquefol}
Let $\cF$ be a foliation by surfaces which is saturated by orbits of $\phi_t$ and such that $\cF^{ws}_\phi \neq \cF$. Then there is an attractor of $\phi_t$ on which $\cF=\cF^{wu}_\phi$. 
\end{prop}

\begin{proof}
We use the spectral decomposition of Anosov flows, see \cite[\S 5.3]{FH},
which also works for topological Anosov flows using essentially
the same arguments.
This implies that the set of points in $M$ whose $\omega$-limit set is contained in an attractor of the topological Anosov flow is open and dense. Note that the set of points on which $\cF^{ws}_\phi \neq \cF$ is open, therefore, there is an open set  $U$ of points whose $\omega$-limit set is contained in an attractor $A \subset M$ of the flow $\phi_t$ and such that $\cF^{ws}_\phi \neq \cF$. 

In particular, there is a point $x \in U$ which belongs to the stable manifold $\cF^{ws}_\phi$ of a periodic orbit $o$.

\begin{claim}
Let $S$ be a surface intersecting $\cF^{ws}_\phi(o)$ but not
contained in $\cF^{ws}_\phi(o)$,
where $o$ is a periodic point of $\phi$. Then, any point of $\cF^{wu}_\phi(o)$ is a limit of points in $\phi_t(S)$ as $t \to +\infty$. 
\end{claim}
\begin{proof}
The fact that the surface $S$ is not contained in $\cF^{ws}_\phi(o)$ means that (up to iterating forward) there is a small transversal $D$ to the flow through $o$ on which the trace of $S$ contains a curve not contained
in the trace of $\cF^{ws}_\phi(o)$ with $D$. The Poincar\'e first return map to $D$ is conjugate to a fixed saddle on $o \cap D$ and its forward iterates then make $S$ converge to the trace of $\cF^{wu}_\phi(o)$ by forward iteration. 
\end{proof}

Since the foliation is continuous, this implies that the leaf $\cF^{wu}_\phi(o)$ is contained in $\cF$. Since this leaf is dense in the attractor $A$ it follows that $\cF$ coincides with $\cF^{wu}_\phi$ in $A$ as announced. 
\end{proof}

A direct corollary is:

\begin{cor}\label{cor.uniquefol}
Let $\phi_t$ be a transitive topological Anosov flow, then, there are 
exactly two $\phi_t$-saturated foliations, which are $\cF^{ws}_\phi$ and $\cF^{wu}_\phi$. 
\end{cor}
\begin{proof}
Note that if $\phi_t$ is transitive, then $M$ is the unique attractor and repeller. 
\end{proof}

\begin{remark}\label{rem.nouniquefol}
If $\phi_t$ is not transitive, then Corollary \ref{cor.uniquefol} does not hold. Indeed, it is possible to construct several $\phi_t$ saturated foliations which coincide with the weak-stable and unstable foliations in subsets
of the non-wandering set, but 
which are different from both of these foliations 
in the wandering region. Indeed, to make a concrete example, consider the Franks--Williams \cite{FranksWilliams} Anosov flow $\phi_t\colon M \to M$ with an attractor $A$ and a repeller $R$ such that every orbit not in $A \cup R$ intersects a $C^1$ smooth
torus $T$ transverse to $\phi_t$ and choose a foliation $\cG$ of $T$ which is transverse\footnote{Notice that these foliations are indeed $C^1$, so one can take any foliation generated by a vector field between the two tangent spaces.} to both $\cF^{ws}_\phi \cap T$ and $\cF^{wu}_\phi \cap T$. If one considers the orbit of $\cG$ by $\phi_t$ one gets a $\phi_t$-saturated foliation on $M \smallsetminus (A \cup R)$ that can be completed to a $\phi_t$-saturated foliation by taking the foliation $\cF^{wu}_\phi$ in $A$ and $\cF^{ws}_{\phi}$ in $R$. Notice that one can construct uncountably many such foliations. Other examples can be constructed along the same lines using the zoo 
 of examples from \cite{BBY}. 
\end{remark}

Notice however that while non-transitive (topological) Anosov flows may have several flow saturated foliations, one cannot choose them to be pairwise transverse as we will show:

\begin{prop}\label{p.nontransitiveAF}
Let $\phi_t$ be a topological Anosov flow and let $\cF_1$ and $\cF_2$ be two topologically transverse $\phi_t$-saturated foliations. Then, up to relabeling, one has that $\cF_1= \cF^{ws}_\phi$ and $\cF_2=\cF^{wu}_\phi$. 
\end{prop}

\begin{proof}
The proof is very similar to that of Proposition \ref{prop.uniquefol}. In the transitive case the result follows directly from Corollary \ref{cor.uniquefol}, so we will assume that $\phi_t$ is non-transitive. 

Consider a point $x \in M$ such that its forward orbit accumulates in an attractor $A$ and its backward orbit in a repeller $R$. Then, we claim that in 
a neighborhood of
$x$ the foliations must coincide with $\cF^{ws}_\phi$ and $\cF^{wu}_\phi$. If this were not the case, then, say $\cF_1$ does not coincide with either of them in a neighborhood of $x$. Assume that $\cF_2$ does not coincide with $\cF^{ws}_\phi$ in a neighborhood of $x$ (if it does not coincide with $\cF^{wu}_\phi$ one makes a symmetric argument). Then, it follows by the argument in Proposition \ref{prop.uniquefol} that both $\cF_1$ and $\cF_2$ must coincide with $\cF^{wu}_\phi$ in $A$, so they cannot be transverse. 

Now, notice that points whose forward orbit accumulates in an attractor and backward orbit in a repeller are an open and dense subset of $M$ in~\cite[\S 5.3]{FH}  this is proven for hyperbolic flows, but the proof also applies to topological Anosov flows).
This completes the proof. 
\end{proof}

\subsection{Leaf space collapsed Anosov flows respect weak foliations}\label{ss.weakimpliesstrong}

Recall that, given a topological Anosov flow $\phi_t$, we denote by $\cO^{ws}_\phi$ and $\cO^{wu}_\phi$ the one-dimensional foliations of $\cO_\phi$ induced respectively by $\widetilde{\cF^{ws}_\phi}$ and $\widetilde{\cF^{wu}_\phi}$, the weak stable and weak unstable foliations of $\tilde \phi_t$ which are precisely the lifts of the foliations $\cF^{ws}_\phi$ and $\cF^{wu}_\phi$ to $\mt.$ We also denote by $\cO^{cs}_f$ and $\cO^{cu}_f$ the foliations induced in $\lc$ by the center stable and center unstable branching foliations.

We now show that the map $H$, in the definition of a leaf space collapsed Anosov flow (Definition \ref{defiw2}), respect the weak foliations:

\begin{prop}\label{p.wcaf2impliescaf2} 
If $f$ is a leaf space collapsed Anosov flow (Definition \ref{defiw2}) associated with Anosov flow $\phi_t$ and map $H\colon \cO_\phi \to \cL^c$, then, up to taking $\phi_{-t}$ instead, the map $H$ maps $\cO^{ws}_\phi$ to $\cO^{cs}_f$ and $\cO^{wu}_\phi$ to $\cO^{cu}_f$.  
\end{prop}

\begin{proof}  Assume that $f$ verifies Definition \ref{defiw2}. 
Here $H$ is the map $H\colon \cO_\phi \to \lc$ which is
$\pi_1(M)$-invariant.
Consider the preimage under $H$ of the center stable and center unstable foliations $\cO^{cs}_f$ and $\cO^{cu}_f$ in $\lc$.
These clearly project to foliations in $M$ by the $\pi_1(M)$-invariance,
and provide different foliations which are $\phi_t$ saturated.  

It follows from Proposition \ref{p.nontransitiveAF} that one must be $\cF^{ws}_\phi$ and the other $\cF^{wu}_\phi$. Thus, up to changing the flow $\phi_t$ to the flow $\eta_t$ defined by $\eta_t = \phi_{-t}$, the homeomorphism $H$ must map the foliations $\cO^{ws}_\phi$ and $\cO^{wu}_\phi$ to $\cO^{cs}_f$ and  $\cO^{cu}_f$ respectively. 
\end{proof}

\subsection{Expansive flows and topological Anosov flows}\label{ss.expansiveflow}

We first recall the notion of expansive flow:

\begin{defi}\label{def.expflow}
A non singular flow $\phi_t \colon M \to M$ is \emph{expansive} if for every $\eps>0$ there exists $\delta>0$ such that if $x,y \in M$ and $\sigma\colon \RR \to \RR$ is an increasing homeomorphism with $\sigma(0)=0$ such that $d(\phi_t(x), \phi_{\sigma(t)}(y)) \leq \delta$ for every $t \in \RR$ then $y = \phi_s(x)$ for some $|s|< \eps$. 
\end{defi}

\begin{remark}\label{rem.expunivcov}
The use of $\eps$ in the definition of expansivity is to account for the recurrence of the flow in $M$ itself and so that orbits that auto-accumulate also separate. If one knows that the flow $\phi_t\colon M \to M$ has properly embedded orbits\footnote{This cannot happen if $M$ is compact, but will sometimes be easy to know for instance, when lifting the flow to the universal cover.} then to establish expansivity it is enough to show that there is some $\delta$ such that different orbits cannot be Hausdorff distance less than $\delta$ from each other. In such cases we will call $\delta$ an \emph{expansivity constant} for $\phi_t$. 
We refer the reader to \cite{BowenWalters} for more on expansive flows. 
\end{remark}

The following is a direct consequence of \cite[Theorem 1.5]{IM} or \cite[Lemma 7]{Pat}:

\begin{teo}\label{teo.expimpliesTAF}
Let $\phi_t\colon M \to M$ be a flow tangent to a non vanishing vector field, such that $\phi_t$ is expansive and preserves a foliation. Then $\phi_t$ is a topological Anosov flow. 
\end{teo}

\begin{proof}
The results \cite[Theorem 1.5]{IM} or \cite[Lemma 7]{Pat} show that an expansive flow preserves transverse \emph{singular} $2$-dimensional foliations (one weak stable and one weak unstable) whose singularities consist of periodic orbits whose local structure is of a $p$-prong with $p\geq 3$.

We claim that prong singularities of singular stable and unstable
foliations are incompatible with preserving a foliation. 
Suppose that $\phi_t$ preserves a foliation $\cF$ and $\phi_t$ has
a singular $p$-prong orbit $\alpha$. 
Let $L$ be the leaf of $\cF$ through $\alpha$. While we cannot use Proposition \ref{prop.uniquefol} at this stage (since we do not yet know $\phi$ is a topological Anosov flow), the arguments in its proof apply equally well to show that, on each side
of $\alpha$, the leaf $L$ has to agree with a prong of a stable or an unstable leaf of $\phi_t$ .

Now, looking transversely, since there are at least six prongs of $\phi_t$ at $\alpha$ (at least $3$ stable and at least $3$ unstable), then
locally transversally one component of $M \smallsetminus L$ intersects at least
one stable and one unstable prong of $\alpha$. A nearby leaf $L'$ of
$\cF$ intersecting that component will intersect a stable and an
unstable prong of $\alpha$. Flowing forward along $\phi_t$ preserves $L'$,
brings it closer to $\alpha$ along the stable of $\alpha$ and farther
from $\alpha$ along the unstable of $\alpha$. This forces $L'$ to
topologically cross itself, which is impossible for a foliation.

This contradiction shows that the weak stable and unstable foliations of $\phi_t$ cannot have any singularities.
Therefore $\phi_t$ is a topological Anosov flow. 
\end{proof}

\begin{remark}\label{rem.expTAF}
Notice that the previous Theorem does not require the foliation invariant by the flow  to coincide with the weak stable and unstable foliations of the topological Anosov flow (cf.~Remark \ref{rem.nouniquefol}). If $\phi_t$ preserves two transverse foliations then these must coincide with the weak stable and unstable foliations of the topological Anosov flow thanks to Proposition \ref{p.nontransitiveAF}. 
\end{remark}

The next result follows quickly:

\begin{cor}\label{orbfoliations}
Suppose that $\varphi_t$ is orbit equivalent to a topological
Anosov flow. Then it has weak stable and unstable
foliations. 
In other words, any flow that is orbit equivalent to a topological Anosov flow satisfies conditions \ref{item_TAF_weak_foliations} to \ref{item_TAF_backwards_expansivity} of Definition \ref{def_TAF}, but may fail to satisfy condition \ref{item_TAF_tangentVF}.
\end{cor}

\begin{proof}
A topological Anosov flow is expansive, and any orbit equivalence preserves the property of being expansive. Therefore $\varphi_t$ has
weak stable and unstable possibly singular foliations.
In addition, $\varphi_t$ preserves a $2$-dimensional foliation
--- the image of the stable foliation under the orbit equivalence.
The arguments of the previous theorem show that this is incompatible
with singularities in the stable and unstable foliations of $\varphi_t$.
Hence the stable and unstable foliations of $\varphi_t$ do not
have singular orbits, proving the result.
\end{proof}

\subsection{Smoothness, Gromov hyperbolicity and the quasigeodesic property} 

Let $\phi_t$ be a topological Anosov flow on a closed $3$-manifold
$M$. 
The definition of a topological Anosov flow assumes very little 
regularity. In particular, the weak stable or unstable leaves may
only be $C^0$, and may not have the structure of
a path metric space. Up to an orbit equivalence 
we will show that one can assume
more regularity (for at least one of the foliations) and prove the following properties: The leaves of one of the weak foliations are 
Gromov hyperbolic and the flow lines are quasigeodesic within each leaf (these properties were proved in \cite[\S 5]{FenleyAnosov} for Anosov flows, but we need a different argument in our setting).

Note that, if $\phi_t$ is transitive then Shannon \cite{Shannon} proved that $\phi_t$
is orbit equivalent to an Anosov flow. Hence \cite[\S 5]{FenleyAnosov} implies the results in that case. 
What we prove here applies also to the non transitive case and is independent of Shannon's result.

In order to prove our results, we go through an intermediary flow, that is not quite a topological Anosov flow: Specifically
it lacks property \ref{item_TAF_tangentVF} of Definition \ref{def_TAF}, but satisfies all the other properties.

Recall that Corollary \ref{orbfoliations} shows that any flow that is orbit equivalent to a topological Anosov flow admits weak stable and unstable foliations (with proper asymptotic behaviors as described in conditions \ref{item_TAF_forward_asymptotic} and \ref{item_TAF_backwards_expansivity} of Definition \ref{def_TAF}).
Thus we can start our first step in building an orbit equivalence from $\phi_t$ to one with smooth weak stable leaves:

\begin{lemma}\label{smoothstr}
Let $\phi_t$ be a topological Anosov flow.
Then $\phi_t$ is orbit equivalent to a flow $\varphi_t$, 
such that the weak stable leaves of $\varphi_t$ are smooth
($C^{\infty}$)
surfaces.
\end{lemma}

\begin{proof}
This directly follows from a result of Calegari in \cite{Cal-smoothing}:
He shows that $\cF^{ws}_{\phi}$ is isotopic to a foliation
with smooth leaves. 
The isotopy produces a homeomorphism from $M$ to itself
which is isotopic to the identity.
This homeomorphism induces an orbit equivalence from $\phi_t$
to another flow $\varphi_t$ so that the weak stable leaves of
$\varphi_t$ are smooth.
As remarked above $\varphi_t$ has weak stable
and unstable foliations.
\end{proof}

We stress that the flow $\varphi_t$ has smooth weak
stable leaves, but a priori only $C^0$ weak unstable leaves.
Of course one can do the same procedure with the weak unstable foliation
instead of the weak stable foliation.

\begin{remark}
One could use Calegari's work in a different way: It is also shown in \cite{Cal-smoothing} that one can change the \emph{smooth structure} of $M$, to make the weak stable leaves of $\phi_t$ smooth for that new smooth structure.
This is not the way we choose to proceed and hence emphasize that the smooth structure of $M$ is fixed throughout this article.
\end{remark}

The issue we have to deal with here is that, after the orbit equivalence, 
the orbits of $\varphi_t$ may not be tangent to a vector field, in fact, a priori,
they may not even be rectifiable.
So technically $\varphi_t$ may not be a topological Anosov flow, but it still has the weak stable and unstable
foliations, which are denoted the same way.

\begin{prop}\label{prop.topAnosovQG} 
Suppose that $\varphi_t$ is a flow in $M$ which verify conditions \ref{item_TAF_weak_foliations} to \ref{item_TAF_backwards_expansivity} of Definition \ref{def_TAF}.
Suppose that the weak stable leaves of $\varphi_t$ are smooth.
Then they are Gromov hyperbolic and the orbits of $\widetilde{\varphi_t}$ are uniform quasigeodesics in each leaf of $\widetilde{\cF^{ws}_\varphi}$.

Moreover, writing $S^1(L)$ for the Gromov-boundary of $L$, there exists a unique $\xi\in S^1(L)$ such that every forward ray of $\widetilde{\varphi_t}$ ends at $\xi$, and for any $\zeta\in S^1(L)$, $\zeta\neq \xi$, there exists a unique orbit of $\widetilde{\varphi_t}$ that has $\zeta$ and $\xi$ as its endpoints.
\end{prop} 

 The reason we are working with flows preserving smooth weak stable or unstable foliations is so that we can consider the induced path metric on the leaves. The result stays true for any orbit equivalent flows if one considers the metric in the leaves induced by the homeomorphism realizing the orbit equivalence.

\subsubsection{Proof of Proposition \ref{prop.topAnosovQG}}

In order to prove Proposition \ref{prop.topAnosovQG}, we will ``discretize'' the flow in a given leaf $L$ obtaining an oriented tree that we will show is quasi-isometric to $L$ and such that any orbit of $\wt\varphi_t$ on $L$ is quasi-isometric to a unique oriented path in that tree. One important thing is that all the constants of quasi-isometry are uniform, i.e., depend only on the flow, not on the leaf $L$.

First, as $\wt\varphi_t$ is expansive in $\wt M$, it admits an expansivity constant, say $\rho>0$.
Moreover, we can pick $\rho$ so that if two forward (resp.~backwards) rays in $\wt M$ are at distance less than $\rho$ apart then they are in the same weak stable (resp.~unstable) leaf. (This is because the ``$\eps$-stable manifold'' of a point $x$ is a regular neighborhood of $x$ in its stable leaf, see e.g.~\cite{IM}.)

We pick a constant $\eps>0$ much smaller than $\rho$. 
We will choose two finite coverings, one contained in the other.
First choose a finite covering $\{C_i\}$ of $M$ by flow boxes of the following form: $C_i = \bigcup_{|t|<\eps} \varphi_t(E_i)$ where $E_i$ are pairwise disjoint transversals to the flow of diameter less than $\eps/10$.
In the same way choose a finite covering $\{ B_i \}$ of $M$
of the form $B_i  = \bigcup_{|t| < 
\eps} \varphi_t(D_i)$ where $D_i$ are also pairwise disjoint transversals
of diameter less than $\eps$. We choose $D_i$ so that $E_i \subset D_i$ and the distance from any point in $\partial E_i$ to $\partial D_i$ is at least $\eps/4$.

Now, given a leaf $L \in \wt\cF^{ws}$, we call $\{\wt B^L_\alpha\}$ the connected components of the intersection of lifts of the $B_i$ with $L$. Note that $\{\wt B^L_\alpha\}$ is an $\eps$-dense, locally finite, open cover of $L$.
For each $\alpha$, we call $I_\alpha \subset \wt B^L_\alpha$ the transversal obtained as the intersection of the corresponding lift of $D_i$ with $L$. For this choice of $I_\alpha$, we have:
\begin{itemize}
\item Each $\wt B^L_\alpha$ is of the form $\bigcup_{|t|< \eps} \wt\varphi_t(I_\alpha)$,
\item the distance between two different $I_\alpha$ is bounded below and above by some constants $a,b>0$. (Note that these constants only depends on our choice of open cover of $M$ by the flow boxes, not on the leaf $L$). 
\end{itemize}

Similarly we define the sets $\wt C^L_\alpha$.
A standard consequence of expansivity is the following (see \cite{Pat,IM} for similar results).

\begin{lema}\label{l.expansiveness} There exists $T>0$ (independent of $L$) such that, for every $I_\alpha$ there is some $I_\beta$ such that, for every $x \in I_\alpha$, there exists $t_x \in [T^{-1},T]$ with $\wt\varphi_{t_x}(x) \in I_\beta$.  
\end{lema}

\begin{proof}
Now, assuming the statement is not true, we can find sequences of points $x_n,v_n \in I_{\alpha_n}$ such that, for times $t_n \to \infty$, we have that $z_n := \wt \varphi_{t_n}(x_n)$ 
is in $I_{\beta_n} \cap \wt C^L_{\beta_n}$, 
but $\wt \varphi_{[0,\infty)}(v_n)$ does not intersect
$I_{\beta_n}$. So the distance in $L$ from this flow ray to 
$I_{\beta_n} \cap \wt C^L_{\beta_n}$
is $\geq \eps/4$ (this is where the two covers are used).
Therefore we can then choose $w_n \in I_{\beta_n}$ with 
$d(z_n,w_n) \in [\eps/4,\rho]$ so that $\wt \varphi_{(-\infty,0]}(w_n)$
intersects $I_{\alpha_n}$ in a point, denoted by $y_n$.
(Recall that $\rho$ is the expansivity constant.)
Let $s_n > 0$ be such that $\wt \varphi_{s_n}(y_n) = w_n$.
Notice that $s_n \to \infty$ as $n \to \infty$.
Using connectedness of the tranversal, we have that the distance between the pieces of orbits induces a distance in the transversal. We can then, up to changing $x_n,y_n$ pick $t_n$ and $s_n$ so that for every $t \in [0,t_n]$ there is $s \in [0,s_n]$ such that $\wt \varphi_t(x_n)$ and $\wt \varphi_s(y_n)$ are in the same flow box and at distance less than $\rho$. 
In addition $d(z_n, w_n) \geq \eps/4$.
Without loss of generality we can assume that $\rho$ is smaller than the size of foliation boxes of the weak stable foliation.

Up to the action under deck transformations and taking subsequences, we can pick $\gamma_n \in \pi_1(M)$ so that $\gamma_n \cdot z_n$ and $\gamma_n \cdot w_n$  converge to  points $z_\infty$ and $w_\infty$.
The points $\gamma_n \cdot z_n$ and $\gamma_n \cdot w_n$ lie in the same leaf in a foliation box thus $z_\infty$ and $w_\infty$ lie on the same stable leaf, but in distinct orbits (because the transversal distance of $\gamma_n \cdot z_n$ and $\gamma_n \cdot w_n$ is larger than $\eps/4$).

Now, by construction the backwards rays of $z_\infty$ and $w_\infty$ are at distance less than $\rho$ apart, so they must be on the same unstable leaf. But they are already on the same stable leaf, thus they should be on the same orbit, a contradiction.
\end{proof}

Now, we choose a point (for instance, the mid point) $x_\alpha \in I_\alpha$ for every $\alpha$. We define a directed graph $\cT$ whose vertices are the points $x_\alpha$ and there is a directed edge $x_\alpha \to x_\beta$ if: 

\begin{itemize}
\item For every $x \in I_\alpha$ there is some $t_x \in [T^{-1},T]$ such that $\wt\varphi_{t_x}(x) \in I_\beta$. 
\item $I_\beta$ is the first transversal verifying the previous property, that is, if $I_\gamma$ also verifies the previous property, then, for every $x \in I_\alpha$ we have that the orbit of $\wt\varphi_t(x)$ intersects $I_\beta$ before intersecting $I_\gamma$. 
\end{itemize}

We give $\cT$ a distance given by assigning length $1$ to each edge and defining the distance between two vertices to be the minimal length path from one to the other. Note that the distance does not take into account the direction of the arrows. 

\begin{remark}
Notice that the construction of $\cT$ depends on $L$. However, the constants we will find depend only on the initial construction (i.e., our choice of the finite covering by flow boxes) and thus will be independent on the leaf $L$. 
\end{remark}

Remark that, by construction of $\cT$ together with Lemma \ref{l.expansiveness}, every point $x_\alpha$ has a unique successor in $\cT$.  Notice also that, by definition of a topological Anosov flow, given any two points $y,z\in L$, we can find $t_1,t_2$ such that $\wt\varphi_{t_1}(y)$ and $\wt\varphi_{t_2}(z)$ are close. Thus the orbits of $y$ and $z$ will eventually hit the same transversal. Hence, $\cT$ is connected. Moreover, since orbits cannot intersect a transversal twice (by Poincar\'e--Bendixson's theorem) it follows that: 

\begin{lema} 
The graph $\cT$ is a tree, in particular, it is Gromov hyperbolic. 
\end{lema}

Now, we can map $\cT$ inside $L$: The vertices are already points in $L$, and we identify the oriented edge between $x_\alpha$ and its successor $x_\beta$ with the segment of orbit of $\wt\varphi_t$ from $x_\alpha$ to $I_\beta$. Thus we obtain a map $i\colon \cT \to L$. (Note that this map has discontinuities at the vertices, but that is not an issue.)

Since the distance in $L$ between two consecutive vertices $x_\alpha$ and $x_\beta$ is bounded above by some uniform constant $b>0$ it follows that
\[
b d_{\cT}(x_\alpha, x_\beta) \geq d_L(x_\alpha, x_\beta), 
\]
because any path in $\cT$ induces a path in $L$ whose length is at most $b$ times the number of vertices it passes through. 

The key point is that there is an inverse estimate:

\begin{lema}\label{lem.graph_qi_to_L} 
The graph $\cT$ is quasi-isometric to $L$ (more precisely, the map $\imath\colon \cT \to L$ is a quasi-isometry). Moreover, the quasi-isometry constant is independent on $L$. 
\end{lema}

\begin{proof}
The vertices of $\cT$ are $\eps$-dense in $L$ for the distance on $\wt M$. Now, since we built $\cT$ by taking points inside small flow boxes, the distance in $L$, $d_L$, restricted to each box is bi-Lipschitz to the distance in $\wt M$, with a factor as close to $1$ as we want (the factor only depends on the size of the box, not on $L$). So we can assume that the vertices of $\cT$ are $2\eps$-dense in $L$ for the distance $d_L$.

Thus the map $\imath\colon \cT \to L$ is coarsely surjective. Hence, we only have left to show that there exists $b'>0$, independent of $L$, such that 
\[
d_{\cT}(x_\alpha, x_\beta) \leq b' d_L(x_\alpha, x_\beta). 
\]

To do this let us consider a path $\eta$ in $L$ joining $x_\alpha$ to $x_\beta$ so that $\mathrm{length}(\eta)= d_L (x_\alpha, x_\beta):=D$. 

Since the image of $\cT$ is $2\eps$ dense in $L$, we can cut $\eta$ in a sequence of points $p_1, \ldots, p_k$ where each $p_i$ is a point in the image of (the inclusion of) $\cT$ and such that 
\begin{itemize}
\item $p_1=x_\alpha, p_k=x_\beta$, 
\item $d_L(p_i, p_{i+1}) < 4\eps$, 
\item $D/\eps \geq k$. 
\end{itemize}

Note that these constants are independent of $L$. Now, it is then enough to prove the following claim.

\begin{claim}
There exists $C>0$ such that given two points in $\cT$ which are at distance less than $1$ in $L$ their distance in $\cT$ is less than $C$. 
\end{claim}

\begin{proof}
Note that if the transversals are at distance less than $2$ in $L$ then it takes a uniform amount of time for both to be contained in the same transversal (cf. Lemma \ref{l.expansiveness}). It follows that the directed path forward from each will reach the same transversal in a uniform amount of time, equivalently, there exists a uniform $C$ (independent on $L$) so that the distance between the transversals in $\cT$ is bounded by $C$.  
\end{proof}

Given the previous claim, it then follows that $D \geq k \eps \geq \frac{\eps}{C} d_{\cT}(p_1,p_k)$ which gives the desired estimate. 
\end{proof}

This result readily implies that leaves of the weak stable foliation in the universal cover are uniformly Gromov hyperbolic\footnote{Note that this fact could also have been obtained from Candel's Uniformization Theorem, since there are no transverse invariant measures for $\cF^{ws}$.}. We will now study the relationship between orbits of the flow inside a leaf and directed paths in $\cT$.

\begin{lema}\label{lem.uniform_bound_directed_path}
There is some uniform constant $c>0$ such that, for every directed path in $\cT$, there is $x \in L$ such that its forward ray remains at Hausdorff distance less than $c$ from the directed path.

Moreover, for every orbit of $\wt\varphi_t$ in $L$ there is a unique bi-infinite directed path in $\cT$ that remains at bounded distance from the orbit. 
\end{lema}

\begin{proof}
First consider a directed path $\{p_1, p_2, \ldots, p_n, \ldots \}$ in $\cT$ and choose any point $x$ in the transversal corresponding to $p_1$. From the construction of the graph it follows that the forward orbit of $x$ intersects the transversals corresponding to each $p_i$ and the time it takes to go from one to the other is bounded between $T^{-1}$ and $T$ for some $T>0$. This allows to construct some $c$ with the desired property since the transversals have bounded length and the flow is continuous. 

Now, if $\{\ldots, p_{-n}, \ldots, p_0, p_1, \ldots p_n, \ldots \}$ is a bi-infinite directed path in $\cT$ then we can pick points $x_n$ in the transversal corresponding to $p_{-n}$ and consider the point $y_n$ in the transversal corresponding to $p_0$ so that $y_n$ is the intersection of this transversal with the orbit of $x_n$. Since the transversal is compact, we can assume that $y_n \to y_{\infty}$ and it follows that the orbit of $y_{\infty}$ is bounded distance from the directed path as desired. 

Finally, given $x \in L$ we can take points $x_n = \wt\varphi_{t_n}(x)$ and define for each $x_n$ a directed path $\{p_0^n, p_1^n, \ldots, p_k^n, \ldots \}$ defined by taking $p_0^n$ to correspond to the first transversal intersected by the forward orbit of $x_n$ and then taking the directed path associated to it (recall that each vertex has a unique successor). It follows that this directed path is distance less than $c$ from the forward orbit of $x_n$. It also follows that the sequence of directed paths stabilizes, by this, we mean that if $n>m$ it follows that $p_1^m, \ldots, p_k^m, \ldots$ is contained in the directed path induced by $x_n$ (note that the first point may not be in the path since it could be that it does not have predecessors). This way, we can define a bi-infinite directed path in $\cT$ which is unique and well defined by construction. Moreover, any other directed path in $\cT$ has to be infinite distance away, showing the stronger form of uniqueness.   
\end{proof}

We can finally finish the proof of Proposition \ref{prop.topAnosovQG}.

We have already shown that the leaves of $\cF^{ws}$ are Gromov-hyperbolic. Now, given an orbit of $\wt \varphi_t$ in $L$, Lemmas \ref{lem.graph_qi_to_L} and \ref{lem.uniform_bound_directed_path} imply that it is a uniform quasi-geodesic of $d_L$. Moreover, each orbits in $L$ converge to the same point, that we denote by $\xi$ on $S^1(L)$, the Gromov-boundary of $L$.

Furthermore, if $\zeta\in S^1(L)$ is the backward endpoint of a directed path in $\cT$, then there exists a unique orbit of $\wt \varphi_t$ which admits $\zeta$ as its backwards endpoints (thanks to Lemma \ref{lem.uniform_bound_directed_path}).

Since orbits of $\wt\varphi_t$ are quasigeodesics in $L$, the usual Morse Lemma states that they are a bounded distance away from a unique geodesic (for $d_L$). In fact, this bound is uniform:
\begin{remark}[Uniform bound]\label{rem.uniformbound} There is a positive constant $k>0$ such that 
for any flow line, $\ell$, of $\widetilde \varphi_t$
in a leaf $L$ of $\widetilde{\cF^{ws}_\varphi}$, if 
$\hat \ell$ is the geodesic in $L$ with the same endpoints as $\ell$ in $S^1(L)$, then the Hausdorff distance, $d_H$, in $L$, satisfies $d_H(\ell,\hat \ell)<k$.

Moreover, for any $x, y\in \ell$, denoting by $\hat \ell_{x,y}$ the geodesic segment in $L$
from $x$ to $y$ and by $\ell_{x,y}$ the compact segment in $\ell$ from
$x$ to $y$, then $d_H(\ell_{x,y},\hat \ell_{x,y}) < k$. See \cite[Theorem III.H.1.7]{BH}. 
\end{remark}

This uniform bound implies that the backwards endpoints of orbits of $\wt\varphi_t$ in $L$ move continuously has one moves transversally to the orbits of $\wt\varphi_t$ in $L$. Hence, for any $\zeta\in S^1(L)$, $\zeta\neq \xi$, there exists an orbit (which has to be unique) with $\zeta$ has its backwards endpoints.

This ends the proof of Proposition \ref{prop.topAnosovQG}.

\subsubsection{Building the topological Anosov flow with smooth leaves}
We now can build a new flow, that will be a true topological Anosov flow and orbit equivalent to $\varphi_t$. That is, we prove the following.

\begin{proposition} \label{prop.smoothness}
Let $\phi^0_t$ be a topological Anosov flow.
Then $\phi^0_t$ is orbit equivalent to a topological Anosov flow
$\phi^2_t$ in $M$ such that the weak stable leaves $\cF^{ws}_{\phi^2}$
are smooth. In particular the leaves of $\cF^{ws}_{\phi^2}$
are Gromov hyperbolic, and the flow lines of $\phi^2_t$ are
uniform quasigeodesics in the corresponding leaves of $\cF^{ws}_{\phi_2}$.
Finally the strong stable leaves of $\phi^2_t$ are also smooth.
\end{proposition}

\begin{proof}
By Lemma \ref{smoothstr} the flow $\phi^0_t$ is orbit equivalent
to a flow $\varphi_t$ in $M$ (that satisfies conditions \ref{item_TAF_weak_foliations} to \ref{item_TAF_backwards_expansivity} of Definition \ref{def_TAF}) for which the weak stable leaves
are smooth. Moreover, by Proposition \ref{prop.topAnosovQG} the leaves of $\cF^{ws}_{\varphi}$
are Gromov hyperbolic and the flow lines of $\varphi_t$ 
are uniform quasigeodesics in their leaves of $\cF^{ws}_{\varphi}$.

The only missing property is that $\varphi_t$ is tangent to a non vanishing vector field.

Now we will build another flow, $\phi^2_t$, that will be orbit equivalent to $\varphi_t$ (hence also to $\phi^0_t$), with the same weak stable leaves as $\varphi_t$, and tangent to a non vanishing vector field.

To simplify notation, in this proof, we let $\fol = 
\cF^{ws}_{\varphi}$.
Since the leaves of $\fol$ are smooth and Gromov
hyperbolic, Candel proved that there is a leafwise tensor metric
making each leaf a hyperbolic surface. More precisely, 
let $g$ denote the induced Riemannian metric in $T \fol$, the
tangent space of $\fol$. Then, there is a function $\eta(x)$ in $M$
so that $\eta(x) g$ induces a hyperbolic metric in each leaf of $\fol$.

When we restrict to a leaf $L$ of $\fol$ the function $\eta$ is just a uniformizing conformal function. In particular it is smooth when restricted to each leaf: See for example \cite[page 502]{Candel} or \cite[page 253]{Calegari} where there are explicit formulas for $\eta(x)$. These formulas show that $\eta(x)$ is smooth when restricted to leaves, but vary only continuously transversally to the leaves.

\vskip .1in
Now we construct the flow $\phi^2_t$.
Let $g_1 = \eta(x) g$ be Candel's leaf-wise smooth hyperbolic metric. Fix a leaf $L$ of $\fn$. By Proposition \ref{prop.topAnosovQG}, the flow lines of $\varphi_t$ in $L$ are uniform quasigeodesics in $L$ (for the metric $g$ and hence for the metric $g_1$), they all share the same forward ideal point, call it $z_L\in S^1(L)$ and their backwards ideal points are in one-to-one correspondence with $S^1(L)\smallsetminus\{z_L\}$.

Now we define $\wt \phi^2_t$ as the flow (with unit speed in the $g_1$-metric) such that its orbits on each leaf $L$ are $g_1$-geodesics between $y$ and $z_L$ for any $y\in S^1(L)\smallsetminus\{z_L\}$.

Since the flow lines of $\varphi_t$
are uniform quasigeodesics in leaves of $\fn$ (independent
of the leaf of $\fn$), it follows
that the forward ideal points of flow lines vary continuously
in $\mt$ (not just in leaves of $\fn$).

In particular, $\wt \phi^2_t$ is a continuous flow on $\mt$. As $g_1$ and $\fn$ are invariant under the action of deck transformations, the flow $\wt \phi^2_t$ projects to a flow $ \phi^2_t$ on $M$.

Now, thanks again to Proposition \ref{prop.topAnosovQG}, there is a one-to-one correspondence between the flow lines of $\widetilde \varphi_t$ in a given leaf $L\in \fn$ and the flow lines of $\wt\phi^2_t$. Thus we get a map $\tau$ from the orbit space of $\widetilde \varphi_t$
to the orbit space of $\widetilde \phi^2_t$. 

The map $\tau$ is a homeomorphism (thanks to the uniform quasigeodesic behavior of the orbits of $\varphi_t$), and clearly $\pi_1(M)$-equivariant. 
Under these conditions, it follows that $\varphi_t$ is orbit equivalent to $\phi^2_t$.
Work on this was first done by Haefliger \cite{Haefliger} who
showed that there is a homotopy equivalence sending orbits
of $\varphi_t$ 
to orbits of $\phi^2_t$ induced by $\tau$.
This was upgraded to a homeomorphism sending
orbits to orbits by Ghys \cite{Ghys} and Barbot
\cite{Barbot2} using averaging techniques.

By construction, the orbits of $\phi^2_t$ are smooth (because $g_1$ is a smooth Riemannian metric on each leaf), with constant speed, and their velocity varies continuously in $M$ (more precisely, the velocity vary smoothly on each leaf and continuously transversally, because, as mentioned previously, the forward ideal points vary continuously transversally) thus $\phi^2_t$ is tangent to a non vanishing continuous vector field. Since it is immediate that $\phi^2_t$ is expansive, (and preserves a foliation by construction), Theorem \ref{teo.expimpliesTAF} implies that $\phi_2^t$ is a topological Anosov flow.

Now we can just check all the conditions claimed in the statement of the proposition: $\phi_2^t$ is a topological Anosov flow that is orbit equivalent to $\phi^0_t$, its weak stable foliation is $\fn$, so has smooth leaves, the leaves are Gromov-hyperbolic and the flow lines are uniform quasigeodesics in each leaf. Finally, the strong stable leaves corresponds to horocycles of the hyperbolic metric $g_1$ in each leaf, hence are smooth.
\end{proof}

Along the way we proved the following result:

\begin{cor} Let $\phi^0_t$ be a topological Anosov flow
in a smooth $3$-manifold $M$. Then $\phi^0_t$ is orbit equivalent
to a topological Anosov flow $\phi^2_t$ such that the leaves of
the weak stable foliation
$\cF^{ws}_{\phi^2}$ are smooth, and that satisfy the following additional properties:
\begin{enumerate}[label = \arabic*)]
 \item There is a 
leafwise Riemannian metric $g_1$ in $T \cF^{ws}_{\phi^2}$,
conformal with the
induced metric from $M$, such that leaves of $\cF^{ws}_{\phi^2}$
are hyperbolic surfaces with the metric $g_1$;
\item  The flow lines
of $\phi^2_t$ are geodesics in the respective leaves
of $\cF^{ws}_{\phi^2}$ with the hyperbolic metrics given
by $g_1$; 
\item The strong stable leaves are projection of 
horocycles in the respective
leaves of $\cF^{ws}_{\phi^2}$ with the hyperbolic metrics
given by $g_1$.
\end{enumerate}
\end{cor}

As pointed out before, the metric $g_1$ varies only continuously in $M$.

\subsection{Discretized Anosov flows revisited}\label{ss.DAF}

Here we show that discretized Anosov flows defined in \cite{BFFP_part1} fit well with all the definitions of collapsed Anosov flows. 

\begin{defi}\label{def.DAF}
A partially hyperbolic diffeomorphism $f\colon M \to M$ is a \emph{discretized Anosov flow} if there exists a topological Anosov flow $\phi_t\colon M \to M$ and a continuous function $\tau\colon M \to \RR$ such that $f(x)=\phi_{\tau(x)}(x)$ for every $x \in M$. 
\end{defi}

In \cite[Appendix G]{BFFP_part1} we asked for the function $\tau$ to be positive, but this is unnecessary: 

\begin{prop}\label{p.daf1}
If $f$ is a discretized Anosov flow, then the function $\tau\colon M \to \RR$ cannot vanish. 
\end{prop}

\begin{proof}
In \cite[Proposition G.2]{BFFP_part1}, we proved that if $f$ is a discretized Anosov flow, then $f$ must be dynamically coherent. The argument presented in \cite[Proposition G.2]{BFFP_part1} assumed that the map $\tau$ was positive, but we will show below how they can easily be modified in order not to use this assumption. Once we know that $f$ is dynamically coherent, we will directly deduce that $\tau$ cannot vanish.

First, we prove, that the vector field $X$ tangent to the flow $\phi_t$ needs to be in the center bundle of $f$:

If $X$ is tangent to $E^s$ at some point, then as in \cite[Proposition G.2]{BFFP_part1}, 
this implies that there is an interval along the flow direction totally tangent to $E^s$. Since $E^s$ is uniquely integrable, then this interval
in the flow direction is contained in a stable leaf.
Now, the function $\tau$ is bounded, so the length
along a flow line from $x$ to $f(x)$ is bounded.
Iterating negatively by $f$ increases the stable length exponentially,
so first we can assume that there is $x$ in $M$ such that the 
interval in the center leaf from $x$ to $f(x)$ is contained in
a stable leaf. Then again applying negative powers of $f$,
produces a contradiction to $\tau$ being bounded.
It follows that $X$ is never tangent to $E^s$. The symmetric argument implies that it is never tangent to $E^u$ either.
Then, as in the proof of \cite[Proposition G.2]{BFFP_part1} one
proves, that $X$ is always tangent to $E^c$ and that $f$ is 
dynamically coherent.

Since $f$ is dynamically coherent, we can consider the good lift $\wt f$ obtained via the lift of the natural homotopy along the flow lines of the lifted topological Anosov flow. This lift $\wt f$ cannot have fixed points (see, e.g., \cite[Corollary 3.11]{BW} or \cite[Lemma 3.13]{BFFP_part1}), thus $\tau$ cannot vanish.
\end{proof}

The following relates the notion of discretized Anosov flows and collapsed Anosov flows. 

\begin{prop}\label{p.daf2}
If $f$ is a discretized Anosov flow, then it is a
strong collapsed Anosov flow
with $h$ being a homeomorphism and $\beta$ being a trivial self orbit equivalence. Conversely, if $f$ verifies Definition \ref{defi1} with $\beta$ a trivial self orbit equivalence, then $f$ is a discretized Anosov flow. 
\end{prop}

\begin{proof}
To prove the direct assertion let us just take $h$ to be the identity. In \cite[Proposition G.2]{BFFP_part1} it is shown that the center stable and center unstable foliations of $f$ correspond to the weak stable and weak unstable foliations of the topological Anosov flow (in particular, these weak foliations whose leaves are a priori only $C^0$ have $C^1$-leaves). 
Then $\beta(x) = \phi_{\tau(x)}(x)$ which proves the result.

For the converse statement, notice first that since $f$ verifies
Definition \ref{defi1} then the 
image under $h$ of leaves of $\cF^{ws}_{\phi}$
provides a branching foliation tangent to $E^{cs}$, and likewise
for $\cF^{wu}_{\phi}$. Finally the image of any flow line is a curve
tangent to $E^c$, providing a branching center foliation.
Consider a good lift $\ft$ corresponding to lifting $\beta$ to a
homotopy along the flow lines, and using a lift of $h$ lifting
a homotopy to the identity. The equation 
$f \circ h(x) = h \circ \beta(x)$ then implies
that $\ft$ preserves every center 
leaf in $\mt$.
Again the argument of \cite[Corollary 3.11]{BW} (or \cite[Lemma 3.13]{BFFP_part1}) implies that the lift to the universal cover cannot have fixed points. Moreover, when lifted to the universal cover, one has the same situation as in the \emph{doubly invariant case}, studied in \cite[\S 7.2]{BFFP_part2}. 
Doubly invariant means that $\ft$ fixes every leaf of both
the center stable and the center unstable foliations
lifted to $\mt$.
This proves that the branching foliations are actual
foliations,  proving dynamical coherence of $f$. Now this immediately implies that $f$ is a discretized Anosov flow  (see also \cite[\S 6]{BFFP_part1}). 
\end{proof}

With what we proved so far, we can show that when the self orbit equivalence is trivial, then the two notions of collapsed Anosov flows (Definition \ref{defi1w}) and strong collapsed Anosov flows (Definition \ref{defi1}) coincide.

\begin{prop}
 Let $f$ be a collapsed Anosov flow such that the associated self orbit equivalence $\beta$ is trivial, then $f$ is a strong collapsed Anosov flows.
\end{prop}

\begin{proof}
 Since $\beta$ is trivial, the image of the flow line foliation by $h$ is 
an $f$-invariant branching foliation,
whose leaves are tangent to $E^c$. This is a center branching foliation
in this case.
Moreover, a good lift $\ft$ leaves invariant every center leaf.

As in \cite[Lemma 7.3]{BFFP_part2} we know that $\ft$ moves point a bounded distance in each center. An argument similar to \cite[Lemma 7.4]{BFFP_part2} allows to show that center curves are disjoint or coincide. This implies that $f$ verifies Definition \ref{def.DAF}. 
\end{proof}

In view of the above, we can even wonder whether it would be sufficient for a definition of collapsed Anosov flow to only require the partially hyperbolic diffeomorphism to be semi-conjugate to a self orbit equivalence:
\begin{quest}
Let $f\colon M \to M$ be a partially hyperbolic diffeomorphism such that there exists a (topological) Anosov flow $\phi_t \colon M\to M$, a self orbit equivalence $\beta\colon M \to M$ and a map $h\colon M \to M$ continuous and homotopic to the identity such that $f \circ h = h \circ \beta$. Is $f$ a collapsed Anosov flow? 
\end{quest}


\section{Quasigeodesic behavior inside foliations}\label{s.QG}

In this section we study some properties of one dimensional foliations which subfoliate a two dimensional foliation with Gromov hyperbolic leaves. Then, we restrict to the partially hyperbolic setting and show Theorem \ref{teo.QG} that is the key step to obtain Theorem \ref{teo.main5} which will be shown in the next section. 

\subsection{One dimensional foliations inside two dimensional foliations} 
Let $\cF$ be a foliation on a 3-manifold. In this section, we will assume that there is a metric on $M$ that makes every leaf of $\cF$ negatively curved. Then we can even assume the metric on each leaf is constant curvature $-1$ by Candel's uniformization theorem. This assumption is verified whenever the foliation does not have a transverse invariant measure of zero Euler characteristic (by Candel's uniformization theorem, see \cite[\S I.12.6]{CanCon} or \cite[\S 8]{Calegari} for a precise statement). 

Consider a one dimensional foliation $\cG$ which subfoliates $\cF$ (i.e., leaves of $\cF$ are saturated by leaves of $\cG$). 
We assume that the leaves of $\cG$ have continuous parametrizations. In our case $\cG$ will be one of two types: Either a foliation with rectifiable leaves (in which case the parametrization can be chosen to be arclength), or a foliation by orbits of a continuous flow (where, while the orbits may not be rectifiable, the flow itself gives a continuous parametrization).

\begin{defi}\label{d.uQGfol}
The foliation $\cG$ is a \emph{uniform quasigeodesic} subfoliation of $\cF$ if every leaf $\ell \in \cG$ is a quasigeodesic in its corresponding leaf $L \in \cF$ with uniform constants. 
\end{defi}

Let us make precise what we mean by uniform constants in the above definition:

Call $\wF$ and $\wG$ the lifts of $\cF$ and $\cG$ respectively to
the universal cover. 
Let $\ell$ be a leaf of $\wG$ in a leaf
$L$ of $\wF$. Then $\ell$ is a $C$-\emph{quasigeodesic} if there is a constant $C>1$ such that for every $x,y\in \ell$ we have that $d_\ell(x,y)< C d_L(x,y) + C$. Here $d_L$ denotes 
the distance in $L$ given by a path metric in $L$ 
and $d_\ell$ denotes the distance in $\ell$ induced by the change
in the parameter in the respective leaf.

In Definition \ref{d.uQGfol}, we require that there exists a constant $C>1$ such that, for any $L \in  \wF$ and any $\ell \in \wG$, the leaf $\ell$ is a $C$-quasigeodesic.
Note that, by compactness of $M$, this definition does not depend on the choice of metric (see Proposition \ref{prop-GH}). 

\begin{remark}
 One can in fact prove, by adapting the proof of \cite[Lemma 10.20]{Calegari}, that if $\cG$ subfoliates $\cF$ with quasigeodesic leaves, then it is automatically a \emph{uniform} quasigeodesic foliation.
\end{remark}

By Remark \ref{rem.uniformbound} any $\ell$ leaf of $\wG$ in a leaf $L$
of $\wF$ is a uniformly bounded Hausdorff distance in $L$ from
a geodesic $\hat \ell$ in $L$.

Let $\wG|_L$ be the foliation $\wG$ when restricted to $L$.
As is the case for foliations by geodesics
\cite[Construction 5.5.4]{Cal-Promoting}, one can show 
that foliations by quasigeodesics of a hyperbolic plane are quite restrictive: 

\begin{prop}\label{prop.fandichotomy}
Given $L \in \wF$ we have that the leaf space $\cL_{\wG,L}=L/_{\wG}$ of the foliation $\wG$ is homeomorphic to $\RR$ and either there is a point $p \in S^1(L)$ such that every leaf of $\wG|_{L}$ has $p$ as one of its endpoints or there are exactly two points in $S^1(L)$ invariant under every isometry of $L$ preserving the foliation $\wG|_{L}$. 
If $r_n$ is a sequence of rays\footnote{A \emph{ray} of a leaf $\ell$ of $\wG$ in a leaf $L \in \Ft$ is the closure of a connected component of $\ell \smallsetminus \{x\}$ for some $x \in \ell$. Each ray has a well defined ideal point $r_\infty \in S^1(L)$ which coincides with the corresponding ideal point of $\ell$.}  in leaves of $\wG$ converging
to a ray $r$, then the ideal points of $r_n$ in $S^1(L)$ converge to the ideal point of $r$.
\end{prop}

\begin{proof}
We first show that the leaf space $L/_{\wG}$ is Hausdorff. Suppose this is
not true and there are $\ell_n$ leaves in $\widetilde \cG$ converging
to two distinct leaves $\ell, \ell'$ of $\wG$.

 Let $x, y$ be
points in $\ell, \ell'$ respectively. Then there are $x_n, y_n$ in $\ell_n$
converging to $x, y$ respectively. Hence $d_L(x_n,y_n)$ is bounded.
We claim that $d_{\ell_n}(x_n,y_n)$ goes to infinity. Otherwise up
to subsequence we would have $d_{\ell_n}(x_n,y_n) \leq a_0$. But using the local product
structure of foliations, we would deduce that $y$ is in $\ell$, a contradiction.

Hence, $d_{\ell_n}(x_n,y_n)$ must converge to infinity. However, since $d_L(x_n,y_n)$ is bounded, this contradicts the uniform quasigeodesic behavior.
Therefore $L/_{\wG}$ is Hausdorff, and hence
homeomorphic to $\RR$.

We now show that the ideal points of rays of leaves of $\wG$
in $S^1(L)$ vary continuously. Let $x_n$ be a sequence in $L$ 
converging to $x$ in $L$, and $\ell_n, \ell$ the leaves of $\wG$
through $x_n$ and $x$ respectively. 

Let $r_n$ be rays in $\ell_n$ starting in $x_n$ converging to a
ray $r$ in $\ell$ starting in $x$. Let $r_{n,\infty}, r_{\infty}$ be the ideal points
of $r_n, r$ respectively. We want to show that $r_{n,\infty}$ converges to
$r_{\infty}$.

Suppose this is not the case, then, up to taking a subsequence, we can assume that $r_{n,\infty}$ converges to $s_{\infty} \neq r_{\infty}$.
Since $r_n$ converges to $r$ and $r$ has ideal point $r_{\infty}$,
then for any $n$ large enough, there exist points $u_n \in r_n$ very close to $r_{\infty}$ in the compactification $L \cup S^1(L)$. As $r_{n,\infty}$ converges to $s_{\infty}$, there are also points $v_n\in r_n$ very close to $s_{\infty}$ in $L \cup S^1(L)$. Hence, we can choose such points such that $u_n \rightarrow r_{\infty}$, $v_n \rightarrow s_{\infty}$.
The compact segments $I_n$ of $l_n$ from $u_n$ to $v_n$ are at most $k$ distant
in $L$ from the geodesic segment connecting them. Since $u_n$ is
very close to $r_{\infty}$ and $v_n$ is very close to $s_{\infty}$, then all of these geodesic segments intersect a fixed compact set of $L$. Hence, up to taking a further subsequence, the segments $I_n$ must converge to a leaf $\ell'$ of $\wG$. But this leaf is not $\ell$, contradicting that $L/_{\wG}$ is Hausdorff. 
Thus we proved that ideal points of ray vary continuously.

Identify the leaf space $L/_{\wG}$ with $\RR$, with parametrization
$\ell_t, t \in \RR$ 
 and consider a sequence $\ell_{t_n}$, $t_n  \to +\infty$.
Notice first that the endpoints $(\ell_{t_n})^\pm$ determine a weakly nested sequence of intervals in $S^1(L)$ which needs to shrink as $n\to \infty$.
If the geodesics $g_n$ with ideal points $(\ell_{t_n})^\pm$
do not shrink to a point, then $g_n$ limits
to a geodesic $g$ in $L$. But recall that
the $\ell_{t_n}$ are at distance at most $k$ from $g_n$ (cf.~Remark \ref{rem.uniformbound}): If the endpoints are trapped by the endpoints of $g$, then the leaves are
trapped by a neighborhood of size $k$ of $g$ and cannot escape in $L$,
contradiction.

Hence we get two points of $S^1(L)$ one for $t \to +\infty$ and one
for $t \to -\infty$.
If these two points coincide then we get that every leaf of $\wG$ must have that limit point as a limit point. Otherwise, we get the other condition. 
Obviously any isometry of $L$ leaving the foliation invariant
has to preserve this pair of ideal points.
\end{proof}

We now define some structures related to what follows from the previous proposition. 

\begin{defi}\label{defi-qgfan}
We say that a leaf $L \in \wF$ is a \emph{weak quasigeodesic fan} for the foliation $\wG$ if there is a point $p \in S^1(L)$ such that every leaf of $\wG|_L$ has $p$ as one of its limit points. In this case we call $p$ the \emph{funnel point} of $\wG|_L$.  The leaf $L \in \wF$ is a \emph{quasigeodesic fan} if moreover given a point $q \in S^1(L)\smallsetminus \{p\}$ there is a unique leaf of $\wG|_{L}$ whose endpoints are $p$ and $q$.  
We say that a leaf $A$ of $\cF$ is a quasigeodesic fan or a weak
quasigeodesic fan if a lift $L$ of it to $\widetilde M$ is
a quasigeodesic fan or a weak quasigeodesic fan respectively.
\end{defi}

\begin{remark}\label{rem.topAnosovQG}
Proposition \ref{prop.topAnosovQG} gives that the orbits of a topological Anosov flow make up a quasigeodesic fan in each weak (un)stable leaf.
\end{remark}

\begin{lemma} \label{rem.sweep}
If $\cG$ is a uniform quasigeodesic 
subfoliation of $\cF$ then, for every leaf $L \in \wF$ we have that there are at most two points in $S^1(L)$ which are not endpoints of any of the curves in $\wG|_{L}$. 

Moreover, if a leaf $L \in \wF$ is a weak quasigeodesic fan, then every point of $S^1(L)$ is the endpoint of a leaf of $\cG$. See Figure \ref{fig.QGfol}.
\end{lemma}

\begin{figure}[ht]
\begin{center}
\includegraphics[scale=0.62]{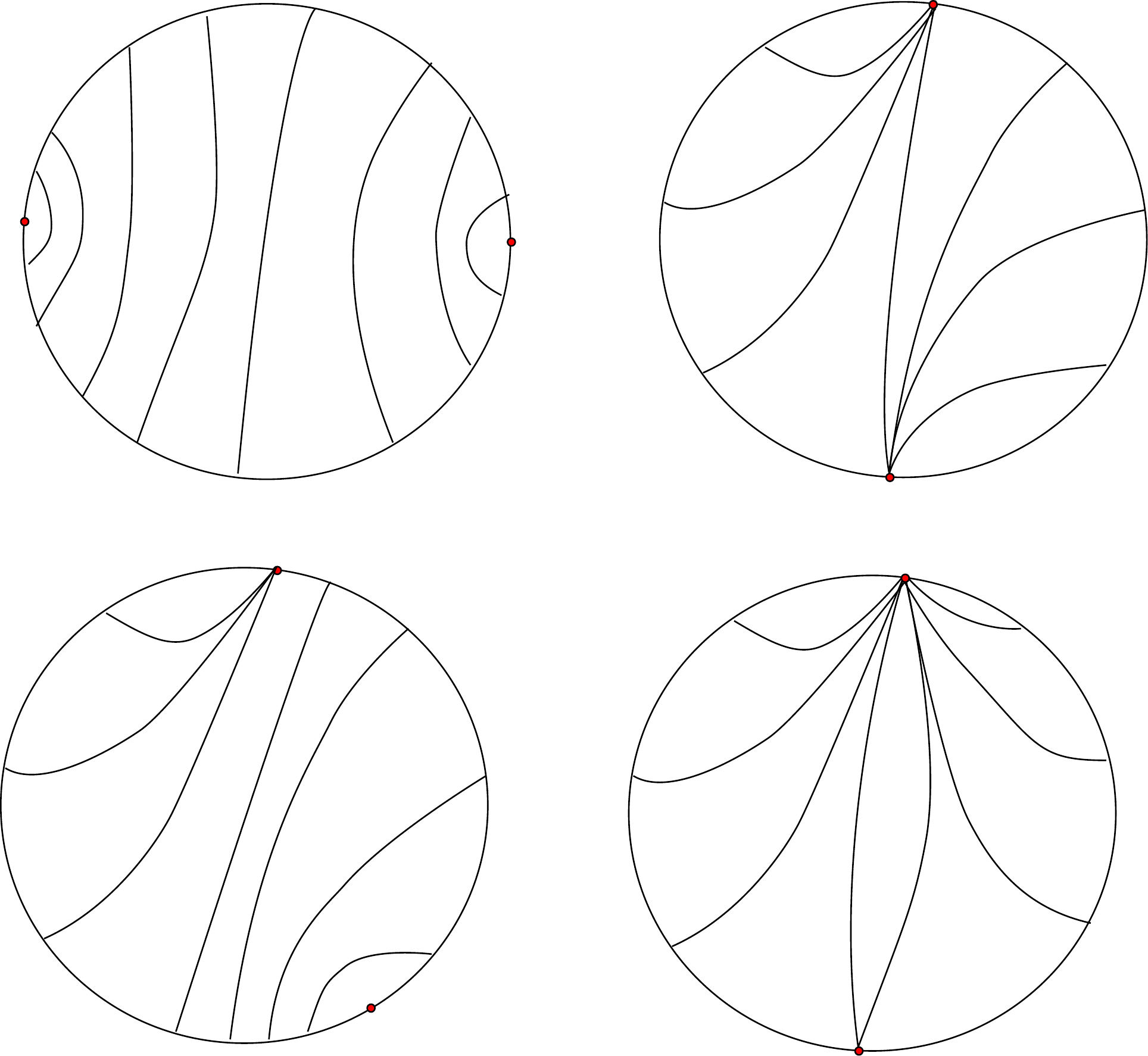}
\end{center}
\vspace{-0.3cm}
\caption{Some quasigeodesic foliations of the disk which are not quasigeodesic fans (the bottom right one is a weak quasigeodesic fan).}\label{fig.QGfol}
\end{figure}

\begin{proof}
This fact was essentially done in the proof of Proposition \ref{prop.fandichotomy}: Recall that we can consider $\{l_t\}_{t\in \R}$ a parametrization of the leaf space $L/_{\wt\cG }$, and we proved that the endpoints of $(l_t)$ must be weakly nested and converge to points $x, y$ in $S^1(L)$ as $t\to \pm\infty$.

Moreover, since the leaves are uniform quasigeodesic, the map that sends $t$ to the endpoints of $l_t$ is continuous. So, if $x=y$, we are in the case of a weak quasigeodesic fan.

If $x \neq y$, then, calling $I$ and $J$ the complementary
intervals of $x, y$ in $S^1(L)$, we must have that each $l_t$ has one endpoint in $I\cup\{x,y\}$ and the other in $J\cup\{x,y\}$ (because the endpoints are weakly nested). Thus, using continuity of endpoints again, we see that $x$ and $y$ are the only two points that may not be the endpoints of a leaf.
\end{proof}

From Proposition \ref{prop.fandichotomy} we deduce: 

\begin{cor}\label{cor.ciclic}
Suppose that a foliation $\cF$ 
by hyperbolic leaves of a $3$-manifold admits a 
uniformly quasigeodesic one dimensional
subfoliation $\cG$. Then every leaf of $\cF$ has cyclic fundamental group (thus a leaf is either a plane, an annulus or a M\"{o}bius band).
\end{cor}

\begin{proof}
Deck transformations of $M$ act as isometries, so if a deck transformation fixes some leaf $L \in \wF$ then it is an isometry which preserves $\wG|_{L}$. Note that it must be a hyperbolic isometry since there is a uniform injectivity radius of leaves of $\cF$. Hyperbolic isometries that fix a given point or a pair of points at infinity commute, Proposition \ref{prop.fandichotomy} thus ensure that the $\pi_1$ of $L$ is at most cyclic. 
\end{proof}

From now on, we assume $\cG$ to be a uniform quasigeodesic subfoliation of $\cF$.
Our next goal will be to show that there are weak quasigeodesic fan 
leaves of $\cF$, and that the collection of such leaves forms a sublamination
of $\cF$. 
This is the analogue of \cite[Lemma 5.3.6]{Cal-Promoting} (see also \cite[Lemma 5.5.5]{Cal-Promoting}) which is done for the case of 
geodesic subfoliations in leaves of $\cF$, when $M$ is atoroidal
(see \cite[Lemma 5.5.5]{Cal-Promoting}).

We first need a technical result, which will be used repeatedly, that produces some weak quasigeodesic fan leaves from certain configurations.

\begin{lemma} \label{lem.existweakfan}
Suppose that $x_n$ is a sequence in $\mt$ such that there
are disks $D_n$ in the leaves $L_n \in \Ft$ centered
at $x_n$ with radius converging to infinity and satisfying
the following: There are disks $E_n$ in $L_n$ of bounded diameter, such that the distance in $L_n$ from $E_n$ to $D_n$ goes to infinity and such that any leaf of $\wG|_{L_n}$ intersecting
$D_n$ also intersects $E_n$. Then, given a sequence of deck transformations $\gamma_{n} \in \pi_1(M)$ such that for some subsequence $n_j \to \infty$,
we have $\gamma_{n_j} x_{n_j} \to x$, it follows
that the leaf through $x$ is a weak quasigeodesic fan. 
\end{lemma}

\begin{proof} We can assume without loss of generality that $x_n \to x$ up to
changing by deck transformations and taking a subsequence.

Call $a_1>0$ an upper bound of the diameters of the $E_n$.
Assume by contradiction that the leaf $L$ of $\wF$ through $x$ is not a weak quasigeodesic
fan. Then there is a pair of leaves 
$\ell, \ell'$  of $\wG|_L$  which do not share any ideal point in $S^1(L)$.
These curves are at most $k$ distant in $L$ from the
corresponding geodesics because of the uniform bound,
see Remark \ref{rem.uniformbound}.

Since $\ell, \ell'$ do not share ideal points, then
two properties follow: 
\begin{enumerate}
\item there are points $y, y'$ in $\ell, \ell'$ respectively
attaining the minimum distance $a_0$ between points in $\ell, \ell'$,

\item there is $t > 0$ such that if $z\in \ell$, and $z' \in \ell'$
then if both $d_L(z,y)$ and $d_L(z',y')$ are larger than $t$ then 
$$d_L(z,\ell'), d_L(z',\ell)  > 10 (a_0 + a_1 + k + 4).$$
\end{enumerate}

The points $x_n$ converge to $x$ in $L$. The distance from $x$ 
to $y, y'$ in $L$ is finite, so up to changing $x_n$ in $L_n$ by
a bounded distance (and choosing a subdisk of $D_n$ with
radius still going to infinity with $n$) we may assume that $x_n$
converges to $y$.
Since the foliations $\wG|_{L_n}$ converge to $\wG_L$
we see that the foliations in the
disk of radius $100(t + a_0 +1)$ (recall that $t, a_0$ 
are fixed) around $x_n$ converge
to the foliation in a disk of radius $100(t + a_0 +1)$ around $y$ in $L$.
In $L$ on both sides of $y, y'$ the leaves $\ell, \ell'$ 
spread more than $10(a_0 + a_1 +k + 4)$ from each other. 
So we see this in some of the leaves of $\wG|_{L_n}$ as well.
This is within fixed distance $t$. But the property of $E_n$
means that these leaves come back within $a_1$ of each
after a distance larger than $t$ if $n$ is big enough.

We now use that these curves are uniform quasigeodesics.
Recall their properties:
\begin{enumerate} \item they are within $a_0 + 1$ from each
other near $y, y'$; \item they are within $a_1  +1$ from each other when they both intersect $E_n$. 
\end{enumerate}
This implies that the geodesic segments connecting these pairs
of points are within $(a_0 + a_1 + 2)$ throughout.
By the uniform quasigeodesic property the segments in leaves
of $\wG|_{L_n}$ are within
$(a_0 + a_1 + 2) + 2k$ from each other throughout.
But we proved that they have points where the curves are 
more than $10(a_0 + a_1 + 2 +k +2)$ apart from each other
in between.

This contradiction shows that the limit leaf is a weak quasigeodesic
fan and finishes the proof of the lemma.
\end{proof}

\begin{prop}\label{prop.closedfans}
The set of leaves $L \in \wF$ which are weak quasigeodesic fans for $\wG$ 
is non empty, closed and $\pi_1(M)$-invariant. 
Hence it induces a sublamination of $\cF$ in $M$.
\end{prop}

\begin{proof}
The $\pi_1(M)$ invariance property is obvious.

We first show that the set of weak quasigeodesic fan leaves
is non empty.
Let $L$ be a leaf of $\wF$. 
We will construct sets $D_n, E_n$
in $L$ satisfying the hypothesis of the previous lemma.
Let $\ell_1$ be a leaf of $\wG|_L$. Let $I$ be the closed 
interval of leaves of $\wG|_L$ all of which share both
endpoints with $\ell_1$. This could be a degenerate interval,
that is, $\ell_1$ itself. Let $\ell$ be a boundary leaf of
$I$. Now consider a leaf $\ell'$ sufficiently near $\ell$
intersecting a transversal $\tau$ from $x$ in $\ell$ to $x'$ in
$\ell'$. In addition assume that $\ell'$ is not in $I$.
Let $E$ be a disk containing $\tau$. Let $E_n = E$
of fixed diameter.

Since $\ell'$ is not in $I$,  it has at least one ideal point
$z'$ which is not an ideal point of $\ell$. Let $r'$
be the ray of $\ell'$ starting in $x'$ and with ideal point
$z'$. Let $r$ be the ray of $\ell$ starting in $x$ and
going in the same direction as $r'$.

Recall that the leaf space of $\wG|_L$ is the reals $\RR$.
Let $V$ be the complementary region of $r \cup \tau \cup r'$ which
only contains rays of leaves of $\wG|_L$ which intersect $\tau$.
Hence every leaf of $\wG|_L$ intersecting $V$ also intersects
the fixed set $E$.

Finally since $r'$ and $r$ do not have the same ideal points and are 
quasigeodesics we can find $D_n$ a set of diameter greater than $n$ contained
in $V$ and such that the distance in $L$ from $D_n$ to $E$ is greater than $n$.
Taking $L_n = L$ for any $n$, we can apply the previous lemma and get leaves of $\cF$
which are weak quasigeodesic fans. This proves the first
assertion of the proposition.

\vskip .1in
Now we prove that the set of leaves that are not weak quasigeodesic fan is open.

Let $L$ be a leaf that is not a weak quasigeodesic fan. Then there are 
leaves $\ell, \ell'$ which do not share any ideal points. 
As in the
previous lemma: 
\begin{enumerate}
\item there are points $y \in \ell, y' \in \ell'$
realizing the minimum distance $a_0$ between them, 
\item for any $a_2 > 0$,
there is $t > 0$ such that if distance along $\ell$ from $y$ to
$z$ is greater than $ t$ then $d_L(z,\ell') > a_2$ and vice versa for points
in $\ell'$.
\end{enumerate}
Hence once $a_1, a_2, t$ are fixed we obtain for any leaf $F$
sufficiently near $L$ that we have leaves $\ell_F, \ell'_F$ 
in $\wG|_F$ satisfying this property in $F$. Specifically this does not
hold for 
\emph{every} point $z$ in $\ell_F$ with distance in 
$\ell_F$ from a fixed point is $> t$, but for some points.
We choose $a_2 > a_0 + 100k$.
Fix this pair of leaves $\ell_F, \ell'_F$.

Now suppose that $F$ is a weak quasigeodesic fan. We will obtain
a contradiction.
For any two leaves $\zeta, \zeta'$ in $\wG|_F$
they have a common endpoint in some direction. If they share
both endpoints then they are within $2k$ of each other.
Since, by choice, $a_2 > 2k$, the pair
$\ell_F, \ell'_F$ cannot be $\zeta, \zeta'$.

Next, suppose that $\zeta, \zeta'$ share
one but not both ideal points.
The corresponding geodesics $\hat \zeta, \hat \zeta'$ of $F$ to $\zeta, \zeta'$
 are asymptotic, but
disjoint. By negative curvature in the direction where they
are asymptotic, the distance in $F$ between points $y_t$ in $\hat \zeta$
converging to the common ideal point  and
$\hat \zeta'$ is always decreasing, modulo a bounded error,
and converging to zero. Since $\zeta, \zeta'$ are $k$ distant
from $\hat \zeta, \hat \zeta'$ respectively, then the distance in $F$ between points
$y_t$ in $\zeta$ converging to the common ideal point and $\zeta'$
in $F$ is roughly decreasing modulo an error of at most $4k$.
But the leaves $\ell, \ell'$ have points very distant ( at least $ a_2 > a_0 + 100k$) from the other leaf,
then follow along to points roughly $a_0$ distant, then again
some points very distant ($> a_2$). Therefore $\ell, \ell'$ cannot be 
$\zeta, \zeta'$. 

This contradicts the existence of leaves $\ell, \ell'$ in $F$,
which have to be some pair $\zeta, \zeta'$.
This contradiction finishes the proof that the set of non
weak quasigeodesic fans is open.
This finishes the proof of the proposition.
\end{proof}

\subsection{Branching foliations}\label{branfol}
Now consider two transverse branching foliations $\cs$ and $\cu$ in $M$ (the names are given for obvious reasons) which determine a one dimensional branching foliation 
$\cW^c$ by intersection. We consider $\wcs$, $\wcu$ the lifts to the universal cover.
We assume that $\cs, \cu$ are transversely orientable.
Let $\fes, \feu$ be the approximating foliations from $\cs, \cu$
given by Theorem \ref{teo-openbranch} for some small $\eps>0$. 
Let $\wfes,\wfeu$ be their lifts to $\mt$.
These determine a foliation $\wfec$ which subfoliates both. 
The foliation $\wfec$ also projects to a one dimensional
foliation $\fec$ which subfoliates both $\fes, \feu$.
Since $\cF^{cs}_\eps, \cF^{cu}_\eps$ have $C^1$ smooth
leaves, then leaves of $\fec$ are $C^1$.

We can then copy the notions above to define: 

\begin{defi}\label{defi-branching}
We say that $\cW^c$ is by \emph{uniform quasigeodesics} in $\cs$ and $\cu$ if leaves of $\fes,\feu$ are Gromov hyperbolic and $\wfec$ is a foliation by uniform quasigeodesics in both $\wfes$ and $\wfeu$ as in Definition \ref{d.uQGfol}. Similarly we can define as in Definition \ref{defi-qgfan} leaves of $\wcs$ or $\wcu$ (or leaves of $\cs, \cu$)
 being \emph{(weak)-quasigeodesic fans} by the identification between the leaves of $\wfes$ and $\wcs$ (resp. $\wfeu$ and $\wcu$). 
\end{defi} 

The leaves of $\wcs, \wcu$ have their intrinsic geometry
induced from the Riemannian geometry of $\mt$.
These leaves are quasi-isometric to the corresponding leaves
of $\widetilde \cF^{cs}_\eps, \widetilde \cF^{cu}_\eps$.
In particular the notions above are independent of $\eps$.

\subsection{The partially hyperbolic setting} 

Here we state the main result of this section: 

\begin{teo}\label{teo.QG}
Let $f\colon M \to M$ be a partially hyperbolic diffeomorphism preserving branching foliations $\cs$ and $\cu$  such that the foliation $\cW^c$ is by uniform quasigeodesics in each leaf of
$\cW^{cs}$ (cf.~Definition \ref{defi-branching}). Then, the center leaves of $\cW^c$ form a quasigeodesic fan in each $\cs$ and $\cu$ leaf. 
\end{teo}

We will split the proof of Theorem \ref{teo.QG} into two parts.
Proposition \ref{prop.WQGF} shows that every leaf of $\wcs$ must be a weak quasigeodesic fan and Proposition \ref{prop.Unique} shows that different centers in 
a leaf of $\wcs$ do not have the same pair of points at infinity. Both proposition follow the same strategy, first we construct an invariant lamination of \emph{good leaves} where the property we want holds, and then we apply Proposition \ref{prop.laminations} below to show that every leaf is a good leaf. 

There are some very important situations where our proof can be  much simplified: If $f$ is transitive, or more generally if $\cW^{cs}$ (or $\cW^{cu}$) is $f$-minimal
then Proposition \ref{prop.laminations} is immediate as there are no $f$-invariant sublaminations of $\cW^{cs}$. There would also be some simplifications if $M$ was assumed to be hyperbolic or Seifert fibered.

Recall that we proved in Proposition \ref{prop.closedfans} that the
set $P$ of leaves of $\wcs$ which are weak quasigeodesic fans is non empty, $\pi_1(M)$-invariant, $\ft$-invariant, and closed.
We did that in the (non branching) foliations setting, but
subsection \ref{branfol} implies the result in the branching foliations
setting as well.
Let $\Lambda$ be the projection of the leaves in $P$ to $M$.
This is a closed, $f$-invariant set of $\cs$ leaves, that is
a sublamination of $\cs$.
We want to show that these are all the leaves of $\cs$.

\subsection{A result about invariant laminations}\label{ss.general_result_laminations} 

The following result is stated for $\cs$, but obviously works for $\cu$ as well. It is a statement that will be useful to show that all leaves are weak quasi-geodesic fans in the next subsection (but can be skipped if one is working in the $f$-minimal case).

\begin{prop}\label{prop.laminations}
Let $f\colon M \to M$ be a partially hyperbolic diffeomorphism preserving a branching foliation $\cs$ tangent to $E^{cs}$ and $\ft$ a lift to $\mt$. 
Suppose the foliation $\cW^c$ is by uniform quasigeodesics
in each leaf of $\cs$. 
Let $\cP \subset \cL^{cs}$ be a closed $\pi_1(M)$- and $\ft$-invariant subset of the leaf space of $\wcs$ containing $P$.
Then, for every connected component $\cN$ of $\cL^{cs} \smallsetminus \cP$ there is $\gamma \in \pi_1(M)$, a leaf $L \in \cN$ and a leaf $L' \in \cN$ such that $\gamma L = L$ but $\gamma L' \neq L'$. 
\end{prop}

We can and will assume by taking finite covers and iterates that all the bundles are orientable and $f$ preserves orientation, this does not result inany loss of generality.

Recall that $P$ is the set of leaves of $\wcs$ which are weak quasi-geodesic fans and $\Lambda \subset M$ the projection onto $M$ consisting of the projection of leaves of $P$. The proposition can be restated as saying that every closed $f$-invariant lamination containing the weak quasi-geodesic fan leaves cannot have trivial holonomy in the complement.  Recall that every leaf projects into a plane or an annulus, so this says that in a complementary region of such a lamination there must be some annulus and not all leaves can be homotopic to it. We note that the proposition holds in manifolds with virtually solvable fundamental group thanks to the classification of such partially hyperbolic diffeomorphisms \cite{HPsurvey}, so we shall assume throughout this subsection that $\pi_1(M)$ is not virtually solvable. 

We first obtain a property of annular leaves of $\cs$. 

\begin{lemma}\label{lem.annuluslimit}
Let $A$ be an annular leaf of $\cs$.
Then $A$ only limits on points in $\Lambda$, that is, the closure of $A$ is contained in $\Lambda \cup A$. 
\end{lemma}

\begin{proof}
 Let $A$ be an annular leaf of $\cs$ in the complement of $\Lambda$. Let $\gamma$ be a generator of $\pi_1(A)$. Let $L$ be a lift of $A$ to $\mt$ invariant by $\gamma$.

By Proposition \ref{prop.fandichotomy} (and Lemma \ref{rem.sweep}) since $L$ is not a weak quasi-geodesic fan, it follows that points which are not fixed by $\gamma$ in the boundary of $L$ are endpoints of some rays of the center foliation. Consider then $x_n \in A$ accumulating in some point $y \in M \smallsetminus A$. Fix a fundamental domain of $\gamma$ in $L$ so that its closure in $L \cup \partial L$ is far from the fixed points of $\gamma$ and take $z_n$ points in this fundamental domain so that they project to $x_n$. We can choose a subsequence so that $z_n$ converges to some point $\xi$ in the boundary of $L$ and far from the fixed points of $\gamma$. 

Fix a orientation for centers in $L$. Let $\ell$ be a center leaf in $L$ with one endpoint in $\xi$ (say, oriented in the backward direction). Assume first that the other endpoint (in the forward direction) of $\ell$ is a fixed point of $\gamma$. Then, it follows that every center between $\gamma^{-1}(\ell)$ and $\gamma(\ell)$ has the same endpoint in the forward direction. Since $z_n \to \xi$, there exists $a_n\to \infty$ so that disks around $z_n$ of radius $a_n$ are contained between $\gamma^{-1}(\ell)$ and $\gamma(\ell)$. Centers through points between $\gamma^{-1}(\ell)$ and $\gamma(\ell)$ must remain close in the forward direction thus we can apply Lemma \ref{lem.existweakfan} to see that the points $x_n$ converge to $\Lambda$. 

Now, assume that the endpoint of $\ell$ (in the forward direction) is not a fixed point of $\gamma$. Since the center leaf space in $L$ is Hausdorff we get that $\ell$ separates $\gamma(\ell)$ from $\gamma^{-1}(\ell)$. This implies that all three curves separate the fixed points of $\gamma$. If one chooses a closed geodesic $\alpha \subset A$ and lifts it to $\tilde \alpha$ in $L$ this gives a geodesic joining the fixed points of $\gamma$. It follows that every center curve between $\gamma(\ell)$ and $\gamma^{-1}(\ell)$ must intersect $\tilde \alpha$ in  a fundamental domain of $\gamma$ which is a bounded length interval. Then, since there are arbitrarily large disks around $z_n$ between $\gamma(\ell)$ and $\gamma^{-1}(\ell)$ one can again apply Lemma \ref{lem.existweakfan} to see that the points $x_n$ converge to $\Lambda$. This concludes.

\end{proof}

We now begin the proof of Proposition \ref{prop.laminations}. 

The proof will be by contradiction, assuming that every leaf of $\cN$  is invariant by the same deck transformations. Recall that since every leaf has cyclic fundamental group, then either all are invariant under a fixed cyclic subgroup of $\pi_1(M)$ or they all project to plane leaves. 
 
Consider the approximating foliation $\fes$, with lift $\wfes$ and
leaf space $\lecs$, which is canonically equivariantly homeomorphic
to $\lcs$. Let $\cP_{\eps}$ be the closed set corresponding to $\cP$ and
$\cN_{\eps}$ the open set corresponding to $\cN$.
Then the set of leaves in $\cN_{\eps}$ projects to an open 
$\fes$ foliated set $U$ in $M$.
Let $\Lambda_\eps$ be the lamination of $\fes$ corresponding to
$\Lambda$.
Let also $\Lambda^*$ be the sublamination of $\cs$ corresponding
to $\cP$ and $\Lambda^*_\eps$ the similar sublamination of $\fes$.

The contradiction assumption means that every leaf in $\cN_{\eps}$ is invariant
by the same deck transformations. In particular the foliation
$\fes$ restricted to $U$ has trivial holonomy (the germ
of holonomy of every closed curve in a leaf of $\fes$ in $U$ is
trivial).

The strategy of the proof of Proposition \ref{prop.laminations} is
to control the topology of $\cN$ to be able to get a contradiction with a \emph{volume vs length} argument. In order to control the topology of $\cN$ it is useful to show that the leaf space of $\fes|_U$ is Hausdorff in the universal cover since then we will obtain a free action on the line and reduce the possible deck transformations that fix $\cN$ getting the desired control on its topology (abelian fundamental group). 

Since $\fes$ in $U$ has trivial holonomy we would like to
apply Sacksteder's theorem \cite[Theorem I.9.2.1]{CanCon}. 
But Sacksteder's theorem requires that the foliation
is $C^2$ to avoid exceptional leaves. However, we do not need the full power of
Sacksteder's theorem, what we want is to prove that the leaf space of
$\wfes$ in  each connected component of $\widetilde U$ is homeomorphic to $\RR$ (we say in this case that $\fes|_U$ is $\RR$-covered). 
For that we will instead use \cite[Theorem 3.1]{Imanishi}. 
This result basically says that if there is trivial holonomy, then one can
extend holonomy along paths with domains open intervals to
holonomy with domains being the closed intervals. 
In particular when lifted to the universal cover
the foliation has leaf space homeomorphic to $\RR$. 

We will use several times the \emph{octopus decomposition}: Let $U$ be the open set in $M$ which is the projection of the leaves
in $\cN_\eps$.
The completion $\hat U$ of $U$ has an octopus decomposition (cf.~\cite[Proposition I.5.2.14]{CanCon}) with a thin part $T$ and a core $K$ such that $K$ is compact and $\hat U$ retracts onto $K$  (this last fact is true because
leaves have fundamental group at most cyclic). In particular, we know that $\pi_1(\hat U)$
is finitely generated.

\begin{lemma}\label{lem.rcov}
The foliation $\fes |_U$ is $\RR$-covered.
\end{lemma}

\begin{proof}
We will fix a connected component of $U$ which we will still call $U$ for simplicity. 

We consider first the case that $\fes |_U$ has an annular
leaf $A$.
Lemma \ref{lem.annuluslimit} shows that $A$ limits only
on $\Lambda_\eps$. 
We consider
the octopus decomposition $\hat U = K \cup T$.
The annular leaf $A$ intersects $K$ in a compact subannulus
$A_0$ and intersects $T$ into two half open annuli
$A_1, A_2$. Since $\fes$ has trivial holonomy in $U$,
then $A_0$ has an open neighborhood in $K$ which is product
foliated. The same happens for $A_1, A_2$ in $T$ so
it follows that $A$ has an open neighborhood in $U$
which is product foliated.

Let $Z$ be the union of the leaves in $U$ which are isotopic
to $A$ in $U$. We just proved that this set is open. 
In addition the intersection of any leaf in $Z$ with $K$ is
a compact annulus isotopic to $A_0$ in $K$. Now use that
the set of
compact leaves of $\fes |_K$ in $K$ is closed, and, in addition, that outside
of $K$ we have products. So now it follows that $Z$ is also
closed in $U$. As $U$ is connected, it follows that
$Z = U$. 

In other words we proved that $\fes |_U$ is
a product foliation, so the leaf space of $\fes |_U$ is
homeomorphic to $\RR$. The same happens in $\widetilde U$.
This proves the lemma in this case.

\vskip .1in
From now on suppose there are no annular leaves of $\fes$ in $U$.
Since $\fes$ has only plane or annular leaves in $M$ it follows
that $\fes$ has only plane leaves in $U$.

Fix a flow transverse to $\fes$ in $M$. To do so, just take a smooth vector field $X$ transverse to $E^{cs}$ (which is continuous) and integrate it to a flow $\phi: M \times \RR \to M$. Pick a point $x \in U$ in some arm of the octopus decomposition (i.e. far from the core of $U$) so that the flowline through $x$ intersects the boundary in both sides. That is, for some small $t_1 < 0 < t_2$ we have that $\phi(x,(t_1,t_2)) \subset U$ but $\phi(x,t_1)$ and $\phi(x,t_2)$ belong to $\partial U$. Using that all leaves in $U$ are planes, we can apply \cite[Theorem 3.1]{Imanishi} to deduce that for every curve $\gamma\colon [0,1] \to M$ such that $\gamma(0) =x$ and $\gamma(s) \in \fes(x)$ for all $s \in [0,1]$ we have a well defined holonomy from $\phi(x, (t_1,t_2))$ to the flow line $\phi(\gamma(s),\RR)$ through $\gamma(s)$. This implies that when we lift to the universal cover (where flowlines cannot intersect the same leaf twice) the leaf space of $\wfes$ in the lift of $U$ has to be homeomorphic to $\RR$ as we wanted to prove. 
\end{proof}

In other words this result implies
that the leaf space of $\fes$ in $\cN_\eps$ is homeomorphic to $\RR$.
In particular the same is
true for the leaf space of $\wcs$ in $\cN$.

Our assumption is that either every leaf in $\cN$ has trivial stabilizer,
 or that every leaf of $\cN$ has exactly the same stabilizer which is $\ZZ$.
Denote by $G < \pi_1(M)$ the subgroup of deck transformations fixing $\cN$. 
The group $G$ is the same group which fixes $\cN_{\eps}$.
We need the following property:

\begin{claim} Up to deck transformations $\cN$ is $\ft$-periodic.
\label{claim.nperiodic}
\end{claim}

\begin{proof}
The projection $V$ of $\cN$ to $M$ may not be open if $\cs, \cu$ are not foliations
and rather branching foliations. Nevertheless $V$ is not a single
leaf and has non empty interior, hence contains an open unstable segment
$\tau$. Let $x$ in $\tau$.  Iterating positively by $f$ one gets
a limit point of the sequence $f^n(x)$. 
Since $\cP$ is $\ft$ and $\pi_1(M)$-invariant then 
for some fixed $n$,
$f^n(V)$ and $V$ intersect in their interiors. Hence $\ft^n(\cN)$ is
a deck translate of $\cN$.
\end{proof}

By the claim after taking an iterate of $f$ and perhaps a different
lift, we may assume that $\ft$ preserves $\cN$. 
We take such an iterate and lift.
We fixed an identification
of $\pi_1(M)$ as the group of deck translations of $\mt$, and $G$ is the
subgroup of deck transformations fixing $\cN$.  Then $f$ acts
on $G$ by $g \ \to \ \ft \circ g \circ (\ft)^{-1}$, this
action is denoted by $f_\ast$.
We will need some arguments from standard $3$-manifold topology.
If the stabilizer of leaves in $\cN_{\eps}$ is always trivial, 
then as it acts freely on $\RR$ it follows that $G$ is abelian.
By \cite[Theorem 9.13]{Hempel} we get that $G$ can be either $0$, $\ZZ$, $\ZZ^2$ or $\ZZ^3$. 
Suppose on the other hand that the subgroup stabilizing every
leaf of $\cN_{\eps}$ is infinite cyclic, generated by $\gamma$.
It is very easy to see that for any $\alpha$ in $G$ then
$\alpha \gamma \alpha^{-1} = \gamma^{\pm}$, hence $\langle \gamma \rangle$
is a normal subgroup. In addition $G/ \langle \gamma \rangle$ acts freely
on $\RR$ hence it is abelian. 
Since $\alpha \gamma \alpha^{-1} = \gamma^{\pm}$ it follows that
$G$ has a subgroup of index $2$ which is abelian. 
Again by \cite[Theorem 9.13]{Hempel} this subgroup $G'$ 
of index $2$
can only be $\ZZ, \ZZ^2, \ZZ^3$.
Notice that $f_\ast(\gamma) = \gamma^{\pm}$ so $f_\ast$ preserves
$G'$. So in any case $f_\ast$ preserves an abelian subgroup $G'$
of index at most $2$, which can only be $0, \ZZ, \ZZ^2, \ZZ^3$.
We need the following:

\begin{claim}
The action of $f_\ast$ in $G'$ does not have eigenvalues of modulus larger than $1$. 
\end{claim}
\begin{proof}
Recall that using the octopus decomposition we saw that $\pi_1(\hat U)$ the completion of $U$ is finitely generated. 

In the proof of Lemma \ref{lem.rcov} we analyzed the case that
a boundary component of $\hat U$ is a plane.
In that proof we showed that this implies that $\hat U$
is homeomorphic to $\RR^2 \times [0,1]$ and the length of
unstable segments between the boundary components of
$h_\eps(\hat U)$ is bounded.
But the difference here is that the lamination $\Lambda^*$ is
$f$-invariant. In particular the previous claim
showed that up to deck transformations $\cN$ is $\ft$ periodic.
This contradicts that lengths of unstable segments between
the boundary leaves is bounded, as $f$ increases unstable
lengths by a definite amount.

It follows that boundary components of $\hat U$ are either
tori or cylinders, this uses the orientability condition on $M$ and the transversal orientability. 
For each such cylinder component, the intersection with $K$
is a compact cylinder $A$. Then there is a cylinder in $\partial T$
connecting a boundary component of $A$ with another leaf
in $\partial \hat U$. This other leaf $Z$ in $\partial \hat U$ must
necessarily also be a cylinder and there is an associated compact
cylinder $Z \cap K$. One continues this process, after finitely
many steps one arrives back at $A$. This produces a torus.

We proved that all
boundary components of $K$ are tori. Since leaves of 
$\widetilde{\cs}$ are properly embedded in $\mt$ it follows that either $K$ is a solid torus 
or that all boundary components of $K$ are $\pi_1$-injective in $\pi_1(M)$ and hence $\pi_1(\hat U)$ injects in $\pi_1(M)$. Note that the image is exactly $G$. 

We proved before that $G'$ can be only $0, \ZZ, \ZZ^2, \ZZ^3$.

The claim is trivial if $G'$ is either $0$ or $\ZZ$. 
If $G'=\ZZ^3$, using that $M$ is prime we can apply \cite[Theorem 9.11]{Hempel} to deduce that $M$ has virtually abelian fundamental group contradicting that we have assumed that the fundamental group of $M$ is not virtually solvable. 

Finally, if $G'= \ZZ^2$ then \cite[Theorem 10.5]{Hempel} implies that $K$ is $\TT^2 \times [0,1]$ up to double cover. This case was dealt with in \cite{HHU-tori}. 
We explain the main steps: 
the leaves
of $\cs$ in the boundary of $\hat U$ are infinite cylinders.
Let $\alpha$ be a generator of the fundamental
group of one of these cylinders. Since $f$ preserves $\cN$ then
up to a power it preserves this boundary component of $\hat U$
and up to another power preserves $\alpha$. This implies that
one of the eigenvalues of $f_\ast$ has a power which is one.
This implies the result.
\end{proof}

Notice in particular that since $f_\ast$ is invertible, then
the above claim implies that all
eigenvalues of $f_\ast$ have modulus $1$.

We now complete the proof of Proposition \ref{prop.laminations}.
The contradiction will be given by a volume versus length argument that will imply that the action of $f_\ast$ on $G'$ must have an eigenvalue of modulus larger than one. More precisely, \cite[Proposition 5.2]{HaPS} implies that if there is an open $f$-invariant set $X \subset M$ such that the inclusion $\imath \colon X \en M$ verifies that $\imath_\ast (\pi_1(X))$ is abelian and there is a strong unstable manifold inside $X$ which is at distance $\geq \eps$ 
from the boundary of $X$, then $f_\ast$ must have an eigenvalue of modulus larger than $1$ in $\imath_\ast (\pi_1(X))$. 
The same proof applies if $i_\ast (\pi_1(X))$ has a subgroup
of index $2$ which is abelian and preserved by $f_\ast$.

We will apply this result from \cite{HaPS} to the following set.
Let $X$ be interior of the projection to $M$
of the closure of $N$. Here $N$ is the union of the leaves
which are in $\cN$. This is an open $f$-invariant set (after taking the iterate we considered before). 
Notice that 

$$\imath_\ast(\pi_1(X)) \ \subset \ \imath_\ast(\pi_1(\hat U)) \
= \  \imath_\ast(G).$$

Let $x$ be a point in $\partial X$.
Let $\ell_x \in [0,\infty]$ be the length of the 
open unstable segment inside $X$ whose boundaries are in $\partial X$ and one of them is $x$. 
This interval is possibly trivial giving $\ell_x=0$ or a complete ray giving $\ell_x = \infty$. 
It follows that the function $\ell_x$ of $x$
cannot be bounded in $\partial X$: consider forward iterates which increase
without bound the length of unstable segments.
Take a limit of a sequence of such unstable segments of lengths
converging to infinity to obtain a full
unstable curve completely contained in the interior of $X$. Moreover, the closure of such unstable leaf must be at positive distance of $\partial X$ because of local product structure. This completes the proof of Proposition \ref{prop.laminations}.

\subsection{Funnel leaves} 
Here we show: 

\begin{prop}\label{prop.WQGF} 
In the setting of Theorem \ref{teo.QG} we have that every leaf of $\cs$ and $\cu$ is a weak quasigeodesic fan for $\cF^c$. 
\end{prop}

Recall that we proved in Proposition \ref{prop.closedfans} that the
set $P$ of leaves of $\wcs$ which are weak quasigeodesic fans is non empty, $\pi_1(M)$-invariant, $\ft$-invariant, and closed.
We did that in the (non branching) foliations setting, but
subsection \ref{branfol} implies the result in the branching foliations
setting as well.
Let $\Lambda$ be the projection of the leaves in $P$ to $M$.
This is a closed, $f$-invariant set of $\cs$ leaves, that is
a sublamination of $\cs$.
We want to show that these are all the leaves of $\cs$.

Notice again that, if we assumed that the branching foliations are $f$-minimal (see \cite{BFFP_part1,BFFP_part2}), which happens for instance when $f$ is transitive, then (by definition of $f$-minimality) $\Lambda$ would automatically consist of all the leaves of $\cs$.

So the rest of this section will deal with the general case, and the reader only interested in the transitive case can skip this section.

In order to prove that $\Lambda$ covers all the leaves of $\cs$, we will first consider a slightly larger lamination such that the leaves in the complementary region are all planes. This will allow us to apply Proposition \ref{prop.laminations}. First we show that annular leaves which are not in $\Lambda$ can only accumulate on $\Lambda$.

We need the following lemma.

\begin{lemma}\label{l.addannuli}
Consider the set $P_0 \subset \cL^{cs}$ consisting of leaves invariant under some non trivial deck transformation. Then, the set $P  \cup P_0$ is a
closed set of leaves of $\lcs$ 
which is $\ft$- and $\pi_1(M)$-invariant.
In other words the set of leaves in $P \cup P_0$ projects
to an $f$-invariant lamination of $\cW^{cs}$.
\end{lemma}

\begin{proof}
Recall that $\Lambda$ is the projection of $P$ to $M$. We will also use the approximating foliation $\fes$ of $\cs$ (which can be acheived up to finite lifts, see Theorem \ref{teo-openbranch}) and denote by $\Lambda_\eps$ the lamination in $\fes$ induced by the blown up leaves of $\Lambda$. Since leaves with non trivial fundamental group are clearly $f$-invariant, 
the set $P \cup P_0$ is $\ft$- and $\pi_1(M)$ invariant.
We next show that $P \cup P_0$ is a closed subset of $\cL^{cs}$.

Consider the completion $\hat U$ of a connected component $U$ of $M \smallsetminus \Lambda_\eps$ and its octopus decomposition (cf.~\cite[Proposition I.5.2.14]{CanCon}) with a thin part $T$ and a core $K$ so that $K$ is compact and $T = T_1 \cup \ldots \cup T_m$ where each $T_i$ (an arm) is an $I$-bundle. 
By augmenting $K$ we may assume that $K$ is connected.

Lemma \ref{lem.annuluslimit} implies that every annulus leaf $B$ of $\fes$ in
$M \smallsetminus \Lambda_\eps$ accumulates only in
$\Lambda_\eps$.
Suppose that it is contained in the component $U$ as above.
Recall that $\hat U  = K \cup T$. 
We choose $K$ big enough so that each component $K \cap T$ (which
is also an annulus) is transverse to $\fes$.
Then except for a compact subannulus in $B$, the rest of $B$ 
is contained
in $T$. In particular since the foliation restricted to each 
component of $T$ is a foliated $I$-bundle, 
it follows that $B \cap K$ is a compact annulus $K_B$.
Using \cite[Theorem I.6.1.1]{CanCon} we know that the set of leaves 
of $\fes$ restricted to $K$ which are compact is a compact set.
Notice that the intersection of a leaf $B$ of $\fes$ in $U$ with
$K$ is compact if and only if $B$ is an annulus (the other
option is $B$ is a plane).
Hence the set of annuli leaves in $\fes|_K$ is a compact set.

This shows that $P \cup P_0$ is a closed subset, and shows
$P \cup P_0$ is a sublamination of $\cs$ which is
$f$-invariant. 
This proves the lemma.
\end{proof}

We now prove Proposition \ref{prop.WQGF}:

\begin{proof}[Proof of Proposition \ref{prop.WQGF}]
Corollary \ref{cor.ciclic} shows that every leaf of $\cW^{cs}$ is
either a plane or an annulus.
Suppose by way of contradiction that $P$ is not all of $\cL^{cs}$.

We use the setup of the previous lemma.
Let $U$ be a non empty 
connected component of $M \smallsetminus \Lambda_\eps$.

As in the previous lemma, we have that 
$\hat U = K \cup T$.

We consider first the case that
every leaf of $\cW^{cs}$ in $U$ is annulus.
In this case 
we show that 
that every leaf in $U$ is invariant under the same deck transformation. 
This will directly contradict Proposition \ref{prop.laminations}.
We first claim that 
since $K$ is compact, there is a finite set $\{\gamma_1, \ldots, \gamma_k\}$ in $\pi_1(K)$ such that every leaf in $U$ must be fixed by one of the $\gamma_i$. 
This is because any such annulus leaf is incompressible in $K$, and distinct
leaves are disjoint.  Hence
there are finitely many of these which 
are pairwise not isotopic \cite{Hempel}.
If they are isotopic then they correspond to the same deck
transformation.

This gives a partition of $K$ by disjoint compact sets each of which is fixed by some $\gamma_i$.
Since these sets are disjoint for distinct $i$, and $K$ is connected,
it follows that there is a single $\gamma_i$. In other words,
all leaves in this component $U$ are left invariant by the same
deck transformation.
Proposition \ref{prop.laminations} shows that this is impossible.

The other possibility is that not all leaves in $U$ are annuli.
In other words the set $\cQ:= P \cup P_0$ is not $\cL^{cs}$.
The previous lemma shows that the set of leaves in $\cQ$
projects to an $f$-invariant sublamination of $\cs$.
Let $\cN_1$ be a complementary component
of $\cQ$. Since we took out
all leaves which are annuli, it follows that all leaves of
$\wcs$ in $\cN_1$ have trivial stabilizer. 
Proposition \ref{prop.laminations}, now applied
to $\cP = \cQ$ shows that this is impossible.
We conclude that this case cannot happen either.

This contradiction shows that the assumption that $\cP$ is 
not $\cL^{cs}$ is impossible.
Since every leaf in $\cP$ is a weak quasigeodesic fan,
this finishes the proof of Proposition \ref{prop.WQGF}.
\end{proof}

\subsection{Unique centers for given limit points} 
Here we show the following, that together with Proposition \ref{prop.WQGF} completes the proof of Theorem \ref{teo.QG}: 

\begin{prop}\label{prop.Unique}
If every leaf of $\cs$ and $\cu$ is a weak quasigeodesic fan, then they all are quasigeodesic fans.
\end{prop}

We are going to prove Proposition \ref{prop.Unique} by contradiction, dealing with leaves of $\cu$, the case of $\cs$ being symmetric.

By Lemma \ref{rem.sweep}, for any leaf $V$ of
$\wcu$ with funnel point $p \in S^1(V)$ and every point
$q$ in $S^1(V) \smallsetminus \{ p \}$, there is a center leaf in $V$ with 
ideal point $q$.
By contradiction, we will assume that there is a leaf $V_0$ of $\wcu$ which has more than one center curve with the same pair of limit points $p, q \in S^1(V_0)$. Since the leaf $V_0$ is a weak quasigeodesic fan, the set of center leaves that have $p$ and $q$ as limit points forms a non trivial closed interval. Let $I$ be the interior of the interval of leaves of $\wcs$ which intersects $V_0$ in some of those centers. We think of $I$ as an open interval of $\cL^{cs}$. 

\vskip .1in

We claim that
the funnel direction in leaves of $\wcs$ varies
continuously. 
In order to understand this we put a topology on the circle 
bundle  over $\cL^{cs}$, made up of $S^1(L)$ where $L \in \cL^{cs}$.
We denote this circle bundle by $\cV$. 
Consider an approximating foliation $\cF^{cs}_\eps$,
with an associated circle bundle $\cV_\eps$. There is a canonical
bijection between the two circle bundles. 
One can also put a Candel metric in $\cF^{cs}_\eps$ without
changing $\cV_\eps$. Let $\cL^{cs}_\eps$ be the leaf
space of $\widetilde{\cF^{cs}_\eps}$. The topology in $\cV_\eps$ 
is defined as follows: 
given a transversal $\tau$ to $\widetilde{\cF^{cs}_\eps}$
consider the topology in 

$$A \  = \ \bigcup \{ L \in \cL^{cs}_\eps,
L \cap \tau \not = \emptyset \}$$

\noindent
 induced by a natural bijection
between the unit tangent bundle of $T \widetilde{\cF^{cs}_\eps}$
restricted to $\tau$: For every $x$ in $\tau$ contained in $L$
leaf of $\widetilde{\cF^{cs}_\eps}$ and unit vector $v$ in $L$ at
$x$ it defines a unique geodesic ray $r$ in the hyperbolic metric
in $L$, so that $r$ starts in $x$ with direction $v$. 
The ideal point of $r$ is a point in $S^1(L)$ and it is
associated with $v$. For details we refer to \cite{Calegari}
where it is proved that this topology is independent of the
choice of $\tau$ and it is invariant under deck transformations.
In the same way it is not hard to prove that 
the topology induced in $\cV$
is independent of the approximation $\cF^{cs}_\eps$ and
it is also $\pi_1(M)$-invariant.

The claim is that the funnel direction is continuous as
a function of $L \in \cL^{cs}$.
This is because the funnel direction $x \in S^1(L)$ in a leaf $L$ of
$\wcs$ is the one where center leaves are eventually within $2k$ of
each other. In the other direction of the
center leaves some of them
diverge a lot from each other. 
So near $L$ one sees in directions close to $x$ 
in $\cV$, the center leaves
which are within $2k +1$ of each other for a long distance,
while in the opposite direction (with respect to centers)
they diverge substantially
from each other. This means that the funnel direction
in leaves near $L$ is close to the funnel direction in $L$
when seen in $\cV$.
\vskip .1in

For any $L$ in $I$ the funnel direction in $L$ defines a direction
in the center $L \cap V_0$. Since these vary continuously with
$L$, it follows that up to switching $p$ and $q$,
 the stable funnel direction for any $L$ in
$I$ is the direction in $L \cap V_0$ with ideal point $p$. 
This implies that for any $L$ in $I$,  the rays in the funnel direction
of $L \cap V_0$ are eventually $2k +1$ from each other.
This is the fundamental fact here.
We let 

$$ Q = \bigcup_{n \in \ZZ} \bigcup_{\gamma \in \pi_1(M)} \ft^n (\gamma I). $$

This is a non empty, open $\ft$ 
and $\pi_1(M)$-invariant subset of $\cL^{cs}$ and we consider $P = \cL^{cs} \smallsetminus Q$. 
Let $\Lambda$ be the lamination in $M$ obtained by projecting
the leaves in $P$ to $M$. 
We want to show that $P$ is everything, and therefore get a contradiction, since $I$ and hence $Q$ is not empty.
For this, we will again apply Proposition \ref{prop.laminations} to a lamination $\Lambda^*$ that contains $\Lambda$; to construct it we need some preliminary results. 

We will use the approximating foliation
setting. Let $\Lambda_\eps$ be the sublamination of $\fes$ associated
with $\Lambda$ and let $U$ be a connected component of $M \smallsetminus
\Lambda_\eps$.
We will need the following technical property:

\begin{claim} 
Let $L_1, L_2$ be two leaves in the same 
component of $Q$. Then, there is a constant $K=K(L_1,L_2)>0$ such that for every pair of center 
leaves $c_i \in L_i$, for $i = 1,2$, we have that there is a ray $r_1$ of $c_1$ and a ray $r_2$ of $c_2$ both in the 
funnel directions of $L_1$ and $L_2$ respectively,
such that the Hausdorff distance $d_H(r_1,r_2)$ in $\mt$ is
less than $K$.
\end{claim}

\begin{proof} We can cover a path joining $L_1$ and $L_2$ by finitely many translates and iterates of $I$. Each translate is a 
deck translate of an $\ft$ iterate of $I$.
Deck translates do not change the geometry. The map $\ft$ has bounded
derivatives so distorts distances by a bounded multiplicative amount.
Hence it is enough to prove this for leaves in $I$.
Let then $L_1, L_2$ in $I$ and $r_1, r_2$ rays of centers
$c_i$ in $L_i$ such that $r_i$ is in the funnel direction in $L_i$.
Then in $L_i$ the center $c_i$ has in the funnel direction
the same ideal point as $V_0 \cap L_i$. Hence $r_i$ has
a subray with Hausdorff distance in $L_i$ less than $2k + 1$
from a subray of $V_0 \cap L_i$ in the funnel direction of $L_i$.
Then in $V_0$, the centers $V_0 \cap L_1$, $V_0 \cap L_2$ have subrays
which are less than $2k + 1$ in Hausdorff distance in $V_0$ from 
each other. Then the ray $r_1$ has a subray less than 
$2k+1$ in $L_1$ from a ray of $V_0 \cap L_1$, which in turn is less
than $2k+1$ in $V_0$ from a ray of $V_0 \cap L_2$ $-$ this is the
fundamental fact referred to above. Finally there is a subray
of $V_0 \cap L_2$ 
less than $2k+1$ in $L_2$ from a subray of $r_2$.
It  follows that $r_1, r_2$ have subrays which are
$6k + 3$ Hausdorff distant from each other in $\mt$.
This gives the desired bound. 
\end{proof}

The main property we need is the following:

\begin{lema}\label{lem.annuli}
Let $B$ be an annular leaf of $\fes$ in $U$. Then $B$ only
limits on points in $\Lambda_\eps$.
In particular this shows that $Q$ cannot be all of $\lcs$.
\end{lema}

\begin{proof}
Recall that $U$ is a component of $M \smallsetminus \Lambda_\eps$.
Let $A$ be the leaf in $\cs$ corresponding to $B$ under the map $h^{cs}$ given by Theorem \ref{teo-openbranch}. Since $B$ is an annulus, so is $A$ and we call again $\gamma$ a generator of $\pi_1(A)$.

Thanks to Proposition \ref{prop.WQGF}, every center leaf shares one ideal point (the funnel point), which is therefore a fixed point of $\gamma$. We explained before (cf. Corollary \ref{cor.ciclic}) that by compactness of $M$, $\gamma$ cannot act parabolically on $S^1(L)$, so it must fix two points on $S^1(L)$.
Hence, $\gamma$ fixes a center curve in $L$, which projects to a closed curve in $A$.

Let $e$ be the corresponding closed center curve in $B$.
Let $\widetilde U$ be a lift of $U$ to $\mt$.
Suppose that $B$ limits to a point in $U$. Hence there
are infinitely many lifts $L_i$ of $B$ contained in $\widetilde U$
and limit to $L$ leaf in $\widetilde U$.
Each such lift $L_i$ contains a lift $c_i$ of $e$. 
The leaf $L$ is contained in an image $\gamma \ft^n(I)$, so there is some $K>0$ as in the claim above that
works for any pair $E_1, E_2$ in $\gamma \ft^n(I)$.
The claim also works for the approximating foliations,
taking the intersections of leaves of $\wfes, \wfeu$.
Hence for any $i, j$ then $c_i, c_j$ have rays
a fixed bounded distance $K$ from each other in the funnel
direction in $L_i, L_j$. But every $c_i$ is a lift of
a fixed closed curve $e$. As the bound is the same, we get a contradiction, since the lifts of $e$ form a uniformly
discrete set in $\mt$.
This contradiction proves the first assertion of the lemma.

To prove the second assertion suppose that $Q = \lcs$.
First recall that $\cs$ has an annular leaf $A$. 
Otherwise all leaves of $\cs$ are planes. 
This implies that $M$ is the $3$-torus - this was proved by 
Gabai, see \cite[Corollary 1.2]{Li}.
In particular $\pi_1(M)$ is abelian, which we are assuming is not the case.
Hence $\cs$ has an annular leaf $A$. Since it is non compact it
limits somewhere. If $Q = \lcs$ the argument to prove the
first assertion leads to a contradiction. 
This shows that $Q$ is not $\lcs$.
\end{proof}

\begin{proof}[End of the proof of Proposition \ref{prop.Unique}]
From Lemma \ref{lem.annuli}, we deduce that $\Lambda$ is 
not empty.
Let $\Lambda'$ be the union of the annular leaves of $\cs$.
For any annular leaf $A$ not in $\Lambda$, the previous
lemma shows that it limits only on $\Lambda$.
This is the technical property that is needed to
deduce that $\Lambda \cup \Lambda'$ is a sublamination
of $\cs$ (as in the proof of the first assertion of Lemma \ref{l.addannuli}). 

Hence we can finish in exactly the same way as Proposition \ref{prop.WQGF}:
$\Lambda \cup \Lambda'$ is lamination such that the complement has no holonomy so we can apply Proposition \ref{prop.laminations}  to get that $\Lambda \cup \Lambda'$ is all of $\cs$. Now the proof of Proposition \ref{prop.WQGF} 
also applies here to deduce that $\Lambda$ is itself all of $\cs$. This contradicts the fact that $Q= \lcs\smallsetminus P$ is non empty, and thus ends the proof of Proposition \ref{prop.Unique}.
\end{proof}

\begin{remark} 
Notice that in order to obtain, in Proposition \ref{prop.WQGF}, that each leaf of $\cs$ is a \emph{weak} quasigeodesic fan, we only needed to use the fact that the center leaves were uniform quasigeodesic in $\cs$ (and vice versa for $\cu$). To get here that it is actually a quasigeodesic fan, we need to use the fact that center leaves are uniform quasigeodesic in \emph{both} $\cs$ and $\cu$.
\end{remark}

\section{A criterion. Proof of Theorem \ref{teo.main5}}\label{s.criteria}

In this section we prove Theorem \ref{teo.main5}.
We start by proving the converse direction, in \S\ref{ss.constructingexpansiveflow} and \S\ref{ss.deducingCAF}, and prove the direct implication in \S\ref{ss.converse}.

In particular, we consider $f\colon M \to M$ to be a partially hyperbolic diffeomorphism preserving branching foliations $\cs$ and $\cu$ whose leaves are Gromov hyperbolic with the induced metric. We assume that centers in each leaf of $\cs$ and $\cu$ are uniform quasigeodesics so that Theorem \ref{teo.QG} applies.  

To show that being quasigeodesic partially hyperbolic diffeomorphism
implies leaf space collapsed Anosov flow we will assume that the bundles $E^{s}$, $E^c$ and $E^u$ are orientable. (Note that orientability of $E^c$ is a consequence of the definition and Theorem \ref{teo.QG}.)

\subsection{Constructing an expansive flow}\label{ss.constructingexpansiveflow}

Let $f\colon M \to M$ be a quasigeodesic partially hyperbolic diffeomorphism. We assume that the bundles $E^{s}$, $E^c$ and $E^u$ are orientable. 

Let $\cs$ and $\cu$ be the center stable and center unstable branching foliations given by Definition \ref{defiQG}. Since the bundles are assumed to be orientable, we can apply Theorem \ref{teo-openbranch} to obtain approximating foliations $\fes$ and $\feu$ with maps $\hs$ and $\hu$. The intersection of $\fes$ and $\feu$ gives rise to an orientable foliation $\fec$ tangent to a vector field $X^{c}$. 

Note that Theorem \ref{teo.QG} shows that in each leaf of $\fes$ (resp. $\feu$) we have that the foliation $\fec$ is made of uniform quasigeodesics and that no two of them share both points at infinity. 
(In fact, Theorem \ref{teo.QG} implies that inside each leaf of $\fes$ (resp. $\feu$) the foliation $\fec$ is a quasigeodesic fan, but we will not need this in the following.) 

\begin{prop}\label{prop.flowisexpansive}
The flow $\phi^c_t\colon M\to M$ generated by $X^c$ is expansive and preserves the transverse foliations $\fes$ and $\feu$.
\end{prop}

\begin{proof}
Recall (see \S\ref{branfol}) that since $f$ is a quasigeodesic partially hyperbolic diffeomorphism, the leaves of the approximating foliations $\fes$ and $\feu$ are also Gromov hyperbolic
(one can even choose these to be by hyperbolic 
surfaces \cite[Chapter 8]{Calegari}).
By hypothesis, the orbits of the flow $\phi^c_t$ are quasigeodesics in the leaves of each of the foliations. 

There is $\delta_0>0$ such that every leaf of $\wfes$ and $\wfeu$ is properly embedded in its $\delta_0$-neighborhood in $\mt$ (see, e.g., \cite{Calegari}). 

By that we mean that: 
\begin{enumerate}
\item any set of diameter less than $\delta_0$ is
contained in a foliated chart of each of these foliations; and 
\item if $p$ is in a leaf $L$ of $\wfes$ or $\wfeu$ then the ball of
radius $\delta_0$ around $p$ in $\mt$ intersects $L$ only in the local
sheet of $L$ through $p$.
\end{enumerate} 

Now choose $\delta < \delta_0$ so that if two points $x, y$ in $\mt$
are less than $\delta$ apart then $\wfes(x)$ intersects $\wfeu(y)$
in a point less than $\delta_0$ from both of them and similarly
for $\wfes(y) \cap \wfeu(x)$.
We will show that $\delta$ serves as an expansivity constant for the flow $\widetilde{\phi^c_t}$, and this implies that the flow $\phi^c_t$ is expansive too. Since there is no recurrence for the flow in $\mt$ the definition of expansivity is equivalent to showing that different orbits cannot remain a bounded Hausdorff distance apart, cf.~Remark \ref{rem.expunivcov}. 

Assume by contradiction that two different orbits $o_1$ and $o_2$ of $\widetilde{\phi^c_t}$ in $\mt$ are at Hausdorff distance less than $\delta$. These orbits correspond to leaves of the intersected foliation $\wfec$ between $\wfes$ and $\wfeu$.
Suppose first that they are in the same leaf of $L$ of
 $\wfes$ (or $\wfeu$).
Since they are not the same orbit, they cannot have both ideal
points the same in $S^1(L)$, by Theorem \ref{teo.QG}.
Hence they diverge from each other infinitely in $L$ in some
direction. By the choice of $\delta_0$ they diverge from each other
at least $\delta_0$ (and hence at least $\delta$)
in $\mt$ as well in that direction.
Suppose now that $o_1, o_2$ are not the same leaf of $\wfes$ or $\wfeu$.
Let then $o_3$ be the intersection of $\wfes(o_1) \cap \wfeu(o_2)$. 
Then $o_3$ is distinct from both $o_1, o_2$.
Since $o_1, o_2$  are always less than $\delta$ apart then $o_3$ is 
less than $\delta_0$ apart from either $o_1$ or $o_2$. 
Since $o_3, o_1$ are in the same $\wfes$ leaf the first argument
shows that this is a contradiction, that is $o_1, o_3$ have
to diverge from each other more than $\delta_0$.
This shows that $\delta$ works as an expansive constant for the flow.

It is obvious that the flow preserves the described foliations.
This finishes the proof of the proposition.
\end{proof}

\subsection{Deducing that the map is a collapsed Anosov flow} \label{ss.deducingCAF}
We can now show:

\begin{prop}\label{prop.transitivecase}
The flow $\phi^c_t$ is a topological Anosov flow and $f$ is a leaf space collapsed Anosov flow with respect to $\phi^c_t$. 
\end{prop}

\begin{proof}
Notice first that by Proposition \ref{prop.flowisexpansive} and Theorem \ref{teo.expimpliesTAF} we know that the flow $\phi^c_t$ is a topological Anosov flow. Moreover, by Proposition \ref{p.nontransitiveAF} we know that the foliations $\fes$ and $\feu$ correspond to the weak stable and unstable foliations respectively (maybe up to changing orientation of the vector field $X^c$). 

Using the maps $\hs$ and $\hu$ given by Theorem \ref{teo-openbranch} in the universal cover one can construct a $\pi_1(M)$-invariant homeomorphism $H$ from the orbit space of $\widetilde{\phi^c_t}$ to the center leaf space of $f$ as follows:
A center leaf in $\mt$ is a component $c$ of the intersection of a 
leaf $L$ of $\wcs$ and a leaf $G$ of $\wcu$. There are unique leaves
$$L' \ \in \ \wfes, \ \  G' \ \in \  \wfeu \ \ {\rm so \ that } \ \
\tild{h}_{cs}(L') = L, \ \ \tild{h}_{cu}(G') = G.$$

\noindent
There is a unique
component $\alpha$ of the intersection of $L'$ and $G'$ (that is,
an orbit of $\widetilde{\phi}_t$)
which is $\eps$ close to $c$. The map $H$ is the one that sends
this orbit $\alpha$ to $c$. 

This completes the proof. 
\end{proof}

\subsection{The quasigeodesic property}\label{ss.converse} 

Here we show: 

\begin{prop}\label{prop.convmain5}
Let $f\colon M \to M$ be a leaf space collapsed Anosov flow. Then, the $\cs$-foliation is by Gromov hyperbolic leaves and the center foliation inside each leaf of $\cs$ is a quasigeodesic fan. 
\end{prop}

\begin{proof}
We do the proof for $\cs$, the same proof works for $\cu$.

Up to taking a finite cover and a lift of an iterate of $f$
we may assume that $E^s, E^c, E^u$ are orientable
and $f$ preserves the lifted foliation.
Since the quasigeodesic properties are verified in the universal
cover, this does not change the result.
In addition $f$ is still a leaf space collapsed Anosov flow
in the cover. Let $\phi_t$ be the Anosov flow associated to $f$.
Let $H\colon \oo_{\phi} \to \lc$ be the associated homeomorphism
between orbit space of $\tild{\phi}_t$ and center leaf space
in $\mt$.
Proposition \ref{p.wcaf2impliescaf2} implies that $H$  maps
$\oo^{ws}_\phi$ to $\oo^{cs}_f$ and $\oo^{wu}_\phi$ to $\oo^{cu}_f$.

Using Theorem \ref{teo-openbranch}, we can approximate $\cs$, $\cu$ by actual foliations $\fes, \feu$.
The intersection of $\fes, \feu$ is a one-dimensional
foliation $\cG$ in $M$, with lift $\widetilde{\cG}$.
Given any flow line $\alpha$ of $\tild{\phi}_t$ it is the
intersection of a stable leaf $L_0$ with an unstable leaf
$Z_0$. Under $H$ these leaves $L_0, Z_0$ map to leaves $L_1$ of $\wcs$ and
$Z_1$ of $\wcu$ respectively. 
Thanks to item \ref{item.ApporxFol_uniqueness} of Theorem \ref{teo-openbranch}, the leaves $L_1, Z_1$ are $\eps$ near some unique
leaves $L$ of $\wfes$ and $Z$ of $\wfeu$, respectively. 

Therefore $\alpha$
is associated with a unique leaf of $\widetilde{\cG}$ and
vice versa. This association is a homeomorphism from the
orbit space $\oo_\phi$ to the leaf space
of $\widetilde{\cG}$. This homeomorphism is
clearly $\pi_1(M)$ equivariant.

By results of Haefliger, Ghys and Barbot (see \cite[Prop.1.36]{BarbotHDR}),
it follows
that there is a homeomorphism $\eta$ from $M$ to $M$ sending
the flow foliation of $\phi_t$ to the foliation $\cG$.
We can then orient the foliation $\cG$ using this homeomorphism.
Hence this foliation becomes the flow foliation of a flow
$\psi_t$.
Since the flow $\phi_t$ is expansive then the flow $\psi_t$ is
also expansive. 
By Theorem \ref{teo.expimpliesTAF} it follows that
$\psi_t$ is a topological Anosov flow. By the equivalence of the flow foliations of $\phi_t$ and $\psi_t$
it now follows that 
the stable foliation
of $\psi_t$ is $\fes$. By Proposition \ref{prop.topAnosovQG}
it follows that
the foliation $\fes$ is by Gromov hyperbolic leaves and
the flow lines in leaves of $\fes$ are uniform quasigeodesics.

This implies that the leaves of $\cs$ are Gromov hyperbolic
and the center leaves in leaves of $\wcs$ are uniform quasigeodesics.

This finishes the proof of Proposition \ref{prop.convmain5}.
\end{proof}

This finishes the proof of Theorem \ref{teo.main5}.


\section{Strong implies leaf space collapsed Anosov flow}\label{s.strong_implies_leafspace} 
In this section we show that Definition \ref{defi1} implies Definition \ref{defiw2}. The main point is to construct the branching foliations from the map $h$ provided by Definition \ref{defi1}. The rest of the conditions will be rather direct.

\begin{prop}\label{prop-1implies2}
If $f$ is a strong collapsed Anosov flow, then it is a leaf space collapsed Anosov flow. 
\end{prop}

We first show the following lemma: 

\begin{lema}\label{lema-branch}
Let $f$ be a strong collapsed Anosov flow (Definition \ref{defi1}), then there are $f$-invariant branching foliations $\cW^{cs}$ and $\cW^{cu}$ tangent to $E^{cs}$ and $E^{cu}$ respectively such that the image of each of the leaves of $\cW^{cs}$ (resp. $\cW^{cu}$) coincides with $h(\cF^{ws}_\phi(x))$ (resp. $h(\cF^{wu}_\phi(x))$) for some $x \in M$. 
\end{lema}

\begin{proof}
First note that, if the topological Anosov flow $\phi$ is orbit equivalent to 
a topological Anosov flow $\phi_1$, via a homeomorphism that is homotopic to identity, then, since orbit equivalences can always be made smooth along the orbits, $f$ is also a strong collapsed Anosov flow for the flow $\phi_1$, via a map $h_1$. 
Moreover, we have that $h_1(\cF^{ws}_{\phi_1}) = h(\cF^{ws}_{\phi})$.

Hence, to prove the conclusion of the lemma for $E^{cs}$ or $E^{cu}$ we may choose an appropriate Anosov flow orbit equivalent to $\phi$ (via a homeomorphism that is homotopic to identity).

We do it for $E^{cs}$ by taking an orbit equivalent Anosov flow that has smooth weak stable leaves, which exists thanks to Proposition 
\ref{prop.smoothness}.
Similarly, for $E^{cu}$, we would choose an Anosov flow with smooth center unstable leaves.

We only do the case of $E^{cs}$, the other one being analogous. Abusing notation, we assume that $\phi$ itself has smooth weak stable leaves.

Take the pull back of the ambient Riemannian
metric. For each leaf $L \in \cF^{ws}_\phi$ 
we define a continuous local homeomorphism 
onto its image
$\varphi_L\colon U_L \to M$ where $U_L$ is the universal cover of the leaf $L \subset M$ with this intrinsic Riemannian
metric. Note that $U_L$ is a complete metric space, homeomorphic to $\RR^2$. 

We work in the universal cover, and consider $\tilde h\colon \mt \to \mt$ a lift which is a bounded distance from the identity (it exists because $h$ is homotopic to the identity). Let $\tilde L$ be a lift of $L$ to $\mt$. Since $\tilde L$ is a properly embedded plane and $\tilde h$ is a bounded distance from the identity and maps $\tilde L$ to a $C^1$-surface tangent to $E^{cs}$ by assumptions we get that the image is also a properly embedded plane in $\mt$. This means that when lifted to the universal cover, there is $\hat \varphi_L\colon U_L \to \mt$ a $C^1$-proper embedding tangent to $E^{cs}$ such that its image coincides with $\tilde h \circ \hat \varphi_L\colon U_L \to \mt$. The bounded distance to the identity implies that the image of $\hat \varphi_L$ is complete with the metric induced by the embedding. The non topological crossing is ensured by the 'transversally collapsing' in the definition of strong collapsed Anosov flow which makes backtracking impossible. 

Finally, to show the minimality condition (i.e. item \ref{item.BF_minimal} in Definition \ref{def.branching}), we need to show that the image of two different leaves of $\widetilde{\cF^{ws}_\phi}$ by $\tilde h$ are different. 
Suppose then that $L_1, L_2$ are distinct leaves of 
$\widetilde{\cF^{ws}_\phi}$ which are mapped to the same
surface by $\tilde{h}$.
Suppose first that $L_1, L_2$ intersect a common unstable leaf $F$.
In particular, the set of leaves separating $L_1$ from $L_2$ is an interval.
If for some leaf $L$ in this interval we have $\tilde{h}(L) \not =
\tilde{h}(L_1)$, then since there is no topological crossing
between leaves, it also follows that $\tilde{h}(L_1) \not = 
\tilde{h}(L_2)$, which contradicts our assumption. Thus $\tilde{h}(L) = \tilde{h}(L_1)$.
For any such $L$, the intersection $L \cap F$ is a single flow
line $\alpha_L$. The above shows that $\tilde{h}(\alpha_L)$
is contained in $\tilde{h}(L_1)$. Therefore the region
in $F$ made up of the flow lines between $\alpha_{L_1}$ and
$\alpha_{L_2}$ is mapped into $\tilde{h}(L_2)$. Therefore
this is mapped into a region tangent to $E^{cs}$. This
contradicts the fact that $F$ is mapped to a surface
tangent to $E^{cu}$ because $h$ is close to the identity and the region between $\alpha_{L_1}$ and $\alpha_{L_2}$ contains arbitrarily large disks.

For general leaves $L_1, L_2$ the region between them is the
connected component of $\mt - (L_1 \cup L_2)$ which limits on
both of them. 
If $\tilde{h}(L_1) = \tilde{h}(L_2)$ then for any leaf $L$
between $L_1, L_2$ then $\tilde{h}(L) = \tilde{h}(L_1)$.
There is a leaf $L$ which is between $L_1$ and
$L_2$ and which intersects a common unstable $F$ with $L_1$.
By the first case $\tilde{h}(L_1) \not = \tilde{h}(L)$.
This leads to a contradiction.

This finishes the proof of the lemma.
\end{proof}

\begin{proof}[Proof of Proposition \ref{prop-1implies2}]
We just need to show that there is an equivariant homeomorphism between the leaf spaces. Consider the lift $\tilde h$ of $h$ to the universal cover obtained by lifting the homotopy of $h$ to the identity. Thanks to Lemma \ref{lema-branch}, the map $\tilde h$ sends the leaves of the weak un/stable foliations to leaves of the branching center un/stable foliations. 

We claim that $\tilde h$
induces a bijection between the orbit spaces of the flow $\phi_t$ and the center leaf space. 
Assume that there are distinct flow lines $o_1, o_2$ of $\widetilde \phi_t$
which are sent to the same center leaf $c_1$ by $\tilde h$.
Let $L_i = \widetilde{\cF^{ws}_\phi}(o_i)$ and $F_i = \tilde h(L_i)$,
which are leaves of $\wcs$.
In the proof of the previous lemma we showed that $F_1, F_2$ have to be distinct leaves.
Suppose first that $L_2$ intersects $U_1 := \widetilde{\cF^{wu}_\phi}(o_1)$.
Let $Z_t, 0 \leq t \leq 1$ be the interval of stable leaves
intersecting $U_1$ between $L_1$ and $L_2$.
Let $Y_t = \tilde h(Z_t)$, center stable leaves.
We consider $\ell_t := \tilde h(Z_t \cap U_1)$ which is contained in 
$\tilde h(U_1)$ a leaf of $\wcu$. We have that $\ell_0$ is already in
$F_2$. Since the $\tilde Z_t$ have to be between $F_1$ and $F_2$, it
follows that $\ell_t = c_1$ for all $t$. 
But $\ell_t = \tilde h(U_1) \cap \tilde h(Z_t)$ hence $\ell_t$
is a center leaf in $\tilde h(U_1)$.
This was proved impossible
in the previous lemma $-$ using the center unstable foliation
instead of the center stable foliation.

The remaining possibility is as follows:
Let $\cU$ be the open set in $\mt$ which is the union of leaves
of $\widetilde{\cF^{ws}_\phi}$ intersecting $U_1$. 
Since $L_2$ is disjoint from $\cU$ there is a unique leaf
$L$ in the boundary of $\cU$ which is either $L_2$ or
separates $L_1$ from $L_2$. In any case $c_1$ is contained
in $L$. There are two possibilities: 1) $L_1, L$ intersect
a common transversal, 2) $L_1, L$ are non separated from
each other in the leaf space of $\widetilde{\cF^{ws}_\phi}$.
In case 1) consider $G$ a stable leaf separating $L_1$ from $L$.
Then $G$ intersects $U_1$, and $c_1$ is contained in $\tilde h(G)$.
So the same proof as in the previous paragraph concludes.
In the case 2) consider $G$ a stable leaf intersecting
both $U_1$ and $U_2 := \widetilde{\cF^{wu}_\phi}(o_2)$.
The separation properties show that $\tilde h(G)$ has
to be between $F_1$ and $F_2$. But then $G \cap U_1$ and
$G \cap U_2$ are two distinct orbits that map to $c_1$.
This contradicts that $\tilde h(G)$ intersects a transversal
at most once.

We conclude that $\tilde h$ induces a bijection from the orbit
space of $\widetilde \phi_t$ to the center leaf space.
This bijection respects center stable and center unstable leaves and their ordering inside each of the foliations. Therefore the bijection
is a homeomorphism. This homeomorphism is clearly equivariant (since $\tilde h$ commutes with deck transformations).  
\end{proof}

\section{Leaf space implies strong collapsed Anosov flow} \label{s.leafspace_implies_strong}
In this section we will show that Definition \ref{defiw2} implies Definition \ref{defi1} under some orientability assumption. Together with  Proposition \ref{prop-1implies2} it completes the proof of Theorem \ref{teo.main3}. 

\begin{prop}\label{prop-2implies1}
If $f$ is a leaf space collapsed Anosov flow and $E^{cs}$ is transversally orientable, then it is a strong collapsed Anosov flow. 
\end{prop}

The strategy is quite simple, we wish to map each orbit of the Anosov flow to the corresponding center curve given by Definition \ref{defiw2}. The difficulty in implementing the strategy has to do with the fact that we only have a map at the level of leaf spaces, so we first need to construct an actual map of the manifold which realizes this equivalence, for this, we first construct a specific realization of the (topological) Anosov flow that allows us to get this map in a natural way. Once this is done, a standard averaging argument achieves the local injectivity along orbits of the flow.

\subsection{Constructing a convenient realization of the Anosov flow} 
Let $f$ be a leaf space collapsed Anosov flow with $E^{cs}$ transversally orientable. We consider $\phi_t\colon M \to M$ the topological Anosov flow and $H\colon \cO_\phi \to \cL^c$ given by Definition \ref{defiw2}.

We will start by applying Theorem \ref{teo-openbranch} to $\cs$ to get an approximating foliation $\fes$. We denote by $\wcs$ and $\wfes$ the lifts to $\mt$. As explained in \S \ref{ss.Candelmetric} we can consider a metric on $M$ that makes leaves of $\fes$ negatively curved. 
In Proposition \ref{prop.convmain5} we proved that the center leaves inside each leaf of $\wcs$ form a quasigeodesic fan.
Then we can pull them back to each leaf of $\wfes$ to get a funnel point $p(L) \in S^1(L)$ in each leaf $L \in \fes$.

We consider the flow $\widetilde{ \psi_t} \colon \mt \to \mt$ defined as follows: For a point $x \in L \in \wfes$ we consider $\widetilde{ \psi_t}(x)$ to be moving along the geodesic through $x$ with endpoint the funnel point of $L$ at unit speed. This definition is clearly $\pi_1(M)$-invariant, and this flow descends to $M$ and we denote it by $\psi_t$. 

In Proposition \ref{prop.flowisexpansive} we proved the following:

\begin{prop}\label{prop.hatphiisTAF}
The flow $\psi_t$ is topologically Anosov and orbit equivalent to $\phi_t$ by an orbit equivalence homotopic to the identity.  
\end{prop}

\subsection{Averaging to construct the map} 
\label{averaging}

We will now construct a map $h_0\colon M \to M$ which maps orbits of $\psi_t$ (cf.~Proposition \ref{prop.hatphiisTAF}) to curves tangent to the center. Later we will modify this map and construct the self orbit equivalence to verify Definition \ref{defi1}. Denote by $H_0\colon \cO_\psi \to \cL^c$ the $\pi_1(M)$-invariant homeomorphism between leaf spaces. Recall that Proposition \ref{p.wcaf2impliescaf2} implies that $H_0$ maps the weak-stable/unstable foliations of $\psi_t$ to the center stable and unstable branching foliations of $f$. 

\medskip

\noindent {\bf Construction of a map:} For a fixed small $\eps>0$, we denote by $h^{cs}\colon M \to M$ the collapsing map from $\fes$ to $\cs$ given by Theorem \ref{teo-openbranch}. 

Pick a point $x \in \mt$ and let $\ell_x$ be the center leaf $H_0(o_x)$ where $o_x$ is the orbit of $x$ by $\tild{\psi}$.  Note that $o_x$ is a geodesic in a negatively curved surface, and we can push the Riemannian
metric in $L_x:=\tild{\cF^{ws}_\psi}(x) = \tild{\fes}(x)$ to $\tild{h}^{cs}(L_x)$ which is a leaf of $\wcs$. 
We can push the metric 
because $h^{cs}$ is a local diffeomorphism between respective
leaves of $\fes$ and $\cs$, and this lifts to diffeomorphisms
between respective leaves of $\wfes$ and $\wcs$.
With this metric $\tilde h^{cs}(L_x)$ is negatively curved,
$\tilde h(o_x)$ is a geodesic in $\tilde h(L_x)$ and $\ell_x$ is a quasigeodesic in $\tilde h(L_x)$ with the same endpoints.

We can then define a map $p_x\colon \ell_x \to \tild{h}^{cs}(o_x)$ by orthogonal projection in $L_x$.
Since $L_x$ is negatively curved the orthogonal projection is
a uniquely defined function and it is continuous.

\begin{lema}\label{l.convergencep} 
The map $p_x$ is proper, in particular extends continuously (as the identity) to the compactification of $\ell_x$ and $\tild{h}^{cs}(o_x)$. 
\end{lema}
\begin{proof}
This follows directly from the fact that $\ell_x$ is a quasigeodesic with the same endpoints as the geodesic $\tild{h}^{cs}(o_x)$ with respect to the chosen metric. 
\end{proof}

In principle, the map $p_x$ can fail to be injective, so one cannot define an inverse. But there is a standard procedure of averaging going back at least to \cite{Fuller} (see also \cite[Section 8]{HPsurvey} for discussion) which allows to find a natural way to invert $p_x$. 

We can define from $p_x$ a map $\hat p_x\colon \ell_x \to \RR$ by identifying $\tild{h}^{cs}(o_x)$ with $\RR$ via the map $b_x\colon \tild{h}^{cs}(o_x) \to \RR$ such that  $b_x(\tild{h}^{cs}(\psi_t(x))=t$. 

For $y,z \in \ell_x$ we denote by $[y,z]$ the segment of $\ell_x$ between $y$ and $z$. For any $t \in \RR$ we denote by $y+t$ the point in $\ell_x$ at oriented length $t$ from $y$. Lemma \ref{l.convergencep} then implies that if we choose an appropriate orientation along $E^c$ we have that the map $\hat p_x$ verifies that for every $y \in \ell_x$ we have $\lim_{t \to \pm \infty} \hat p_x(y+t) = \pm \infty$.

Let $p_x^T\colon \ell_x \to \RR$ be the map defined by

$$p_x^T(y) = \int_{[y,y+T]} \hat p_x (z) dz. $$

Lemma \ref{l.convergencep} implies that for $T>0$ large enough we have that not only $p_x^T$ is $C^1$ along $\ell_x$ but also its derivative does not vanish. Indeed, since $\hat p_x$ is continuous, we have 

\begin{multline*}
p_x^T(y+t) - p_x^T(y) = \int_{[y+T,y+T+t]} \hat p_x(z) dz - \int_{[y, y+t]} \hat p_x(z) dz  \\ \sim  t (\hat p_x(y+T) - \hat p_x(y)).  
\end{multline*}

That is, if $u_x(t)= p_x^T(y+ t)$ for some $y \in \ell_x$ then $u_x'(t)>0$ everywhere. 
In addition, the dependence of $T$ on $x$ is lower semicontinuous: if $T$ works for $x$ then $T+\eps$ works for $y$ sufficiently close to $x$.
Therefore there is $T > 0$ that works for all $x \in \mt$.

It follows that, for any fixed $x$, we can define an inverse map $q_x\colon o_x \to \ell_x$, which is a $C^1$-diffeomorphism preserving orientation, as the inverse of $p_x^T$ precomposed with $\tilde h^{cs}$. We collect some properties of $q_x$ in the following statement: 

\begin{lema}\label{l.propsq}
The map $x \mapsto q_x$ varies continuously in the $C^1$-topology in compact parts and is $\pi_1(M)$-invariant in the sense that for $\gamma \in \pi_1(M)$ we have that $q_{\gamma x} (t) = \gamma q_x(t)$. Moreover, there is a $C^1$ increasing diffeomorphism $u_x\colon \RR \to \RR$ such that $u_x(0)=0$ and if $x_t = \psi_t(x)$ then $q_{x_t}(0) = q_x(u_x(t))$. 
\end{lema}

\begin{proof}
Notice that we used the same $T$ for all $x \in \mt$.
All the objects we considered depend on continuous and $\pi_1(M)$-invariant choices.  The last property just follows from the way we defined $q_x$ and the fact that $q_{x_t}$ also has $\ell_x$ as target since $\ell_x = \ell_{x_t}$. 
\end{proof}

Now we can define the map $h\colon M \to M$. For $x \in \mt$ we define $\tild{h}(x)$ to be $q_x(0) \in \ell_x$, since this is continuous and $\pi_1(M)$-invariant it induces a continuous map $h$ in $M$ homotopic to the identity. 

\medskip

\noindent {\bf Verifying the properties.} 
We will now verify the sought properties of $h$. 

\begin{lema}\label{l.prop1h} 
The map $h\colon M \to M$ is smooth along the orbits of $\psi_t$ and the derivative maps the vector field to a (positively oriented) non-zero vector tangent to $E^c$. That is, $h$ verifies condition ($i$) in Definition \ref{defi1}. 
\end{lema}

\begin{proof}
Fix an orbit $o_x$ of $\tild{\psi_t}$ and we get that by definition for every $y \in o_x$ we have that $\ell_x = \ell_y$. Therefore, the map $h$ will map $o_x$ to $\ell_x$. By Lemma \ref{l.propsq} we deduce that the image by $h$ of the vector field is a positively oriented vector in $E^c$.  
\end{proof}

We can now proceed to prove Proposition \ref{prop-2implies1}.

\begin{proof}[Proof of Proposition \ref{prop-2implies1}]
Since the partially hyperbolic diffeomorphism is a leaf
space collapsed Anosov flow it preserves branching foliations
$\cs$ and $\cu$.
 The fact that $h$ maps every weak stable leaf into a surfaces tangent to $E^{cs}$ is direct from its construction since it maps leaves of $\tild{\cF^{ws}_{\psi}}$ to surfaces tangent to $E^{cs}$. 
 
Now, by Proposition \ref{p.wcaf2impliescaf2} we also get that the weak unstable leaves map to surfaces tangent to $E^{cu}$. The lift $\tilde h$ of $h$ to $\mt$ maps every weak stable/unstable by construction into a properly embedded surface in $\mt$ respecting the orientation (which implies the transverse collapsing property). 
 
We now need to construct the self orbit equivalence $\beta\colon M \to M$ which makes the commutation $f \circ h = h \circ \beta$ work. For this, given $x \in M$ consider $y = f \circ h(x)$. Note that $y$ may belong to several center curves. But since $h$ is injective along orbits of the flow $\psi_t$ it makes sense to consider its inverse restricted to the center curve $\ell_y:=f(\ell_x)$ which is also the image of an orbit of $\psi_t$. Then, one can define $\beta(x) = (h|_{\ell_y})^{-1}(y)$. 

One gets that $\beta(x)$ is continuous by construction and continuity of $h$ (as well as continuity of the maps $q_x$, cf.~Lemma \ref{l.propsq}). Moreover, $\beta$ is injective since it is injective along orbits as well as maps different orbits to different orbits. Finally, $\beta$ is surjective since the equation $f \circ h = h \circ \beta$ implies that $\beta$ has degree one as a map. This implies that $\beta$ is a homeomorphism which clearly preserves orbits and its orientation thus a self orbit equivalence for $\psi_t$. 
\end{proof}

The averaging method gives several ways on which a given collapsed Anosov flow can be realized (different choices of $h$ that affect the choice of $\beta$). The following remark should also be taken into account  if one wants to formulate uniqueness properties for collapsed Anosov flows. 

\begin{remark}\label{rem.changeSOE}
Let $f\colon M \to M$ be a collapsed Anosov flow with respect to a topological Anosov flow $\phi_t\colon M \to M$ and the self orbit equivalence $\beta\colon M \to M$. In particular, there exists $h\colon M \to M$ homotopic to the identity such that $f \circ h = h \circ \beta$. 

Assume that $\alpha\colon M \to M$ is another self orbit equivalence of $\phi_t$. Then, it follows that taking $\hat h = h \circ \alpha$ and $\hat \beta= \alpha^{-1} \circ \beta \circ \alpha$ we get that $f\circ \hat h = \hat h \circ \hat \beta$. Thus, if $\alpha$ is homotopic to the identity, then $f$ is also a collapsed Anosov flow associated with the Anosov flow $\phi_t$ via the collapsing map $\hat h$ and the self orbit equivalence $\hat \beta$.

Similarly, if $\psi_t\colon M \to M$ is a topological Anosov flow conjugate to $\phi_t$ by a homeomorphism $g\colon M \to M$, that is, $\psi_t = g^{-1} \circ \phi_t \circ g$. Then, if $h\colon M \to M$ is the map homotopic to the identity such that $f \circ h= h \circ \beta$ then one has that if $\hat h = h \circ g$ and $\hat \beta = g^{-1} \circ \beta \circ g$ then $\hat \beta$ is a self orbit equivalence of $\psi_t$ and $f \circ \hat h= \hat h \circ \hat \beta$.  Thus, if $g$ is homotopic to the identity, then $f$ is also a collapsed Anosov flow associated with the Anosov flow $\psi_t$ via the collapsing map $\hat h$ and the self orbit equivalence $\hat \beta$.
\end{remark}

\section{On the examples obtained via $\varphi$-transversality}\label{s.examples}

\subsection{Proof of Theorem \ref{teo.main1}}

In order to prove Theorem \ref{teo.main1}, we first collect some facts that are easily extracted from \cite{BGHP}.

\begin{prop}\label{p.BGHP}
Let $\phi_s\colon M \to M$ be an Anosov flow generated by a vector field $X$ and $\varphi\colon M \to M$ a diffeomorphism such that $\phi_s$ is $\varphi$-transverse to itself. Then, there exists $t_0>0$  and a function $\delta\colon [t_0,\infty) \to \RR_{>0}$ with $\delta(t)\to 0$ as $t \to \infty$ such that for every $t>t_0$ one has that the diffeomorphism $f_t = \phi_t \circ \varphi \circ \phi_t$ verifies:
\begin{enumerate}[label = (\roman*)]
\item\label{it.BGHP1} $f_t$ is partially hyperbolic and the bundles $E^s_t, E^c_t$ and $E^u_t$ of $f_t$ make an angle less than $\delta(t)$ with the bundles $E^s_\phi$, $\RR X$ and $E^u_\phi$ respectively;
\item\label{it.BGHP2} for every immersed\footnote{We require the speed $c'$ of $c$ to be uniformly away from $0$ and $\infty$.} curve $c\colon \RR \to M$ everywhere tangent to $E^c_t$ there exists $x \in M$ and a homeomorphism $u\colon \RR \to \RR$ such that $d(c(u(s)), \phi_s(x)) < \delta(t)$ for every $s \in \RR$, moreover, the point $x$ is unique in that if $y$ verifies the same, then $y = \phi_s(x)$ for some $s \in \RR$;
\item\label{it.BGHP3} for every $x \in M$ there is an immersed curve $c\colon \RR \to M$ everywhere tangent to $E^c_t$ such that $d(c(s), \phi_s(x)) < \delta(t)$ for every $s \in \RR$. 
\end{enumerate}
\end{prop}

\begin{proof}
Item \ref{it.BGHP1} is a direct consequence of \cite[Proposition 2.4]{BGHP} and \cite[Remark 2.6]{BGHP}.

Item \ref{it.BGHP2} follows from the standard shadowing lemma for Anosov flows (see e.g, \cite[Theorem 5.3]{BGHP}) and item \ref{it.BGHP3} from its global version (cf.~\cite[Theorem 5.5]{BGHP}). Note that \cite[Theorem 5.5]{BGHP} is stated for flows, so to do this we apply the trick in \cite[Proposition 5.11]{BGHP}: we lift to a finite cover, take an iterate, so that we can apply Theorem \ref{BI} to get branching foliations. Using Theorem \ref{teo-openbranch} we construct a flow whose orbits are arbitrarily close to curves tangent to the center, so we can apply \cite[Theorem 5.5]{BGHP} to this flow to get the center curve which then projects to $M$. 
\end{proof}

The strong information we get with the previous proposition allows us to show a result which answers positively Question \ref{q.branchingunique} in the setting of the examples of \cite{BGHP}. We first need the following: 

\begin{lemma}\label{claim.10.4}
 Consider $t> t_0$ and $f=f_t$ as in Proposition \ref{p.BGHP}. Let $c_s, s \in [0,1]$ be a continuous one parameter family of complete curves tangent to $E^c=E^c_t$ in $\mt$ such that the Hausdorff distance between any of them is finite and bounded and that all curves are
contained in a single leaf of either $\widetilde{\cW^{cs}}$
or $\widetilde{\cW^{cu}}$.
Then all curves $c_s$ coincide.
\end{lemma}

\begin{proof}
The proof is by contradiction, assuming that the curves do not
all coincide.
We do the proof for the case of all curves contained in a single
center stable leaf.
From Proposition \ref{p.BGHP} we see that for each $s$ there is 
a unique orbit $\nu_s$ of $\widetilde \phi_t$ which is $\delta=\delta_t$
close to $c_s$. The orbits $\nu_s$ vary continuously with $s$.
In addition the orbits $\nu_s$ are all a bounded Hausdorff
distance from each other.

We first prove that the flow lines $\nu_s$ are in fact a single flow line.
Suppose that this is not true. Then either $\nu_s$ are not contained
in a single leaf of $\widetilde{\cF^{ws}_\phi}$ or
not contained in a single leaf of $\widetilde{\cF^{wu}_\phi}$. 
Assume that the $\nu_s$ are not 
in a single leaf of $\widetilde{\cF^{wu}_\phi}$.
In particular, for any open neighborhood $J$ of $s_0$ in
$[0,1]$ and $s \in J$ the curves $\nu_s$ are not all in $\widetilde{\cF^{wu}_\phi}(\nu_{s_0})$. For each $s$ in $J$ let 
$$\mu_s \  = \ \widetilde{\cF^{wu}_\phi}(\nu_s) 
\cap \widetilde{\cF^{ws}_\phi}(\nu_{s_0}).$$

\noindent
By hypothesis
not all $\mu_s$ are equal to $\mu_{s_0} = \nu_{s_0}$.
Fix $a_0 > 0$ and  a basepoint $x$ in $\mu_{s_0}$. 
Choosing $J$ sufficiently
small we can assume that for all $s \in J$,
there are backward rays of $\mu_s$ starting near $x$ which
are  all $a_0$ near corresponding backward rays of $\nu_s$.
Hence all backward rays of $\mu_s$ for $s$ in $J$ are a bounded
Hausdorff distance from each other. Since they are also in the same weak stable leaf, they have the same endpoint on the Gromov boundary of $\widetilde{\cF^{ws}_\phi}(\nu_{s_0})$. Hence, by Proposition \ref{prop.topAnosovQG}, all the $\mu_s$ must coincide, which contradicts the assumption.

Similarly, if the $\nu_s$ are not in a single leaf of 
$\widetilde{\cF^{ws}_\phi}$, then the same argument as above, switching stable and unstable, applies.
Thus the $\nu_s$ are all the same orbit of $\widetilde \phi_t$.

Now apply iterates of $f^{-1}$: The family $\{f^{-j}(c_s), s \in [0,1]\}$ is still continuous in $s$, in a single center stable leaf, and a bounded Hausdorff distance from each other. Thus, applying the preceding argument to that new family show that it must also be associated with a single orbit of $\wt \phi_t$.

However, since not all curves $c_s$ coincide, for $j$ big enough, the stable length between some curves of $\{f^{-j}(c_s), s \in [0,1]\}$ must be greater than $3\delta$. So item \ref{it.BGHP2} of Proposition \ref{p.BGHP} implies that these curves must be associated with distinct orbits of $\wt\phi_t$, contradicting the above.
\end{proof}

\begin{prop}\label{prop.uniquebranchingBGHP}
Let $\phi_s\colon M \to M$ be an Anosov flow generated by a vector field $X$ and $\varphi\colon M \to M$ a diffeomorphism such that $\phi_s$ is $\varphi$-transverse to itself. Let $f_t=\phi_t \circ \varphi \circ \phi_t$. Assume that the invariant bundles of $f_t$ are orientable and that $t$ is sufficiently large. Then, $f_t$ is a leaf space collapsed Anosov flow and there is a unique $f_t$-invariant branching foliation tangent to $E^{cs}_t$ and a unique $f_t$-invariant branching foliation tangent to $E^{cu}_t$. 
\end{prop}

\begin{proof}
By taking $g = f_t^k$ we can assume that $g$ preserves the orientation of the bundles. So, by Theorem \ref{BI}, there are $g$-invariant branching foliations $\cs$ and $\cu$ tangent respectively to $E^{cs}_t$ and $E^{cu}_t$.

By Proposition \ref{p.BGHP} \ref{it.BGHP3} for any orbit
$o_x$ of $\wt \phi_t$
there is a center leaf $c_x$ which is $\delta$-close
to $o_x$. To show uniqueness of $c_x$, suppose that there
are two distinct $c_x, c'_x$ which are $\delta$ close to $o_x$.
Let $c_x = L_x \cap F_x$ and $c'_x = L'_x \cap F'_x$,
with $L_x, L'_x \in \wt{\cW^{cs}}$, 
$F_x, F'_x \in \wt{\cW^{cu}}$.
We want to show that $L_x \cap F'_x = L'_x \cap F_x = c_x$
which implies that $c_x = c'_x$. We argue for $L_x \cap F'_x$
since the other case is similar. Suppose then that 
$L_x \cap F'_x = c^{''}_x \not = c_x$. Then the center leaves
in $L_x$ between $c_x$ and $c^{''}_x$ form an interval of curves
tangent to $E^c$ in $L_x$ which vary continuously and are
all a bounded distance from each other in $L_x$.
This contradicts Lemma \ref{claim.10.4}.
This shows that there is a bijection between the orbit
space of $\wt \phi_t$ and the center leaf space of $g$.

The invariance of the bijection under the action of $\pi_1(M)$ comes from the fact that Proposition \ref{p.BGHP} \ref{it.BGHP3} is done in $M$ and the continuity follows from the fact that the leaves of the foliations vary continuously in compact sets. This shows that $g$ is a leaf space collapsed Anosov flow (and therefore is also quasigeodesic partially hyperbolic by Theorem \ref{teo.main5}). 

To complete the proof we must show that the ($g$-invariant) branching foliations tangent to $E^{cs}_t$ and $E^{cu}_t$ respectively are unique. Note that, if we show that these are the unique $g$-invariant branching foliations for $g$, then (since $f_t(\cs)$ and $f_t(\cu)$ are $g$-invariant branching foliations) we will deduce that $f_t(\cs)=\cs$ and $f_t(\cu)= \cu$. 

We deal with $E^{cs}_t$. Assume there is a pair of $g$-invariant branching foliations $\cs_1$ and $\cs_2$ tangent to $E^{cs}_t$ and let $\cu$ be one $g$-invariant branching foliation tangent to $E^{cu}_t$.

Note that applying the constructions above to the pairs $(\cs_1,\cu)$ and $(\cs_2,\cu)$ we get the structure of a leaf space collapsed Anosov flow for $g$ in two different ways. If $\cs_1$ and $\cs_2$ are not equal the following
happens: there is a leaf $L_1$ of $\wt{\cW^{cs}_1}$ which is not a leaf of 
$\wt{\cW^{cs}_2}$. We work with the maps between the leaf spaces
and the orbit space of $\tilde{\phi_s}$, because of the leaf
space collapsed Anosov flow structure.
By Proposition \ref{p.wcaf2impliescaf2}, $L_1$ is associated with
a weak stable leaf $E$ of $\tilde{\phi_s}$. By the same proposition
the weak stable 
$E$ is associated with a leaf $L_2$ of $\wt{\cW^{cs}_2}$. 
Since $L_1$ is not a leaf of $\wt{\cW^{cs}_2}$ there is a center leaf
$c_0$ in $L_1$ such that the corresponding center leaf $c^*$ under
these identifications is not contained in $L_1$. In
other words $c_0, c^*$ are distinct curves tangent
to the center bundle, but associated with the same orbit
$o_x$ of $\widetilde{\phi_s}$.
If this is the case, we still know by Proposition \ref{p.BGHP} \ref{it.BGHP2} that $c_0, c^*$ are at distance less than $2\delta$ from each other.

Let $U$ be a center unstable leaf so that $U \cap L_1 = c_0$.
Then the stable saturation of $c^*$ is not contained in $L_1$ since
otherwise using that
$c^*$ is $2 \delta$ close to $c_0$ we would get that $c^*$ is
contained in $L_1$.
Then the stable saturation of $c^*$ intersects $U$ in a curve
$c$ which is less than say $4 \delta$ from $c_0$.
Notice that $c_0, c$ are tangent to the center bundle
and distinct.

We want to produce a continuous collection of curves
tangent to the center bundle, which are a bounded Hausdorff distance
from each other, so that we can apply Lemma \ref{claim.10.4} and
derive a contradiction. 
We fix a transversal orientation to $c_0$ in $U$.
This induces a transversal orientation to the center bundle
in $U$.
First we define two curves with respect to this
transverse orientation: 
$$c^+:= \sup(c_0,c), \quad c^-:= \inf(c_0,c).$$

\noindent where the supremum and infimum are taken with respect to this chosen transverse orientation where the curves are locally graphs.

If $c_0, c$ never intersect, then
$c$ is in a complementary component of $c_0$ in $U$. If say
$c$ is in the plus component, then $c^+ = c, \ c^- = c_0$,
if $c$ is in the minus component, then $c^+ = c_0, \ c^- = c$.
If $c_0, c$ intersect, notice that at each point of intersection
they are tangent, as they are both tangent to the center bundle.
For each component of $c \smallsetminus c_0$, it is either in the plus 
complementary component of $c_0$ or in the minus complementary
component, and we define the parts of $c^+, c^-$ in those
parts as in the situation where the curves are disjoint.

Note the following important facts: $c^+, c^-$ are a bounded
Hausdorff distance $4 \delta$ from each other and from $c_0$;
$c^+$ is contained in the closure of the plus complementary
component of $c_0$; and $c^-$ is contained in the closure of the
minus complementary component of $c_0$. This in particular implies
that $c^+, c^-$ do not topologically cross.

We are going to define the continuous family of curves tangent
to $E^c$ in $U$. First parametrize the center leaves of $\cW^{cs} \cap U$
as $c_s$, where $c_0$ was already defined, and $s > 0$ is on the
closure of the plus side of $c_0$ in $U$.
For each $s \geq 0$, define 

$$d_s := \inf(c_s,c^+)$$

\noindent
Notice that for each $s \geq 0$, we have that $c_s, c^+$ are both contained
in the closure of the plus complementary component of $c_0$.
Since $c_s$ varies continuously with $t$, then so does $d_s$ for $s \geq
0$. Also  $d_0 = c_0$. And finally all curves $d_s$ are
between $c^+$ and $c_0$. So they are all a bounded distance
from $d_0$.
Similarly for $s \leq 0$ define $d_s := \sup(c_s, c^-)$.
They have the same properties as $d_s$ for $s \geq 0$.
Notice also that $c_s$ escapes in $U$ for $|s| \to \infty$, that is, for every compact set $C \subset U$ for large $s$ the curve $c_s$ is not contained in $C$. 
By construction $c^+, c^-$ are distinct curves.
It now follows that $d_s$ cannot be constant with
$s$. So we have a continuous family of curves tangent
to $E^c$ in a center unstable leaf, which are all
a bounded Hausdorff distance from each other.
This contradicts Lemma \ref{claim.10.4}.

This finishes the proof of Proposition \ref{prop.uniquebranchingBGHP}.
\end{proof}

Now we can use the previous proposition to deduce Theorem \ref{teo.main1}.

\begin{proof}[Proof of Theorem \ref{teo.main1}]
Consider $t>0$ large enough so that both Proposition \ref{p.BGHP} and Proposition \ref{prop.uniquebranchingBGHP} hold.

We can choose a finite normal cover $P: \hat M \to M$ such that the lifts of all bundles are orientable. An iterate of $f_t$ lifts to $\hat M$ and we can consider a lift $g$ of a possibly further iterate so that $g$
preserves the orientation of the bundles. Applying Proposition \ref{prop.uniquebranchingBGHP} to $g$ we get that $g$ is a leaf space collapsed Anosov flow and that it admits a unique pair of $g$-invariant
branching foliations $\cs_0$ and $\cu_0$ tangent to $E^{cs}$ and $E^{cu}$ respectively. 

As explained in Remark \ref{r.specialbranch} using the uniqueness
of branching foliations, we obtain that $\cs_0$ (and $\cu_0$) must coincide with the uppermost and lowermost branching foliations constructed in \cite{BI}. More specifically the uppermost center stable foliation is
the same as the lowermost center stable foliation.
This implies that these branching foliations project to $M$ since given $\gamma$ a deck transformation of $\hat M$ with respect to the cover $P$ we get that it preserves the bundles, so it verifies that $\gamma \cs_0$ and $\gamma \cu_0$ are branching foliations tangent to $E^{cs}$ and $E^{cu}$ and depending on how $\gamma$ acts on the orientation it preserves the uppermost branching foliation or it maps it into the lowermost one. Since these are equal by Proposition \ref{prop.uniquebranchingBGHP} we deduce $\gamma \cs_0 =\cs_0$ and $\gamma \cu_0 = \cu_0$.

Now denote by $\cs, \cu$ the projection of these branching foliations to $M$. 
Let $\cB$ be the lift of $f_t(\cs)$ to $\hat M$.
Since $f^k_t$ lifts to $g$ in $\hat M$, it follows that
$f^k_t(\cs) = \cs$. The foliation $g(\cB)$ projects to 
$f^k_t \circ f_t(\cs)$,
which is then equal to $f_t(\cs)$, so $g(\cB) = \cB$.
Then the uniqueness of branching foliations in $\hat M$ implies
that $\cB = \cs_0$, and we finally conclude that
$f_t(\cs) = \cs$.

Hence $f_t$ preserves branching foliations $\cs, \cu$ and $f_t$
is also a leaf space collapsed Anosov flow. 
By Theorem \ref{teo.main5} we get that $f_t$ is also a quasigeodesic
partially hyperbolic diffeomorphism.

\vskip .1in
To show that $f_t$ is a strong collapsed Anosov flow, we point out to the proof of Theorem \ref{teo.main3} in \S\ref{s.leafspace_implies_strong}.

Theorem \ref{teo.main1} states that $f_t$ is also a strong collapsed Anosov flow. However, since we do not assume that the bundles are orientable, we cannot use Theorem \ref{teo.main3} directly to deduce this. Instead, we redo and adapt some of the steps of the proof of Theorem \ref{teo.main3} in \S\ref{s.leafspace_implies_strong} to these particular examples.

 In \S\ref{s.leafspace_implies_strong} we constructed a map which sent an orbit of an Anosov flow $\psi_t$ which was orbit equivalent to the original Anosov flow $\hat \phi_t$ to a curve tangent to $E^c$. This worked fine under orientability assumptions, so we get a map $\hat h\colon \hat M \to \hat M$ with these properties. Our goal is to show that we can project that map to $M$.
 
 We consider an orbit equivalence $\hat k \colon \hat M \to \hat M$ from the flow $\hat{\phi}_t \colon \hat M \to \hat M$ (the lift of $\phi_t$ to $\hat M$)  to the flow $\psi_t$ constructed in 
\S \ref{s.leafspace_implies_strong}.
We let $\hat h_0 = \hat h \circ \hat k^{-1}$ which maps orbits of $\hat \phi_t$ to curves tangent to the centers. If we consider a deck transformation $\gamma$ with respect to $P\colon \hat M \to M$ and an orbit $o_x$ of $\hat \phi_t$ we claim that 
$$\hat h_0 (\gamma o_x) =\gamma \hat h_0 (o_x).$$

Indeed, by construction, for any orbit $o$, $\hat h_0 (o)$ is the unique curve tangent to $E^c$ which is $\delta$ near $o$. Now $\gamma \hat h_0(o_x)$ is a curve tangent to $E^c$ which is $\delta$ near the orbit $\gamma (o_x)$. Hence, the above formula must hold.

Using this we can prove that we can make a quotient map of $\hat h_0$
to $M$. 
Given a center leaf $c$ in $M$ we say that $c$ is closed
if given a lift $\tilde{c}$ in $\widetilde M$, there is
a non trivial deck transformation $\alpha$ such that
$\alpha(\tilde{c}) = \tilde c$. We have already proved that
$f_t$ is a leaf space collapsed Anosov flow,
which implies that $c$ is closed
if and only if it is associated with a closed orbit of $\phi_s$. 
Let $\gamma_1, \ldots, \gamma_n$ be the deck transformations of
the cover $\hat M \to M$. Given $y$ a point in a  non periodic 
orbit of $\phi_s$, let $x_1, \ldots, x_n$ be the lifts of $y$ 
to $\hat M$,
which are related by the $\{ \gamma_i \}$.
We consider the center leaves in $M$ or $\hat M$ which are
not closed, or equivalently the non periodic orbits of
$\phi_s$ or $\hat \phi_s$. So given $y$, there are finitely
many $x_i$. For each $x_i$, we compute $\hat h_0(x_i)$,
which by the formula above projects by $P$ to the same
center leaf in $M$.
In this center leaf there is an induced metric given by 
length along the centers. This metric induces an identification
with $\mathbb{R}$. Using this identification, we can compute
the average of $P(\hat h_0(x_i))$ for $1 \leq i \leq n$.
Let $h_0(y)$ be this average.
Note that we have used that the center leaf is not closed, as otherwise it is more complicated to take averages.

Now we use the following properties: There are finitely many $\gamma_i$, the length along center leaves varies continuously, and $\hat h_0$ is continuous on the non
periodic center leaves. These properties imply that 
this function extends to a continuous
function in all of $M$.

We now obtained the collapsing function $h_0$ sending orbits
of $\phi_s$ to curves tangent to $E^c$ in $M$. Finally we need
to construct the self orbit equivalence $\beta$ to satisfy
$f_t \circ h_0 = h_0 \circ \beta$.
The construction is now exactly as in the end of the proof of Proposition \ref{prop-2implies1} since no orientation is needed then. 
This shows that $f_t$ is a strong collapsed Anosov flow.
\end{proof}

\begin{remark} \label{rem:thmA}
Notice that, in the proof of Theorem \ref{teo.main1} (more precisely, in Proposition \ref{prop.uniquebranchingBGHP}), the time $t_1$ we require to have so that $f_t$ is leaf space collapsed Anosov flow for all $t>t_1$ may be greater than the time $t_0$ required so that $f_t$ is a partially hyperbolic diffeomorphism for all $t>t_0$.

Hence, Theorem \ref{teo.main1} does not directly say that \emph{all} the examples \emph{\`a la} \cite{BGHP} (meaning all examples proven to be partially hyperbolic using Proposition \ref{prop.BGHP_example}) are (leaf space) collapsed Anosov flows.

However, since $f_t$ is partially hyperbolic for all $t>t_0$ and leaf space collapsed Anosov flow for all $t>t_1\geq t_0$, Theorem  \ref{teo.main6} implies that all $f_t$, $t>t_0$ are indeed leaf space collapsed Anosov flows.
\end{remark}

\subsection{Uniqueness of curves tangent to the center bundle}\label{ss.uniqueness_center_curves}
In this section we show that under some uniqueness properties of the branching foliations like the ones obtained in Proposition \ref{prop.uniquebranchingBGHP} we can deduce a stronger form of uniqueness of integrability of the center bundle. This also motivates Question \ref{q.branchingunique} as a way to understand finer geometric properties of the center bundle beyond the fact that it can help to remove orientability assumptions in our results. 

We first prove a general fact about quasigeodesic partially hyperbolic diffeomorphisms that may be of interest and which essentially states that the center direction inside center stable (or center unstable) leaves is a semi-flow (i.e., it can only branch in one direction). 

\begin{lemma}\label{l.semiflow} Suppose that $f$ is a quasigeodesic partially hyperbolic diffeomorphism with branching foliations $\cs$ and $\cu$.
Given $L$ a leaf of $\wcs$ suppose that
two center leaves $c_1, c_2$ in $L$ intersect in $x$.
Then $c_1, c_2$ coincide in the ray from $x$ to the funnel
point in $L$. The symmetric statement holds for leaves in $\wcu$. 
\end{lemma}

\begin{proof} Suppose this is not the case.
There are two options: 
\begin{enumerate}[label=(\roman*)] 
\item\label{finitebigon} There are $y, z$ in the
ray of $c_1$ to the funnel point, so that both
belong to the intersection of $c_1, c_2$ but no
point in the segment of $c_1$ between them is in
$c_2$. This is called a finite bigon; \ or
\item\label{infinitebigon} There is $y$ in $c_1 \cap c_2$ so that the
ray in $c_1$ from $y$ to the funnel point is
disjoint from $c_2$. This is called an infinite bigon.
\end{enumerate}

We first show that option \ref{finitebigon} cannot happen. Let $B$ be the bigon
formed by the segments in $c_1 \cup c_2$ bounded
by $y, z$. Let $\ell_i$ be the segment in $c_i$ 
from $y$ to $z$. Consider the negative iterate by $f$ of $B$: Since the stable lengths
converge to infinity, the diameters of $f^{-n}(B)$ goes to infinity as $n \to +\infty$. 
The curves $f^{-n}(\ell_i)$ are uniform quasigeodesics arcs
with same pair of endpoints, hence they are a uniform
bounded distance from each other. Consider points midway
in $f^{-n}(B)$: Up to subsequences and deck transformations
the two boundary center rays converge to distinct center
leaves in the same center stable leaf, and which have the
same ideal points. This is disallowed by Proposition \ref{prop.Unique}.

A similar argument rules out option \ref{infinitebigon} by considering the infinite bigon $B$ and taking points at increasing distance from the point where they intersect in the direction where they converge to the same point.  The same argument gives two center leaves which have the same ideal points. 

This proves the lemma.
\end{proof}

We can use this to get a precise description of curves tangent to $E^c$ assuming uniqueness of branching foliations. 

\begin{prop}\label{prop.uniquecenters}
Suppose that $f$ is a quasigeodesic partially hyperbolic
diffeomorphism such that all the bundles are
orientable and $f$ preserves the orientations.
Suppose that there is a unique pair of center stable and center unstable branching foliations that are invariant by $f$.
Then any curve in $M$ which is  tangent to $E^c$ is 
the 
intersection of a center stable and a center unstable leaf.
\end{prop}

\begin{proof}
Let $\cs, \cu$ be branching foliations given
by Theorem \ref{BI}.
As explained in Proposition \ref{topmost}, two natural $f$-invariant 
branching foliations tangent to
$E^{cs}$ are constructed in \cite{BI}: The lowermost one in the positive center direction,
and the uppermost one. By hypothesis, these two branching foliations must coincide.

Orient the center bundle to be positive in the center
stable funnel direction.
Now suppose that $c$ is a curve in $\mt$ tangent to $E^c$. 
Let $x$ be a point in $c$. Consider a ray $r$ in $c$ starting at $x$ and 
in the positive direction.
Suppose that $x$ is in a center stable leaf $U$.

\begin{claim}\label{claim.10.7}
The ray $r$ is contained in $U$.
\end{claim}
\begin{proof}
Consider $U_1, U_2$ the  uppermost and lowermost center
unstable leaves  of $\wcu$ through $x$.
Since we assumed that center stable and center unstable branching foliations are unique, $\cu$ is both the lowermost and uppermost branching foliation of \cite{BI} (see Appendix \ref{app.branching}). In particular, this implies that $U_2$ is the lowermost
local center unstable surface from $x$ in the positive
center direction as constructed by Burago--Ivanov
in \cite{BI}. Similarly $U_1$ is the uppermost local
center unstable surface through $x$.
If one does a local saturation $S$  of $c$ through stable 
leaves, then \cite[Lemma 3.1]{BI} shows that $S$ is a
$C^1$ surface tangent to $E^{cu}$. In particular $S$
is locally between $U_1$ and $U_2$.
Let $L$ be a center stable leaf containing $c$. Let 
$c_i = U_i \cap L$. 

Now $c_i$ are center leaves in $L$
both through $x$. Lemma \ref{l.semiflow} shows that
the rays of $c_i$ starting at $x$ and in the positive 
direction coincide. In particular $U_1, U_2$ coincide
locally near $x$ and so does $U$. Hence $r$ is locally
contained in $U$.

This situation has a uniformity: there is fixed $\eps_0 > 0$
so that one can always get a segment of length $\eps_0$ in
$r$ contained in $U$.
This yields a point $x_1$ in $r$ at least $\eps_0$ 
along $r$ from $x$. 
Notice that $x_1$ is in every center unstable leaf
in $[U_1,U_2]$.
Now restart with $x_1$. Get $U^1_1, U^1_2$ the uppermost
and lowermost leaves of $\wcs$ through $x_1$.  Notice that
the intervals 
$[U_1,U_2] \subset [U^1_1,U^1_2]$.
Apply the same argument for a length $\geq \eps_0$ along $r$
to get second segment in $r$ now contained in every leaf
in $[U^1_1,U^1_2]$ and hence in $U$.
Then iterate, obtaining 
points $x_j$  in $r$ escaping in $r$.
This proves the claim.
\end{proof}

Now we prove that there is a $\wcu$ leaf that contains all of $c$.
Let $p_0 = x$. For each $i$, we choose a point $p_i$ in $c$ that is a distance along $c$ at least $1$ from $p_{i-1}$. We choose the sequence so that the $p_i$ escapes in the  direction opposite to the funnel. This direction is opposite to where the points $x_j$ were.
Let $U_i$ a $\wcu$ leaf with $p_i$ in
$U_i$. Let $r_i = [p_i,+\infty)$ be the ray of $c$ starting in $p_i$
and going in the direction of the funnel. 
By Claim \ref{claim.10.7} the entire ray
$r_i$ is contained in $U_i$. All $U_i$ contain $p_0$. 
The set of $\wcu$ leaves through $p_0$ is a compact interval.
Up to taking a subsequence, assume that $U_i$ converges to a
leaf $V$ as $i \to \infty$.
Then since all $U_i$ for $i \geq j$ contain $p_j$ then
$V$ contains $p_j$. Hence $V$ contains all the $p_i$'s.
By the claim then
$V$ contains the entire curve $c$.

By the same arguments $c$ is contained
in a  leaf $E$ of $\wcs$. 
This finishes the proof of the proposition.
\end{proof}

\begin{remark}
Note that in the case of the examples obtained via Theorem \ref{teo.main1} we are able to get that for large enough $t>0$ the diffeomorphism $f_t$ when lifted to a finite cover satisfies
the hypothesis of Proposition \ref{prop.uniquebranchingBGHP}. Hence,  
Proposition \ref{prop.uniquecenters}
can be applied to deduce that every curve tangent to $E^c$ is obtained (in $\mt$) as the intersection of a center stable and a center unstable leaf of the branching foliations.
This is a form of unique integrability of the center
bundle, even if different center curves may merge.
Note in particular that if $f_t$ is dynamically coherent, this implies that $E^c$ is uniquely integrable as a bundle. 
In particular notice the difference: one can prove that
$f_t$ is partially hyperbolic for all $t \geq t_1$, but
to get the unique integrability of the center bundle as above
one needs $t \geq t_0$, where in theory $t_0 > t_1$.

We also note that the property of not having unique $f$-invariant center stable or center unstable branching foliations is an open property among partially hyperbolic diffeomorphisms thanks to Theorem \ref{teo.graphtransform}. The closed property may fail because in the limit different branching 
foliations may collapse to a single branching foliation.

However as a direct consequence of Theorem \ref{teo.main6} we get
the following: in the connected component of partially hyperbolic diffeomorphisms containing some $f_t$ we have that $f_t$ has to be a collapsed Anosov flow with respect to the same flow and same self orbit equivalence of
the flow (same in terms of the action on the orbit spaces), for every pair of branching foliations it may have\footnote{Technically to get this one needs to show that having branching foliations for which $f$ is not a collapsed Anosov flow is also an open and closed property, but this follows directly from Theorem \ref{teo.graphtransform}.}. 
\end{remark}

\begin{remark}
The previous remark applies very well to the case of partially hyperbolic diffeomorphisms in the connected component of the time one map of an Anosov flow. Here, by Theorem \ref{teo.main6} the whole connected component of partially hyperbolic diffeomorphisms 
consists of discretized Anosov flows. 
This uses the last part of the previous remark as well as
Proposition \ref{p.daf2}.
Moreover, since the Anosov flow is generated by a $C^1$ vector field, the center direction of its time one map is uniquely integrable. It follows that in the whole connected component of partially hyperbolic diffeomorphisms, if there were more than one pair of branching foliations, these should correspond to discretized Anosov flows $-$ again by the last part
of the previous remark. 
But in \cite[Lemma 7.6]{BFFP_part2} using that they are discretized Anosov flows,
 we showed that this implies that there is a unique pair of branching foliations. As a consequence we obtain that the center direction is uniquely integrable (since in addition
it integrates to a foliation) in the whole connected component of partially hyperbolic diffeomorphisms containing the time one map of an Anosov flow\footnote{Or, maybe more generally, a discretized Anosov flow for which the center direction is uniquely integrable.}. 
\end{remark}

\subsection{$C^1$ self orbit equivalences and collapsed Anosov flows} \label{ss.C1_soes}

Thanks to the concept of $\varphi$-transversality of \cite{BGHP} and Theorem \ref{teo.main1}, we can readily obtain many collapsed Anosov flows. However, finding a map $\varphi$ for which a flow is $\varphi$-transverse to itself is generally not easy (see \cite{BGHP}). But one instance when it \emph{is} easy is when one has a map $\beta$, which is a (at least) $C^1$ self orbit equivalence of a smooth Anosov flow $\phi$. Indeed, since $\beta$ preserves the weak stable and unstable directions and preserves the flow direction, the flow is trivially $\beta$-transverse to itself (see Definition \ref{def.phi-transversality}).

Hence, for such a $\beta$, the map $\phi_t\circ \beta \circ \phi_t$ is a collapsed Anosov flow of $\phi$ thanks to Theorem \ref{teo.main1}, and it is clearly dynamically coherent as it preserves the weak stable and weak unstable foliations of $\phi$.

The first such examples were constructed in \cite{BW}, but these examples are such that a power is a discretized Anosov flow.

One can wonder whether different smooth self orbit equivalences could lead to genuinely new collapsed Anosov flows (that is, ones such that no power is a discretized Anosov flow).
It turns out that, at least when the Anosov flow is transitive, this is not the case, as we observe a form of smooth rigidity:
\begin{prop}\label{prop.C1soe}
 Let $\beta$ be a $C^1$ self orbit equivalence of a smooth (at least $C^1$) transitive Anosov flow. Then there exists $k$ such that $\beta^k$ is a trivial self orbit equivalence. (Moreover there is an upper bound for $k$ that only depends on the flow and the manifold.)
\end{prop}

\begin{proof}
In the proof of \cite[Proposition 6.6]{BarbotHDR}, Barbot shows that if a map $\wt\beta_{\mathcal O}$ on the orbit space of a smooth Anosov flow $\phi$ is a $\pi_1$-equivariant, $C^1$ diffeomorphism, then there exist a time-change $\psi$ of $\phi$ such that $\beta$ is a conjugation of $\psi$ with itself, where $\beta \colon M \to M$ is a $C^1$-map such that its lift to the orbit space is $\wt\beta_{\mathcal O}$ (or $\wt\beta_{\mathcal O}^2$ if $\wt\beta_{\mathcal O}$ reverses the direction of the flow). (The flow $\psi$ is build on the projectivized bundle of the orbit space, see also \cite{BarbotFenley_new_contact_Anosov}).

In other words, $\beta$ is in the centralizer of $\psi$. By \cite[Lemma 1.4]{BarthelmeGogolev_centralizers}, the centralizer of $\psi$ quotiented out by the elements of the centralizer that act as the identity on the orbit space is finite. Hence, there exists $k$, which can be chosen depending only on the flow and the manifold, such that $\beta^k$ is trivial.
\end{proof}

We end this section with some comments regarding question \ref{q.realization_soe}: If a self orbit equivalence $\beta$ of an Anosov flow $\phi_t$ is smooth, then it follows that one can construct a collapsed Anosov flow with the technique of \cite{BGHP} by taking $\phi_t \circ \beta \circ \phi_t$ with large $t$. In general, the previous proposition indicates that we cannot usually expect the self orbit equivalence to be smooth, therefore, we cannot apply this technique directly. However, it is reasonable to expect that self orbit equivalences can be smoothed in order to preserve some transversality between bundles which would give a way to attack question \ref{q.realization_soe}.

\section{Some classification results}
\label{s.classification_results} 

In this section we will present some relatively direct results giving settings where one can use self orbit equivalences to classify all collapsed Anosov flows or vice-versa. The three settings we will describe are: Collapsed Anosov flows that are homotopic to the identity, Collapsed Anosov flows on $T^1S$, the unit tangent bundle of a hyperbolic surface, and Collapsed Anosov flows associated with the Franks--Williams example.

Those are not the only cases where one can obtain such a complete understanding, but they are among the easiest and nicely showcase the type of tools one has to prove such results.

We emphasize that the result below gives a complete picture of self orbit equivalences of certain Anosov flows, but only a classification up to isotopy for collapsed Anosov flows, as we do not yet know how different two collapsed Anosov flows associated with the same self orbit equivalence can be.

\subsection{The homotopic to the identity case}

In \cite{BaG}, self orbit equivalences of transitive Anosov flows that are homotopic to the identity were completely classified.
Thus, we can translate \cite[Theorem 1.1]{BaG}, using Proposition \ref{p.daf2}, in terms of collapsed Anosov flow to obtain the following.

\begin{teo}\label{teo.CAF_homotopic_to_identity}
If $f$ is a strong collapsed Anosov flow homotopic to the identity associated to a transitive Anosov flow $\phi^t$, then $f$ is either a discretized Anosov flow or a \emph{double translation} in the sense of \cite{BFFP_part2}.

Moreover, if the associated Anosov flow $\phi_t$ is either not $\RR$-covered, or has non transversely-orientable weak foliations, then $f$ must be a discretized Anosov flow.
\end{teo}

Note that it is still unknown whether double translations exists or not outside of Seifert manifolds, but in \cite{FP-2} the second and third authors show that any double translation on a hyperbolic manifold must be a collapsed Anosov flow associated with the ``one-step up'' self orbit equivalence of an $\RR$-covered Anosov flow.

\begin{proof}
 Let $f$ be a strong collapsed Anosov flow that is homotopic to the identity, associated with a transitive Anosov flow $\phi_t$. Let $h$ and $\beta$ be the associated collapsing map and self orbit equivalence. Since $f \circ h = h \circ \beta$ and both $f$ and $h$ are homotopic to the identity, we deduce that $\beta$ is also homotopic to the identity. Thus we can apply \cite[Theorem 1.1]{BaG}.
 
 If the flow $\phi_t$ is not $\RR$-covered or has non transversely-orientable weak foliations, then item (1) and (3), respectively, of \cite[Theorem 1.1]{BaG} implies that $\beta$ is trivial, thus $f$ is a discretized Anosov flow thanks to Proposition \ref{p.daf2}.
 
 If $\phi_t$ is $\RR$-covered, then item (4) of \cite[Theorem 1.1]{BaG} gives that either $\beta$ is trivial, which gives that $f$ is a discretized Anosov flow, or that $\beta$ is a power of the ``one-step up'' self orbit equivalence $\eta$. We will not recall what $\eta$ is exactly, just that a good lift of it acts as a translation on both leaf spaces of $\phi_t$. Let $\wt f$ be a lift of $f$ to the universal cover obtained from lifting an homotopy to the identity. Since $f$ is a strong collapsed Anosov flow, it admits center stable and center unstable branching foliations that are the images by $h$ of the weak stable and weak unstable foliations of $\phi_t$. Hence, a lift $\wt h$ realizes a semi-conjugacy between the action of $\wt\beta$ and the action of $\wt f$ on the respective leaves spaces. Since $\wt \beta$ acts as a translation, so does $\wt f$. So $f$ is a double translation in the sense of \cite{BFFP_part2}. 
\end{proof}

\subsection{Unit tangent bundle of surfaces}

When considering unit tangent bundle of surfaces, it is also possible to give a complete picture of collapsed Anosov flows, at least up to isotopy.

\begin{teo}\label{teo.CAFonT1S}
 Let $T^1S$ be the unit tangent bundle of a hyperbolic surface $S$.
    
 Let $f$ be a collapsed Anosov flow on $T^1S$ with associated flow $\phi$. 
 Then the isotopy class of $f$ is in a lift of $\mathrm{MCG}(S)$ to $\mathrm{MCG}(T^1S)$.
 More precisely, let $g^t$ be the geodesic flow on $T^1S$ for a fixed hyperbolic metric and $e\colon T^1S \to T^1S$ an orbit equivalence between $\phi$ and $g^t$. Let $\widehat{\mathrm{MCG}(S)} \subset \mathrm{MCG}(T^1S)$ be the lift of $\mathrm{MCG}(S)$ given by taking the derivative. Then the isotopy class of $f$ is inside $i_{[e]}\widehat{\mathrm{MCG}(S)}$, the conjugation of $\widehat{\mathrm{MCG}(S)}$ by the isotopy class of $e$.
 
 Moreover, any isotopy class in $i_{[e]}\widehat{\mathrm{MCG}(S)}$ admits a collapsed Anosov flow.
 
 The same statements hold for self orbit equivalences of $\phi$. 
\end{teo}

\begin{remark}
 Note that one can choose the orbit equivalence $e$ above such that it induces a map homotopic to the identity on $S$.
\end{remark}

\begin{remark}
 In \cite{BGHP} it is shown that the only partially hyperbolic diffeomorphisms on $T^1S$ that induce a map homotopic to the identity on $S$ must be homotopic to the identity on $T^1S$ too (that is, there is no non-trivial gauge transformations of $T^1S$ that admits a partially hyperbolic representative).
 
 As noted in \cite{BGHP}, this also implies that the isotopy classes of partially hyperbolic diffeomorphisms on $T^1S$ do not form a subgroup of $\mathrm{MCG}(T^1S)$.
 
 However, as we see here, this lack of a group structure is only because there are many Anosov flows on $T^1S$ that are orbit equivalent to the geodesic flow, but not via an orbit equivalence that is homotopic to identity. 
 Indeed, once an Anosov flow $\phi$ is fixed, the isotopy classes of collapsed Anosov flow associated with $\phi$ form a subgroup of $\mathrm{MCG}(T^1S)$.
\end{remark}

\begin{remark}
 In \cite{FP-2}, the second and third authors show that any partially hyperbolic diffeomorphism $f$ on $T^1 S$ that induces a map on $S$ that is homotopic to a pseudo-Anosov diffeomorphism is a strong collapsed Anosov flow (in fact a, quasigeodesic partially hyperbolic diffeomorphism). So the only cases that are not yet known to be collapsed Anosov flows on $T^1S$ are partially hyperbolic diffeomorphisms that act reducibly (but not trivially) on the base $S$.
\end{remark}

\begin{proof}

Let $\phi$ be an Anosov flow on $T^1M$ and $e\colon M \to M$ an homeomorphism such that $e^{-1} \circ \phi^t \circ e$ is a time-change of $g^t$.

In \cite[Theorem 1.2]{BGHP}, it was shown that any isotopy class in $\widehat{\mathrm{MCG}(S)}$ admits a partially hyperbolic diffeomorphism which is, according to Theorem \ref{teo.main1}, a collapsed Anosov flow associated with $g^t$. Hence, for any class in $\widehat{\mathrm{MCG}(S)}$, there exists a self orbit equivalence $\beta$ of $g^t$. Thus, $e\circ \beta \circ e^{-1}$ is a self orbit equivalence of $\phi^t$ and such self orbit equivalences will cover all of $i_{[e]}\widehat{\mathrm{MCG}(S)}$.

We can also build a collapsed Anosov flow using the same method, but we would need to require smoothness of $e$ which does not a priori hold. So instead, we let $\bar e$ be a diffeomorphism in the same isotopy class as $e$ and define $\bar\phi^t = \bar e \circ g^t \circ \bar e^{-1}$. Then, for any collapsed Anosov flow $f$ associated with $g^t$, the map $\bar e \circ f \circ \bar e^{-1}$ is a collapsed Anosov flow of $\bar\phi$ and the isotopy classes of such collapsed Anosov flow cover all of $i_{[e]}\widehat{\mathrm{MCG}(S)}$.

The other direction can be proven for instance as in \cite{Matsumoto}:
We can show that the isotopy class of a self orbit equivalence of $\phi$ is necessarily in $i_{[e]}\widehat{\mathrm{MCG}(S)}$. This will imply that the isotopy class of a collapsed Anosov flow must also necessarily be in $i_{[e]}\widehat{\mathrm{MCG}(S)}$.

If $\beta$ is a self orbit equivalence of $\phi$, then up to conjugation by $e$, we can assume that $\beta$ is a self orbit equivalence of $g^t$, and we have to show that $[\beta]\in \widehat{\mathrm{MCG}(S)}$. This follows as in the proof of \cite[Proposition 3.6]{Matsumoto} (see also \cite[Theorem 3.6]{BGHP} or \cite{BarbotFenley_new_contact_Anosov}). 
\end{proof}

\subsection{Collapsed Anosov flows of the Franks--Williams example}

The Franks--Williams \cite{FranksWilliams} example is the first, most famous and simplest non-transitive Anosov flow on a $3$-manifolds. We denote the Franks--Williams flow by $\phi_{FW}$ and by $M_{FW}$ the manifold supporting that flow. Note that $\phi_{FW}$ is the only non-transitive Anosov flow up to orbit equivalence on $M_{FW}$ (see \cite{YangYu}).
We will not recall the construction of $\phi_{FW}$ (see \cite{FranksWilliams} or, e.g., \cite{BBY}), but instead list the properties that we will use:
\begin{enumerate}[label = (\roman*)]
 \item \label{item.FW_2pieces} The manifold $M_{FW}$ decomposes into two atoroidal pieces separated by a torus $T$ transverse to $\phi_{FW}$ (which is unique up to isotopy along the flow lines). In particular, a consequence of Mostow's rigidity theorem is that the mapping class group of $M_{FW}$ is up to finite index generated by Dehn twists along the transverse tori (\cite[Corollary 27.6]{Johannson})\footnote{see e.g., \cite{BGHP} for the definition of a Dehn twist on a torus in a $3$-manifold}.
 \item \label{item.FW_torus} The stable and unstable foliations restrict to two transverse foliations on the transverse torus with four closed leaves (two stable and two unstable leaves) and Reeb components in between. We denote by $\alpha$ the element of $\pi_1(T)$ representing the closed leaves;
 \item \label{item.FW_periodic_orbits} Each periodic orbit of $\phi_{FW}$ is unique in its free homotopy class, except for the four periodic orbits (two in each atoroidal pieces) associated with the closed leaves of $T$ which are pairwise freely homotopic.
\end{enumerate}

\begin{teo}\label{teo-SOE_on_FranksWilliams}
 Up to finite power, any self orbit equivalence or collapsed Anosov flow of $\phi_{FW}$ is in the isotopy class of the power of a Dehn twist of $T$ in the direction of $\alpha$.
 Moreover, up to a finite power, two self orbit equivalences of $\phi_{FW}$ in the same isotopy class are equivalent.

 Conversely, any such isotopy classes can be realized by a collapsed Anosov flow or self orbit equivalence of $\phi_{FW}$.
\end{teo}

\begin{remark}
One can show that two self orbit equivalences in the same isotopy class are equivalent without taking a finite power, but the proof is easiest when allowing finite powers and we leave the more precise statement for a future general study of self orbit equivalences.
\end{remark}

\begin{remark}
 A cosmetic adaptation of the following proof allows to more generally classify collapsed Anosov flows and self orbit equivalences of Anosov flows that are obtained in the following way: Create any number of hyperbolic plugs (in the sense of \cite{BBY}) by doing a derived from Anosov construction on finitely many orbits of a suspension of an Anosov diffeomorphism of the torus. Glue the hyperbolic plugs together in any of the ways allowed to get a (transitive or non-transitive) Anosov flow (see \cite{BBY}).
 
 Such Anosov flows will satisfy a version of each of the items \ref{item.FW_2pieces}, \ref{item.FW_torus}, and \ref{item.FW_periodic_orbits} above. That is, the JSJ decomposition of the manifold has only atoroidal pieces, each torus is transverse to the flow with two closed center leaves and Reeb components for each of the weak foliations restricted to the torus, and every periodic orbits aside from finitely many will be alone in their free homotopy class.\footnote{To show that a periodic orbit $\gamma$ crossing one transverse torus $T$ is also alone in its free homotopy class, remark that, otherwise, it would have to be freely homotopic to the \emph{inverse} of another periodic $\gamma'$ (see, e.g., \cite{Fen:QGAF}), but that would imply that $\gamma'$ has to cross $T$ in the opposite direction as $\gamma$, contradicting the transversality of $T$.}
\end{remark}

\begin{proof}
 We start by proving the converse part of the theorem: Since $\alpha\in \pi_1(T)$ represents the free homotopy class of the closed leaves of the weak stable and weak unstable foliations restricted to $T$, by \cite[Theorem 1.3]{BGHP}, the isotopy class of any Dehn twist in the direction of $\alpha$ admits a partially hyperbolic diffeomorphism. This diffeomorphism is a collapsed Anosov flow by Theorem \ref{teo.main1}.
 
 Now, suppose that $\beta$ is a self orbit equivalence of $\phi_{FW}$. Since the Dehn twists on $T$ generate a finite index subgroup of  mapping class group of $M_{FW}$ (see, e.g., \cite[Corollary 27.6]{Johannson}), up to taking a finite power, say $k$, of $\beta$, we can assume that $\beta^k$ preserves both pieces and is isotopic to identity in each pieces.

 So $\beta^k$ must send each periodic orbit to one freely homotopic to it. By construction of the Franks--Williams example (see item \ref{item.FW_periodic_orbits} above), $\beta^{2k}$ will then fix every periodic orbit of $\phi_{FW}$. In particular, the isotopy class of $\beta^{2k}$ must preserve the conjugacy class of $\alpha$, the element of $\pi_1(T)$ that is freely homotopic to the exceptional periodic orbits of $\phi_{FW}$. Therefore, the isotopy class of $\beta^{2k}$ must be generated by the Dehn twist on $T$ in the direction of $\alpha$.
 
 So all we have left to do is show that if two self orbit equivalences are in the same isotopy class, then they are equivalent. Equivalently, it suffices to show that if $\beta$ is homotopic to the identity, then it fixes every orbit of $\phi^t$.
 
 Let $\wt\beta$ be a lift of $\beta$ to the universal cover obtained by lifting the homotopy to identity.  
Recall that by item \ref{item.FW_periodic_orbits} above, $\wt\beta$ must fix all the lifts of periodic orbits, except possibly the lift of the four exceptional periodic orbits. Moreover, $\wt\beta^2$ must preserve each half-leaves of lifted periodic orbits. Hence, any orbit obtained as an intersection of a weak stable and weak unstable leaf of a periodic orbit is fixed by $\wt\beta^2$. That set is dense in $\wt M$ \cite{FranksAnosovDiffeo}. Therefore, by continuity, $\wt\beta^2$ acts as the identity on the orbit space of $\phi_{FW}$, which ends the proof. 
\end{proof}

\appendix
\section{Branching foliations and prefoliations revisited}\label{app.branching}

In this section we obtain more information about
branching foliations. We will assume some familiarity with the constructions in \cite{BI} and repeatedly refer to statements or proofs in that paper.

\subsection{Uniqueness of approximating leaves}

The constructions of Burago and Ivanov have a
lot of inherent redundancy. What we mean
is that there are a lot of surfaces $S\colon dom(S) \rightarrow M$ with
the same image. Since these are not embeddings one has to be
more careful with the meaning of ``same image". We follow 
Burago--Ivanov and say that two surfaces $S_1, S_2$ are 
equivalent if there is a homeomorphism $g\colon dom(S_1) \rightarrow dom(S_2)$
such that $S_1 = S_2 \circ g$.
More generally if this works for subsets of the domains
we say this is a change of parameter of the subsurfaces.
This is another reason to consider $dom(F)$ to be a plane.

What we call by ``leaves" of the branching foliation, are the 
equivalence classes of these identifications.
With this understanding one can prove:

\begin{proposition}\label{prop-uniqueleaf}
Let $\cA$ be a branching foliation. Let $\cB_{\eps}$ be the 
approximating foliations constructed by Burago--Ivanov.
There is a one to one correspondence between the leaves of $\cA$ 
and the leaves of $\cB_{\eps}$ for any $\epsilon > 0$.
\end{proposition}

\begin{proof}
We will use the notations and terminology of the proof of 
\cite[Theorem 7.2]{BI}. They construct a ``push off" function
$F$ which pushes different branching leaves through a point
apart. Then given any $\alpha > 0$ they construct a foliation
$\cA_{\alpha F}$ such that as $\alpha \rightarrow 0$ the
tangent planes to the leaves of $\cA_{\alpha F}$ converge to
the bundle $E$. So $\cB_{\eps}$ is $\cA_{\eta(\eps) F}$ for
some function $\eta$ which converges to $0$ as $\eps$ converges
to $0$.

We review the important points to construct $F$. They consider 
a smooth vector field $W$ which is almost  perpendicular to the bundle
$E$. Let $\phi$ be the flow generated by $W$.
They consider
a finite cover $\{ U_i \}$ of $M$ by foliated boxes of $W$ and
with coordinates $(x_i, y_i, z_i)$ such that $E$ is almost horizontal
(i.e., close to the $(x,y)$ directions)
and $W$ is almost vertical (close to the $z$ direction) in each $U_i$. 

In each $U_i$ they consider $\cA_i$ the set of pairs 
$(S,x)$ where $S$ is an element of $\cA$, $x \in dom(S)$ and
$S(x) \in U_i$.
They define a non strict total order $\geq_i$ on $\cA_i$ as
follows: Choose $A_1, A_2 \in \cA_i$, $A_1 = (S_1,x_1),
A_2 = (S_2, x_2)$. There is an intrinsic ball 
$D = B_r(x_1) \subset dom(S_1)$ such that a piece of $A_2$ is
the graph of a $C^1$ function $f\colon D \rightarrow R$ 
as follows: The surface $S^f_1\colon D \rightarrow M$ given by

$$S^f_1(x) \ = \ \phi^{f(x)}(S_1(x)),   \ \ \ x \in D$$

\noindent
coincides, up to a change of parameter sending $x_1$ to $x_2$,
with a region in $S_2$. Let $r$ be the maximum radius of such
a ball (possibly $r = \infty$). Since the surfaces have
no topological crossing, the function $f$ does not change
sign. We set $A_2 \geq_i A_1$ if $f \geq 0$ and 
$A_1 \geq_i A_2$ if $f \leq 0$.

Burago and Ivanov remark that it is possible that both
inequalities $A_1 \geq_i A_2$ and $A_2 \geq_i A_1$ hold.
This means that $S_1$ and $S_2$ coincide up to a parameter
change, which sends $x_1$ to $x_2$, in which case they
write $A_1 \cong A_2$.

We remark that we identified surfaces of $\cA$ if they
have the same image up to parameter change.
Under this identification 
$\geq_i$ is a total order in $\cA_i$, which is 
denoted by $>_i$. So the
set of equivalence classes of $\cA_i$ is the same as
$\cA_i$.

Then \cite[Lemma 7.2]{BI} shows that $(\cA_i, >_i)$ is
order isomorphic to an open interval and they pick
a homeomorphism $\theta_i\colon \cA_i \rightarrow (0,1)$.
The important point to understand here is that $\theta_i$
is different for different branching leaves $B_1, B_2$: 
even if the leaves
$B_1, B_2$ pass through a common point $y$ in $U_i$,
and even if they coincide on a path through $U_i$.
But if $B_1, B_2$ are not the same leaf globally,
then $\theta_i(B_1) \neq \theta_i(B_2)$.
That is $\theta_i$ differentiates different branching leaves,
even if locally (which can be a big set) they have the same
image.

They use the functions $\theta_i$ to define functions 
$F_i$ which are meant to ``push" leaves of $\cA$ inside
the foliated boxes $U_i$. The push off is done along
flow lines of $\phi$. The functions $F_i$ are averaged to produce
a function $F = 1/k \sum_{i=1}^k F_i$. 
Given $\alpha > 0$, Burago--Ivanov push leaves of $\cA$ 
using the function $\alpha F$ and they show the pushed
off leaves form an actual foliation (that is, with
no branching).
The map $h$ in the statement of the Burago--Ivanov theorem,
which sends leaves of $\cA_{\alpha F}$ to leaves
of $\cA$ is just the opposite of the push off map: The
map $h$ slides points back along flow lines of $\phi$.

Now we come to the property we want to prove. Suppose
that two leaves $B, C$ of $\cA_{\alpha F}$ project to 
the same leaf $G$ of $\cA$. Since it is the same leaf 
$G$, then, by the discussion above, all the functions
$\theta_i$ are specified. Hence the functions $F_i$ are
specified along $G$ and there is only one push off leaf
in $\cA_{\alpha F}$ associated to $G$. This shows that $B, C$
are the same leaf of $\cA_{\alpha F}$.

This finishes the proof of the proposition.
\end{proof}

\subsection{Properties of some branching foliations} 

Now we go back to the specific branching foliations associated
with partially hyperbolic diffeomorphisms, as constructed
by Burago and Ivanov.

The next property we want to consider is the local 
``highest" and ``lowest" leaves from a point.
The construction of Burago and Ivanov of the branching foliations
for partially hyperbolic diffeomorphisms 
starts as follows.
Consider a point $p$ in $M$.
They fix a smooth disk $D$ through $p$, transversal
to the stable foliation, so that $E^c$ is almost tangent to $D$. 
The disk is small to be contained
in foliated boxes of all the bundles $E^c, E^s, E^u$.
The $E^{cs}$ bundle intersects
the tangent bundle to $D$ in a one dimensional bundle,
call it $G$. They consider all $C^1$ curves in $D$ tangent
to $G$. Among all these tangent curves passing through $p$ there 
is a lowest curve in the forward direction. 
The forward direction is the one given by the orientation
on $E^c$ which is almost tangent to $D$.
See \cite[\S 5]{BI}. The local saturation of this is a $C^1$ surface
(see  \cite[Proposition 3.1]{BI}). Locally it is the ``lowest" 
surface tangent to $E^{cs}$ through $p$ in the positive 
direction. This lowest surface is in fact independent of $D$.

We prove the following:

\begin{prop} \label{topmost}
One can do the construction of the branching foliations
of \cite{BI} such that through every point $p$ there is
a branching leaf which is the lowest locally in the positive
$E^c$ direction. More specifically there is a fixed size $\delta > 0$,
so  for every $p$ in $M$ the locally lowest forward surface for $p$
containing a half disk of radius at least $\delta$ centered at 
$p$ is in a leaf of the branching foliation. 
In addition for the same foliation for every $p$ 
there is also a branching leaf
which is the highest locally in the negative $E^c$ direction.
\end{prop}

\begin{proof}
In fact we prove that the branching foliations that 
Burago and Ivanov construct satisfy the conclusions of the
Lemma..
The main result we need is  \cite[Proposition 4.13]{BI}.
This result concerns ``partial" branching foliations, 
which satisfy only the non topological crossing condition
of branching foliations. Proposition 4.13 extends this
partial foliation in a particular way. 
This result is proved in \cite[\S 6]{BI} using
results in dimension 2 developed in  \cite[\S 5]{BI}.

In particular since there may be many leaves through a
given point, one has to keep track of which leaves are
``above" other leaves. They introduce a total order
in the set of leaves through a point $p$ and this has to
be preserved when one moves along paths common to both leaves
being considered.
To introduce a new leaf, they have to specify where it should
be located with respect to already existing order in the 
set of leaves through $p$. A location is given by what they
call a ``section" of the leaves through $p$, which corresponds
to a cut in the ordering of all the already existing
leaves through $p$. The constructed surfaces are called 
``upper enveloping surfaces", see \cite[Definition 6.1]{BI}.
They show that for any section at $p$ one can construct a
new partial branched surface through $p$ that fits exactly
in that section and that does not cross topologically any
of the already existing surfaces.

The beginning step of the induction process is with the empty
set. Through every point the section is empty. In this case
(empty section)
the upper enveloping surface (in the positive $E^c$ direction)
through a point $p$  is locally the lowest surface through
$p$. This is because the surface has to be what is called an
upper enveloping surface. These surfaces are locally obtained
as stable saturations of curves tangent to the $E^c$ bundle,
which are called upper envelope curves  see \cite[page 558]{BI}. The upper envelope is the supremum of descending
curves, see \cite[page 558]{BI} and the definition of
descending curves, cf.~\cite[Definitions 5.2 and 5.4]{BI}. 
In particular in the initial step there are no surfaces
so the sections are the empty sections. In this case 
\cite[Definition 5.2, item 2)]{BI} says that the initial
step is the lowest forward integral curve from the point.

The local stable saturation of the lowest forward integral
curve is the lowest local surface in the forward $E^c$
direction, tangent to $E^{cs}$ and through the point $p$.
In addition this surface is a ``patch": the edges of the
surface are separated by at least a fixed size $\delta > 0$,
see  \cite[Definition 4.7]{BI}.

This proves the first assertion of the proposition.

\vskip .1in
To prove the second assertion, one has to
go in the negative direction of $E^c$. In the construction
of the branching foliations in \cite{BI} they go alternatively
forward and backward, constructing patches of surfaces starting
at the points.

So the initial step puts in the lowest surface tangent to $E^{cs}$
through any $p$ in $M$ and going in the forward direction.
In the second step the orientations are reversed, so going
forward now corresponds to going backwards in the 
original $E^c$ direction, and lowest
is highest in the original partially constructed branching foliation.

This step is done after we already have some partial 
surfaces and sections through points. So given a point
$p$ consider the empty section of all surfaces through $p$.
Then \cite[Lemma 6.11]{BI} shows that there is a forward
envelope surface with $p$ in the boundary and the section
is the empty section at $p$. 
Since the section at $p$ is empty, then the initial step
is labeled by the empty set again (see \cite[Definition 5.2, item (3)]{BI} 
for descending curves). This means that locally
this is the lowest forward surface through $p$. 
But recall that we switched orientations, so
forward means backwards from $p$ in the original orientation,
and lowest means highest
in the original orientation.

This proves the second property of the proposition thus finishes its proof.
\end{proof}

\begin{remark}\label{r.specialbranch} In the same way we could have switched 
the orientation of $E^u$ in the beginning $-$ but not of
$E^c$. Doing the construction in \cite{BI} produces a 
branching foliation containing the highest local surface
tangent to $E^{cs}$ in the forward center direction
and the lowest local surface tangent to $E^{cs}$ in the
backwards direction. In particular, we see that every curve tangent to $E^{c}$ must be locally contained between both branching foliations. 
The reason for this is that if $c$ is any such local center curve 
in the forward direction through the point, then its local
stable saturation is a $C^1$ surface through the point and
tangent to $E^{cs}$. As proved in \cite{BI} this surface
is locally ``above'' the lowest surface through the point.
\end{remark}

\begin{remark}
In general we cannot have both lowest and highest local
surfaces (of fixed size) and in both forward and backwards
directions for all $p$ in $M$ as part of leaves of the
foliation. Here is an example one dimension lower in the
plane. Consider the differential equation $\frac{dy}{dx} = 3 y^{2/3}$.
It generates a vector field in the direction $(1,3 y^{2/3})$.
This vector field is not uniquely integrable along the 
$x$-axis. General solutions are made up of pieces of curves $y = (x+c)^3$
or segments in the $x$-axis.
Outside the $x$-axis this is uniquely integrable producing 
segments of curves $y = (x+c)^3$.

Consider first the curves that are highest forward. For any 
point $p$ in the plane the highest forward curve through that
point is contained in the curve $y = (x+c)^3$ through $p$.
Since the requirement is that one has to have a fixed sized $\delta > 0$
of highest forward for every point, then if $p$ is below the $x$
axis, but sufficiently close to the $x$-axis, the $\delta$ size
highest forward curve through $p$ is a part of the cubic which
crosses the $x$-axis. 
But the highest backward curve of every point in the $x$ is
the ray of the $x$-axis ending negatively at that point. 
One has to have at least a size $\delta$ for every point in the
$x$-axis.
These two sets of curves cross topologically, so cannot be
part of the same branching foliation.
\end{remark}

\subsection{Smooth approximation and Candel metrics}\label{ss.Candelmetric} 
The following states that the coarse nature of leaves of the branching foliations with the metric induced by the manifold is good enough. We refer the reader to \cite[\S III.H]{BH} for the basic notions about Gromov hyperbolic metric spaces.

\begin{prop}\label{prop-GH}
Let $\cF$ be a branching foliation well approximated by foliations $\cF_\eps$ such that $\cF_\eps$ are by hyperbolic leaves (recall that the approximating foliation can be chosen to have smooth leaves). Then, for every Riemannian metric in $M$ the pullback of the metric to the leaves of $\cF$ makes them Gromov hyperbolic. 
\end{prop}

\begin{proof}
For this, we choose $\eps$ small enough so that we have nice local product structure neighborhoods and take a (continuous) Riemannian metric on $M$ by considering the Candel metric (cf.~\cite{Candel}) on $\cF_\eps$ on $T\cF_\eps$ and taking a fixed vector field transverse to an $\eps$-cone around $T\cF$ to complete an orthonormal basis. For this metric, it is possible to verify in these local product structure neighborhoods the $\mathrm{CAT}(\kappa)$ condition for some $\kappa<0$ which is a local condition (see \cite[\S II.2]{BH}). 

Now, since being Gromov hyperbolic is invariant under quasi-isometries and $M$ is compact, we have that changing the metric does not change the fact that leaves are Gromov hyperbolic (see \cite[\S III.H]{BH}). 
\end{proof}

Note that in \cite[\S 4]{Thurston} it is claimed that one can choose a smooth metric in $M$ which makes every leaf of $\cF_\eps$ to have curvature arbitrarily close to $-1$. For smooth foliations this is proved in \cite[Theorem B]{AY} and attributed to Ghys (see also \cite[Remark 6.2]{AY}). In our case, leaves of $\cF$ may be just $C^1$, so it is more delicate to talk about curvature but still we only look at coarse geometric properties, so our statement suffices.  

\begin{remark}
This implies that there is a well defined notion of complete geodesics in leaves, and that through each tangent vector $v \in T_xL$ in a leaf $L \in \cF$ there is a unique geodesic in the leaf through $x$ with velocity $v$. In particular, one can compactify each leaf with a circle and consider a visual metric in this circle in a natural way. See also \cite[\S III.H.3]{BH} for definitions valid for general metric spaces. 
\end{remark}

\section{Graph transform method}\label{app.graphtransform}

Here we revisit the results in \cite{HPS} to get Theorem \ref{teo.graphtransform}. Then we comment on Theorem \ref{teo.uniformHPS} which is similar. For convenience of the reader, we recall the statement: 

 \begin{teo}[Theorem \ref{teo.graphtransform}]\label{teo.graphtransformap}
 Let $f_0\colon M \to M$ be a partially hyperbolic diffeomorphism of a closed 3-manifold $M$. There exists $\cU$ an open neighborhood of $f_0$ in the $C^1$ topology and $\eps>0$ with the property that every $g \in \cU$ is partially hyperbolic and if $\cs_g$ is a branching foliation tangent to $E^{cs}_g$ and invariant under $g$, then, for every $g' \in \cU$ there is a branching foliations $\cs_{g'}$, invariant under $g'$ and $\eps$-equivalent to $\cs_g$. 
 \end{teo}

If we assume that instead of $g$ it is $f_0$ that posses an invariant branching foliation, this result follows immediately from \cite[\S 6]{HPS} (except from the part of the non-crossing of the branching foliation which does not make sense in their setting). The main difference is therefore the uniformity of the statement. The key observation is that the method of proof of \cite{HPS} provides estimates on the size of the neighborhood on which their result hold that depend \emph{only} on certain properties of the partially hyperbolic diffeomorphism and which are uniform in a neighborhood of it. Namely, it depends on the $C^1$-size of the map, the angle between the bundles, and the strength of the contraction/expansion on them. We will overview some of the main arguments to convey the fact that these are the only aspects of the diffeomorphism needed to show the stability result. (We note that in this specific setting, as remarked by a referee, there are some possible shortcuts. In particular, since $g$ preserves a branching foliation we could use the approximating foliation to get a `simpler' coordinate system which can be useful to understand the argument in a more direct way. However, we chose to follow the arguments as they are presented in \cite{HPS} to be able to refer directly to it in some places.)

Let $f\colon M \to M$ be a partially hyperbolic diffeomorphism of a closed 3-manifold $M$. By considering a different Riemannian metric, we can assume that the bundles $E^{s}$, $E^c$ and $E^u$ are almost pairwise orthogonal and that expansion, contraction and domination is seen in one iterate (see \cite[\S 2]{CP}).  We can choose a neighborhood $\cU$ of $f$ so that every $g \in \cU$ is partially hyperbolic and the invariant bundles of $g$ have the same property with respect to the same Riemannian metric. 

We can also choose $E$ a smooth one-dimensional subbundle of $TM$ which is transverse (and almost orthogonal) to $E^{cs}_g$ for every $g \in \cU$. There exists $\eps_0>0$ such that if $0< \eps < \eps_0$ we have that the exponential mapping is a smooth embedding from $E(\eps)$ to $M$, meaning that for every $x \in M$, if we consider $E(x,\eps)$ to be the $\eps$-neighborhood of $0$ in the space $E(x) \subset T_xM$ then the exponential map $\mathrm{exp}_x \colon E(x,\eps) \to M$ is an embedding with derivative close to $1$. 

These are the choices of $\cU$ and $\eps$ that one needs to make, and if one follows the proof in \cite[Pages 94-107]{HPS} one can see that Theorem \ref{teo.graphtransform} follows. For the convenience of the reader, we will indicate the main points of the proof sketching some of the key arguments. 

\begin{proof}[Proof of Theorem \ref{teo.graphtransform}]
Consider $g \in \cU$ admitting an invariant branching foliation $\cs_g$ tangent to $E^{cs}_g$. Whenever we need to fix some constants, we will argue as for why the constants we choose only depend on properties that are constant in a neighborhood of $g$ which is why the arguments will produce a uniform neighborhood. To avoid confusions, we will use Notation \ref{not.branching}. 

We can consider that the collection of immersions $(\varphi,U) \in \cs_g$ as a unique immersion $\imath\colon V \to M$ where $V$ is an uncountable union of complete simply connected surfaces, each connected component corresponding to a leaf of $\cs_g$. The immersion $\imath$ is clearly a $C^1$-\emph{leaf immersion} which is normally expanded with respect to $g$ (\cite[\S 6]{HPS}), that is: 
\begin{enumerate}
\item the connected components of $V$ with the metric induced by $\imath$ by pullback are complete,
\item there is a map $\imath_\ast g \colon V \to V$ such that $g \circ \imath =
\imath_\ast g\circ \imath$, 
\item for every $x \in V$ we have that $D_x\imath T_x V = E^{cs}_g(\imath(x))$.  
\end{enumerate}

The only point which needs some justification is ($ii$) but this follows rather easily by considering the lift of $\cs_g$ to the universal cover where leaves are properly embedded planes and therefore it is easy to induce a map from leaf to leaf even when these may not be injectively immersed in $M$. The existence of such an immersion is the hypothesis of \cite[Theorem 6.8]{HPS} which shows its stability.

Now we want to produce a natural environment where to apply the graph transform argument. Ideally, we would consider neighborhoods of each leaf on which we can consider other leaves tangent to bundles close to $E^{cs}_g$ as graphs over the original leaf. One way to do this, would be to work in the universal cover and use the fact that the leaves $\cs_g$ are properly embedded there, so we can take a normal neighborhood using the bundle $E$ to produce such coordinates. Since our manifolds are in general non-compact (besides being an uncountable union of leaves), and we want to avoid using information about the structure of our leaves in the universal cover, we choose some kind of `local covering spaces' where the same argument can be made. This is achieved by cutting the leaf into pieces and seeing our submanifolds as gluings of many patches of leaves. This is the purpose of \emph{plaquations}.  

As in \cite[(6.2)]{HPS} we can define a \emph{plaquation} of $\imath$ consisting of embeddings $\{\rho\colon D \to V\}_{\rho\in \cP}$ of the unit disk $D= \{ v \in \RR^2 \ : \ \|v\|\leq 1\}$ such that the interiors of $\rho(B)$ as $\rho\in \cP$ cover $V$ and such that the family $\{\imath \circ \rho\}_{\rho \in \cP}$ is precompact in $\mathrm{Emb}^1(D,M)$. Precompactness means that one can extract converging subsequences to embeddings of $D$ in $M$ which are tangent to $E^{cs}$ which may not belong to the family (but will then be covered by elements of the family). 

\begin{claim}\label{claim.plaques}
We can choose the plaquation $\cP$ with the following additional properties: 
\begin{enumerate}
\item For every $x \in V$ there is a unique plaque $\rho \in \cP$ centered at $x$, this means, $\rho(0)=x$. 
\item If one considers the vector bundle $E_\rho$ induced by $E$ over the image of $\rho$ (which is a trivial bundle), we have that the exponential map $\mathrm{exp}\colon E_\rho(\eps) \to M$ is an embedding for every $\eps< \eps_0$.  
\end{enumerate}
\end{claim}
\begin{proof}
For the second item, notice that since $\exp_x \colon E(x,\eps) \to M$ is an embedding tangent to $E(x)$ at $x$ there is $\delta>0$ such that if a disk is tangent to a subbundle making a definite angle with $E$ and the disk has maximal radius smaller than $\delta$ then the exponential map will be an embedding from the bundle $E$ restricted to the disk for vectors of norm less than $\eps$. Now, we can choose a covering of $V$ by disks around every point which are mapped by $\imath$ into disks of maximal radius smaller than $\delta$ but minimal radius larger than $\delta/10$. This family will map by $\imath$ to something tangent to $E^{cs}_g$ which makes a uniform angle with $E$ and it will be clearly precompact in the space of embeddings (see \cite[(6.2)]{HPS}). 

Now, to obtain the first point it is enough to choose one plaque of the family at each point and the property will be verified. 
\end{proof}

Now, for each $g' \in \cU$ we want to construct using a graph transform argument a $C^1$-leaf immersion $\imath_{g'}\colon V \to M$ producing a branching\footnote{The non topological crossing condition is not discussed in \cite{HPS} since they work in higher codimension, but will follow rather directly from the construction in our case.} foliation $\cs_{g'}$ with the same dynamics as the one of $g$ on $\cs_g$. 

To describe the strategy, let $(\varphi, U)$ be a leaf of $\cs_g$, we want to construct a new surface $(\varphi_{g'},U)$ which will be part of the branching foliation $\cs_{g'}$. The surface $(\varphi_{g'},U)$ will be defined as $\lim_n (g')^{-n}(g^n(\varphi,U))$. We need to explain what we mean by this, and this is why the plaquations play a role in the proof: to be able to define a coordinate system on which to make sense of this limit. The proof follows the same strategy as \cite{HPS} and we will emphasize the points where the size of the neighborhood $\cU$ plays a role in the proof since this is the point on which our statement is more general than its statement in \cite{HPS}. 

As in \cite{HPS} we will work directly with $\imath$ and construct $\imath_{g'}$  since it allows to treat all leaves of $\cs_g$ simultaneously. The plaquation, as well as the transverse bundle $E$ will allow to cover each leaf by local coordinates where we can see surfaces tangent to bundles close to $E^{cs}_{g}$ as graphs over the plaque. These local coordinates of each plaques (which cover tubular neighborhoods around each plaque), will be the place where the graph transform will be applied. 

Let us construct what we mean by the graph transform. Consider $\cP$ to be the plaquation of $V$ as in Claim \ref{claim.plaques}. Using the bounds on the derivative of $g$ along $E^{cs}_g$ we can define another plaquation $\hat \cP$ which consists on the restrictions of the plaques $\rho \in \cP$ to $\hat D = \{ v \in \RR^2 : \| v \| < \hat \delta \}$ where $\hat \delta$ is chosen so that the image of the plaque $\hat \rho \in \hat \cP$ centered at $x \in M$ by $g$ is contained in the interior of the plaque $\rho' \in \cP$ centered at $\imath_\ast g(x)$. We will denote by $\rho_x$ and $\hat \rho_x$ the plaques from $\cP$ and $\hat \cP$ respectively which are centered at $x$. Since we chose the plaque families so that there is a unique plaque centered at each point, we get that to each $\hat \rho_x \in \hat \cP$ we associate a (unique) element $\rho_{g(x)} \in \cP$ and it verifies that the image by $g$ of $\hat \rho_x(\hat D)$ is contained in $\rho_{g(x)}(D)$. By restricting $\hat \delta$ a bit more if necessary, we can assume that the image by $g'$ of $\imath \circ \hat \rho_x(\hat D)$ is contained in a small neighborhood of $\imath \circ \rho_{g(x)}(D)$ for every $g' \in \cU$.

Denote $E_x$ (resp. $\hat E_{x}$) to be the vector bundle over $D$ (resp. $\hat D$) induced by $E$ via the map  $\imath \circ \rho_x$ (resp. $\imath \circ \hat \rho_x$). As before, for $z \in D$ (resp. $z \in \hat D$) we denote by $E_x(z,\delta)$ (resp. $\hat E_x(z,\delta)$) the interval of length $2\delta$ centered at $0$ on the fiber of $E_x$ (resp. $\hat E_x$) over $z$. (Note that we could make the vector bundles to be over $\imath \circ \rho_x(D)$ for each $x$, but we chose to construct them all over the same base $D$.)

Given a section $\xi$ of the bundle $E_x(\eps)$ or $\hat E_x(\eps)$ (that is, a continuous map from $D$ to $E_x(\eps)$, or $\hat D$ to $\hat E_x(\eps)$, such that $\xi(z) \in E_x(z,\eps)$ or $\xi(z) \in \hat E_x(z,\eps)$) we can define its graph as $\mathrm{graph}(\xi) \subset M$ to be the image under the exponential map of the image of $\xi$. By the choice of $\eps$ this is a topologically embedded disk. We fix a cone-field $\cC$ around the bundle $E^{cs}_g$ and transverse to $E^u_g$ such that for every $g' \in \cU$ we have that $(Dg')^{-1}$ maps $\cC$ strictly in its interior (in particular, $E^u_{g'}$ is transverse to $\cC$ everywhere), and we say that a section $\xi$ is \emph{Lipschitz} if $\mathrm{graph}(\xi)$ is everywhere tangent to $\cC$. 

By our choices of $D$ and $\hat D$ one can check:

\begin{claim}\label{claim.graphtransform1}
Let $\xi$ be a Lipschitz section of the bundle $E_{x}(\eps)$ then, for every $g' \in \cU$ there is a well defined
Lipschitz section $(g')_\ast \xi$ of the bundle $\hat E_{y}$ where $y = (\imath_\ast g)^{-1}(x)$ such that the image by $g'$ of $\mathrm{graph}((g')_\ast \xi)$ is contained in $\mathrm{graph}(\xi)$. 
\end{claim}

\begin{proof}
A section of the bundle $E_x(\eps)$ consists of a continuous map $\xi: D \to E_x(\eps)$ such that for every $z \in D$ we have that $\xi(z) \in E_x(z,\eps)$. 

Fix $\hat \rho_{y} \in \hat \cP$ where $y  =(\imath_\ast g)^{-1}(x)$. What we need to show is that there is a well defined section $\hat \xi\colon \hat D \to E_y(\eps)$ such that $g'(\mathrm{graph}(\hat \xi)) \subset \mathrm{graph}(\xi)$. But this follows from the fact that since $\mathrm{graph}(\xi)$ is tangent to $\cC$ we have that it is transverse to the strong unstable foliation of both $g$ and $g'$. So, when we apply $(g')^{-1}$ we have that the preimage is still transverse to $E^u_{g'}$ where it makes sense. In fact, the choice of $\hat D$ and $D$ ensure that the preimage by $g'$ of the exponential of $E_x(\eps)$ contains the image of the exponential of $\hat E_y( \lambda \eps)$ for some $\lambda<1$ (which by the choices we made is independent on the point or the diffeomorphism). So, it follows that we can write the preimage by $g'$ of $\mathrm{graph}(\xi)$ as the graph of some section of a subset of $D$ which contains $\hat D$ and which is Lipschitz (because the cone-field is contacted by $(Dg')^{-1}$). This is what we wanted to show. 
\end{proof}

We note that we are not yet claiming that this graph transform is well defined for \emph{global} manifolds, just that this work at each plaque (and it is to avoid ambiguity in this sense that we chose the plaquation to have a unique plaque centered at each point). The rest of the proof consists in two relevant steps: 

\begin{enumerate}
\item Show that if you start with a global manifold partitioned in plaques (we will call this partition a \emph{coherent family of sections}) then the image by the graph transform is still a coherent family of sections. This is the way to deal with the boundary behavior after cutting the surface in pieces given by the plaques.  
\item Show that the action of the graph transform is a contraction in an appropriate space. This is rather standard (see for instance \cite[\S 4.2]{CP} for a modern treatment). 
\end{enumerate}

We say that a family of Lipschitz sections $\{\xi_x\}_{x \in V}$ such that each $\xi_x$ is a section of $E_x(\eps)$ is a \emph{coherent family of sections of } $\cP$ if whenever the images of $\rho_x$ and $\rho_y$ intersect it follows that $\mathrm{graph}(\xi_x)$ and $\mathrm{graph}(\xi_y)$ coincide in the image under the exponential map of the restriction of the section $\xi_x$ to $\rho_x^{-1}(\rho_y(D) \cap \rho_x(D))$ (notice that this is the same as saying that they intersect in the image under the exponential map of the restriction of the section $\xi_y$ to $\rho_y^{-1}(\rho_y(D) \cap \rho_x(D))$).  Similarly, one can define a coherent family of sections of $\hat \cP$. 

Using Claim~\ref{claim.graphtransform1}, we will construct the graph transform of a coherent family of sections $\{\xi_x\}_{x \in V}$ by gluing together enough images under $(g')_\ast$ of plaques. There will be a unique fixed point of this graph transform which will provide the new branching foliation for $g'$ with the desired properties.

Given a coherent family of sections $\{\xi_x\}_{x\in V}$ of $\cP$ one can define a coherent family of sections of $\hat \cP$ by restriction. Similarly, since every plaque of $\cP$ is covered by plaques of $\hat \cP$, the coherent property allows to obtain, from a coherent family of sections $\{\hat \xi_x\}_{x\in V}$ of $\hat \cP$ a coherent family of sections $\{\xi_x\}_{x\in V}$ by gluing the sections in a cover of the image of $\rho_x(D)$ by $\{\hat \rho_{y_i} (\hat D)\}_{i}$; this is independent of the choice of the covering. 

We thus get: 

\begin{claim}\label{claim.graphtransform2}
Given a coherent family of sections $\{\xi_x\}_{x\in V}$  of $\cP$ one can define a new coherent family of sections $\{\zeta_x\}_{x\in V}= (g')_{\sharp} \{\xi_x\}_{x \in V}$ by gluing the coherent family of sections $\{(g')_\ast \xi_x\}_{x \in V}$ over $\hat \cP$. 
\end{claim}
 
The map $(g')_{\sharp}$ is what is called a \emph{graph transform} and it is a standard argument (see, e.g., \cite[\S 4 and \S 5]{HPS} or \cite[\S 4]{CP}) to show that one can metrize the space of Lipschitz sections with bounded Lipschitz constant to get that $(g')_{\sharp}$ is a contraction and therefore has a unique fixed point. This fixed point can be showed to consist on sections whose graphs are tangent to $E^{cs}_{g'}$ (and therefore it is $C^{1}$). Moreover, the uniqueness of the fixed point is stronger, as every $(g')_{\sharp}$ invariant family of coherent sections must coincide with this fixed point which follows by the fact that $Dg'$ expands uniformly the direction generated by $E$.  

This produces a new $C^1$-leaf immersion $\imath_{g'}$ whose leaves are tangent to $E^{cs}_{g'}$ and which are permuted in the same way as $g$ permutes the leaves of $\imath$. The dynamics inside each leaf differs by something that is smaller than the size of the plaques\footnote{This is how the notion of plaque expansivity arises naturally.}. 

We need to check that leaves do not topologically cross which follows quite directly since by uniqueness one obtains the branching foliation by iterating by $(g')_{\sharp}^n$ the original branching foliation (which corresponds to the family of trivial sections corresponding to $\imath$). Since iterates preserve the local orientation, the limit cannot create crossings. 
\end{proof}

Similar argument allow to get Theorem \ref{teo.uniformHPS} (this is indeed totally contained in \cite[\S 6]{HPS}). We refer to \cite{Martinchich} for a more modern treatment in a more general setting. 

\begin{proof}[Comments on the proof of Theorem \ref{teo.uniformHPS}]
The setup of the proof of Theorem \ref{teo.uniformHPS} is very similar to the one in Theorem \ref{teo.graphtransform} except that instead of normal expansion one has normal hyperbolicity (and naturally one cannot talk about topological crossings in higher codimension, but this is not so relevant for the proof). 

Let us comment on this difference. We emphasize again that this is done in \cite[\S 6]{HPS}, and the only difference is the uniformity of the constants that is not precisely stated there, so we will only sketch the argument very briefly to try to convince the reader that the arguments do not require more than a control on the $C^1$-size of the partially hyperbolic map and the angles between bundles (to be able to construct the plaques and set up the graph transform operator). 

In particular, one needs to first use the stable manifold theorem to construct stable manifolds and unstable manifolds through each plaque; this is done with standard graph transform methods (see \cite[Theorem 6.1(a)]{HPS}). This gives families of two dimensional plaques that now are respectively normally expanded and contracted. One can apply the same arguments as in Theorem \ref{teo.graphtransform} to these families and obtain continuations of these plaques (which will now be coherent only in the center direction as the images of the center stable plaques need not coincide out of the center direction). Intersecting these plaques one obtains the desired result. Again, checking that after cutting the plaques and applying the graph transform gives rise to a new family of plaques that is still coherent involves choosing various scales to check that when plaques intersect, they do it in a coherent way. Moreover, the way the graph transform is made ensures that the new map does not move points much along the center curves with respect to the original map, and gives the existence of the  homeomorphism $\tau\colon V \to V$ which is $C^0$-close to the identity verifying that $(\imath_g)_\ast g(x) = (\imath_{g'})_\ast g' (\tau(x))$ for every $x \in V$.

See \cite[\S 6]{HPS}, in particular \cite[pages 94-100]{HPS} for more details on how the constants are chosen (note that some parts of the proof there refer to \cite[\S 4]{HPS} where the computations about contractions of the graph transforms in the appropriate metrics are performed, an alternative, maybe more modern approach can be found in \cite[\S 4.2]{CP}). 
\end{proof}

{\small \emph{Acknowledgments:} We thank S.~Crovisier and S.~Martinchich for discussions that lead to \S \ref{ss.spaceofcollapsed}. We also thank S.~Alvarez and M.~Martinez for discussions about foliations by hyperbolic leaves. Finally, we would like to thank the two reviewers for their important and detailed comments and suggestions on the paper.}

\end{document}